\documentclass[]{amsart}
\usepackage{amssymb}

 \newtheorem{Theorem}{Theorem}[section]
 \newtheorem{Corollary}[Theorem]{Corollary}
 \newtheorem{Lemma}[Theorem]{Lemma}
 \newtheorem{Proposition}[Theorem]{Proposition}

 \newtheorem{Definition}[Theorem]{Definition}

 \newtheorem{Remark}[Theorem]{Remark}

 \numberwithin{equation}{section}

\begin{document}

\title[Hardy space, kernel function and Saitoh's conjecture]
 {Hardy space, kernel function and Saitoh's conjecture  on products of planar domains}

\author{Qi'an Guan}
\address{Qi'an Guan: School of Mathematical Sciences,
Peking University, Beijing, 100871, China.}
\email{guanqian@math.pku.edu.cn}

\author{Zheng Yuan}

\address{Zheng Yuan: School of Mathematical Sciences,
Peking University, Beijing, 100871, China.}
\email{zyuan@pku.edu.cn}

\thanks{}



\date{\today}

\dedicatory{}

\commby{}


\begin{abstract}
In this article, we consider two classes of weighted Hardy spaces on products of planar domains and their corresponding kernel functions, and we prove product versions of Saitoh's conjecture related to the two classes of weighted Hardy spaces.
\end{abstract}

\maketitle

\section{Introduction}
Let $\Delta$ be the unit disc in $\mathbb{C}$. Hardy space $H^2(\Delta)$ consists of the holomorphic functions $f$ on $\Delta$ such that 
$\int_{0}^{2\pi}|f(re^{i\theta})|^2d\theta$
is bounded for $0<r<1$, which was introduced into analysis by Hardy \cite{hardy}. In \cite{rudin55}, Rudin generalized Hardy spaces to any planar domains and studied their properties. 
Let $D$ be a planar regular region with finite boundary components which are analytic Jordan curves (see \cite{saitoh}, \cite{yamada}). 
Let $H^2(D)$ (see \cite{rudin55,saitoh}) denote the analytic Hardy class on $D$ defined as the set of all holomorphic functions $f(z)$ on $D$ such that the subharmonic functions $|f(z)|^2$ have harmonic majorants $U(z)$: 
$$|f(z)|^2\le U(z)\,\, \text{on}\,\,D.$$
Then each function $f(z)\in H^2(D)$ has Fatou's nontangential boundary value a.e. on $\partial D$ belonging to $L^2(\partial D)$ (see \cite{rudin55,duren}). 

Let $\lambda$ be a positive continuous function on $\partial D$. 
Define the weighted Szeg\"o kernel $K_{\lambda}(z,\overline w)$ (see \cite{nehari}) as follows: 
$$f(w)=\frac{1}{2\pi}\int_{\partial D}f(z)\overline{K_{\lambda}(z,\overline w)}\lambda(z)|dz|$$
holds for any $f\in H^2(D)$. Let $G_D(z,t)$ be the Green function on $D$, and let $\partial/\partial v_z$ denote the derivative along the outer normal unit vector $v_z$. For fixed $t\in D$, $\frac{\partial G_D(z,t)}{\partial v_z}$ is  positive and continuous on $\partial D$ because of the analyticity of the boundary (see \cite{saitoh}, \cite{guan-19saitoh}). When $\lambda(z)=\left(\frac{\partial G_D(z,t)}{\partial v_z}\right)^{-1}$ on $\partial D$, $\hat K_t(z,\overline w)$ denotes $K_{\lambda}(z,\overline w)$, which is the so-called conjugate Hardy $H^2$ kernel on $D$ (see \cite{saitoh}).
 When $t=w$ and $z=w$, $\hat K(z)$ denotes $\hat K_t(z,\overline w)$ for simplicity.

In \cite{guan-19saitoh}, Guan proved the following Saitoh's conjecture for conjugate Hardy $H^2$ kernels (see \cite{saitoh,yamada}), which gives a relation between conjugate Hardy $H^2$ kernel  $\hat K(z)$ (a weighted Szeg\"o kernel) and the Bergman kernel $B(z)$ for planar region $D$. 

\begin{Theorem}
	[\cite{guan-19saitoh}]\label{thm:saitoh}
	If $n>1$, then $\hat K(z)>\pi B(z)$, where $B(z)$ is the Bergman kernel on $D$.
\end{Theorem}

In \cite{GY-weightedsaitoh}, we proved a weighted version of Saitoh's conjecture (see Theorem \ref{thm:saitoh-1d}), which deduced a weighted version of Saitoh's conjecture  for higher derivatives (see Corollary \ref{c:saitoh-higher jet}).

In this article, we generalize the Hardy spaces to products of planar domains, and we prove product versions of Saitoh's conjecture.

\subsection{Hardy space $H_{\rho}^2(M,\partial M)$ over $\partial M$}

Let $D_j$ be a planar regular region with finite boundary components which are analytic Jordan curves   for any $1\le j\le n$. Let $M=\prod_{1\le j\le n}D_j$ be a bounded domain in $\mathbb{C}^n$.

 Let  $M_j=\prod_{1\le l\le n,l\not=j}D_l$, then $M=D_j\times M_j$. Note that $\partial M=\cup_{j=1}^n\partial D_j\times \overline{M_j}$.  
Let $d\mu_j$ be the Lebesgue measure on $M_j$ for any $1\le j\le n$ and $d\mu$ is a measure on $\partial M$ defined by
$$\int_{\partial M}hd\mu=\sum_{1\le j\le n}\frac{1}{2\pi}\int_{M_j}\int_{\partial D_j}h(w_j,\hat w_j)|dw_j|d\mu_j(\hat w_j)$$
 for any $h\in L^1(\partial M)$.  For any $f\in H^2(D_j)$, $\gamma_j(f)$ denotes the nontangential boundary value of $f$ a.e. on $\partial D_j$. Let $\rho$ be a Lebesgue measurable function on $\partial M$ such that $\inf_{\partial M}\rho>0$.

 \begin{Definition}\label{def:1}
 	 Let $f\in L^2(\partial M,\rho d\mu)$. We call $f\in H^2_{\rho}(M,\partial M)$ if there exists $f^*\in\mathcal{O}(M)$ such that for any $1\le j\le n$, $f^*(\cdot,\hat w_j)\in H^2(D_j)$  for any  $\hat w_j\in M_j$ and $f=\gamma_j(f^*)$ a.e. on $\partial D_j\times M_j$. 
 \end{Definition}
 We give some properties about $H^2_{\rho}(M,\partial M)$ in Section \ref{sec:partial m}. $H^2_{\rho}(M,\partial M)$ is a Hilbert space equipped with the norm $\ll\cdot,\cdot\gg_{\partial M,\rho}$, which is defined by
$$\ll f,g\gg_{\partial M,\rho}:=\int_{\partial M}f\overline{g}\rho d\mu.$$
 Denote that $P_{\partial M}(f)=f^*$ for any $f\in H^2_{\rho}(M,\partial M)$.
Lemma \ref{l:b0} shows that   $P_{\partial M}$ is a linear injective map from $H^2(M,\partial D_j\times M_j)$ to $\mathcal{O}(M)$. When $n=1$, $P_{\partial M}=\gamma_1^{-1}$.

We define a kernel function $K_{\partial M,\rho}(z,\overline w)$ as follows: 
$$K_{\partial M,\rho}(z,\overline w):=\sum_{m=1}^{+\infty}P_{\partial M}(e_m)(z)\overline{P_{\partial M}(e_m)(w)}$$
for $(z,w)\in M\times M\subset \mathbb{C}^{2n}$, where $\{e_m\}_{m\in\mathbb{Z}_{\ge1}}$ is a complete orthonormal basis of $H^2_{\rho}(M,\partial M)$. The definition of $K_{\partial M,\rho}(z,\overline w)$ is independent of the choices of $\{e_m\}_{m\in\mathbb{Z}_{\ge1}}$ (see Lemma \ref{l:b5}). Denote that $K_{\partial M,\rho}(z):=K_{\partial M,\rho}(z,\overline z)$. When $n=1$, $K_{\partial M,\rho}(z,\overline{w})=K_{\rho}(z,\overline{w})$.

Let $z_0=(z_1,\ldots,z_n)\in M,$ and let  
$$\psi(w_1,\ldots,w_n)=\max_{1\le j\le n}\{2p_jG_{D_j}(w_j,z_j)\}$$ be a plurisubharmonic function on $M$, where $G_{D_j}(\cdot,z_j)$ is the Green function on $D_j$ and $p_j>0$ is a constant for any $1\le j\le n$ such that $\sum_{1\le j\le n}\frac{1}{p_j}\le 1$.

Let $\varphi_j$ be a subharmonic function on $D_j$, which satisfies that  $\varphi_j$ is continuous at $z$ for any $z\in \partial D_j$. Note that $\frac{\partial G_{D_j}(z,z_j)}{\partial v_z}$ is a positive continuous function on $\partial D_j$ by the analyticity of the boundary (see \cite{saitoh},\cite{guan-19saitoh}), where $\partial/\partial v_z$ denotes the derivative along the outer normal unit vector $v_z$.
 Let $\rho$ be a Lebesgue measurable function on $\partial M$ such that 
$$\rho(w_1,\ldots,w_n)=\frac{1}{p_j}\left(\frac{\partial G_{D_j}(w_j,z_j)}{\partial v_{w_j}}\right)^{-1}\times\prod_{1\le l\le n}e^{-\varphi_l(w_l)}$$
on $\partial D_j\times {M_j}$, thus we have $\inf_{\partial M}\rho>0$.
Let $c$ be a positive function on $[0,+\infty)$, which satisfies that $c(t)e^{-t}$ is decreasing on $[0,+\infty)$, $\lim_{t\rightarrow0+0}c(t)=c(0)=1$ and $\int_{0}^{+\infty}c(t)e^{-t}dt<+\infty$.
Denote that 
$$\tilde \rho=c(-\psi)\prod_{1\le j\le n}e^{-\varphi_j}$$
on $M$.
Let $B_{\tilde\rho}(z_0)$ be the Bergman kernel (see Section \ref{sec:B}) on $M$ with the weight $\tilde\rho.$

 We recall some notations (see \cite{OF81}, see also \cite{guan-zhou13ap,GY-concavity}).
 Let $P_j:\Delta\rightarrow D_j$ be the universal covering from unit disc $\Delta$ to $ D_j$, and let $z_j\in D_j$.
 We call the holomorphic function $f$ on $\Delta$ a multiplicative function,
 if there is a character $\chi$, which is the representation of the fundamental group of $ D_j$, such that $g^{*}f=\chi(g)f$,
 where $|\chi|=1$ and $g$ is an element of the fundamental group of $ D_j$. Denote the set of such kinds of $f$ by $\mathcal{O}^{\chi}( D_j)$.

For any harmonic function $u$ on $D_j$, as $P_j^*(u)$ is harmonic on $\Delta$ and $\Delta$ is simple connected,
there exists a character $\chi_{j,u}$ and a multiplicative function $f_{u}\in\mathcal{O}^{\chi_{j,u}}( D)$,
such that $|f_u|=P_j^{*}\left(e^{u}\right)$.
If $u_1-u_2=\log|f|$, then $\chi_{j,u_1}=\chi_{j,u_2}$,
where $u_1$ and $u_2$ are harmonic functions on $ D_j$ and $f$ is a holomorphic function on $ D_j$.
Recall that for the Green function $G_{ D_j}(z,z_j)$ on $D_j$,
there exist a character $\chi_{j,z_0}$ and a multiplicative function $f_{z_0}\in\mathcal{O}^{\chi_{j,z_0}}( D_j)$ (see \cite{suita}) such that $|f_{z_0}(z)|=P_j^{*}\left(e^{G_{ D}(z,z_0)}\right)$.

We give a product version of Saitoh's conjecture related to $H_{\rho}^2(M,\partial M)$:

\begin{Theorem}
	\label{thm:main1-1}Assume that $B_{\tilde\rho}(z_0)>0$ and $n>1$. Then 
	\begin{equation}
		\nonumber K_{\partial M,\rho}(z_0)\geq\left(\int_{0}^{+\infty}c(t)e^{-t}dt\right)\pi B_{\tilde\rho}(z_0)
	\end{equation}
	 holds, and the equality holds if and only if the following statements hold:
	
	$(1)$ $\sum_{1\le j\le n}\frac{1}{p_j}=1$;
	
	$(2)$ $\varphi_j=2\log|g_j|+2u_j$  on $D_j$ for any $1\le j\le n$, where $u_j$ is a harmonic function on $D_j$ and $g_j$ is a holomorphic function on $\mathbb{C}$ such that $g_j(z_j)\not=0$;
	
	$(3)$ $\chi_{j,z_j}=\chi_{j,-u_j}$, where $\chi_{j,-u_j}$ and $\chi_{j,z_j}$ are the  characters associated to the functions $-u_j$ and $G_{D_j}(\cdot,z_j)$ respectively.
\end{Theorem}

\begin{Remark}When the three statements in Theorem \ref{thm:main1-1} hold, we have 
	\label{r:product 1-1}
	$$K_{\partial M,\rho}(\cdot,\overline{z_0})=\left(\int_{0}^{+\infty}c(t)e^{-t}dt\right)\pi B_{\tilde\rho}(\cdot,\overline{z_0})=c_1\prod_{1\le j\le n}g_j ((P_j)_*(f_{z_j}))'(P_j)_*(f_{u_j}),$$
	where $c_1$ is a constant, $P_j$ is the universal covering from unit disc $\Delta$ to $ D_j$, $f_{u_j}$ is a holomorphic function on $\Delta$ such that $|f_{u_j}|=P_j^*(e^{u_j})$, and $f_{z_j}$ is a holomorphic function on $\Delta$ such that $|f_{z_j}|=P_j^*(e^{G_{D_j}(\cdot,z_j)})$ for any $1\le j\le n$.
\end{Remark}

Let $h_0$ be a holomorphic function on a neighborbood $V_0$ of $z_0$, and let $I$ be an ideal of $\mathcal{O}_{z_0}$ such that $\mathcal{I}(\psi)_{z_0}\subset I\not= \mathcal{O}_{z_0}$,  where $\mathcal{I}(\psi)$ is the multiplier ideal sheaf, which is the sheaf of germs of holomorphic functions $h$ such that $|h|^2e^{-\psi}$ is locally integrable.  Assume that $(h_0,z_0)\not\in I,$ and we denote that 
$$K_{\partial M,\rho}^{I,h_0}(z_0):=\frac{1}{\inf\left\{\|f\|_{\partial M,\rho}^2:f\in H^2_{\rho}(M,\partial M)\,\&\,(P_{\partial M}(f)-h_0,z_0)\in I\right\}}$$ 
and 
$$B_{\tilde \rho}^{I,h_0}(z_0):=\frac{1}{\inf\left\{\int_{M}|f|^2\tilde \rho:f\in\mathcal{O}(M) \,\&\,(f-h_0,z_0)\in I\right\}}.$$ 
When $I$ takes the maximal ideal of $\mathcal{O}_{z_0}$ and $h_0(z_0)=1$, we have $K_{\partial M,\rho}^{I,h_0}(z_0)=K_{\partial M,\rho}(z_0)$ (see Lemma \ref{l:sup-b}) and $B_{\tilde \rho}^{I,h_0}(z_0)=B_{\tilde\rho}(z_0)$ (see Section \ref{sec:B}). 
 There exists a unique $f_0\in H_{\rho}^{2}(M,\partial M)$ (by Lemma \ref{l:b7}) such that $(P_{\partial M}(f_0)-h_0,z_0)\in I$ and 
$$K_{\partial M,\rho}^{I,h_0}(z_0)=\frac{1}{\|f_0\|^2_{\partial M,\rho}}.$$
Let 
$E=\left\{(\alpha_1,\ldots,\alpha_n):\sum_{1\le j\le n}\frac{\alpha_j+1}{p_j}=1\,\&\,\alpha_j\in\mathbb{Z}_{\ge0}\right\}$.
 
 We give a generalization of Theorem \ref{thm:main1-1} as follows:

\begin{Theorem}
	\label{thm:main1-2}
	Assume that $B_{\tilde \rho}^{I,h_0}(z_0)>0$. Then
	\begin{equation}
		\label{eq:0831a}K_{\partial M,\rho}^{I,h_0}(z_0)\geq\left(\int_{0}^{+\infty}c(t)e^{-t}dt\right)\pi B_{\tilde \rho}^{I,h_0}(z_0)
	\end{equation}
holds, and the equality holds if and only if the following statements hold:
	
	$(1)$  $P_{\partial M}(f_0)=\sum_{\alpha\in E}d_{\alpha}w^{\alpha}+g_0$ near $z_0$ for any $t\ge0$, where  $d_{\alpha}\in\mathbb{C}$  such that $\sum_{\alpha\in E}|d_{\alpha}|\not=0$ and $g_0$ is a holomorphic function near $z_0$ such that $(g_0,z_0)\in\mathcal{I}(\psi)_{z_0}$;
	
	$(2)$ $\varphi_j=2\log|g_j|+2u_j$ on $D_j$ for any $1\le j\le n$, and $g_1\equiv1$ when $n=1$, where $u_j$ is a harmonic function on $D_j$ and $g_j$ is a holomorphic function on $\mathbb{C}$ such that $g_j(z_j)\not=0$;

    $(3)$ $\chi_{j,z_j}^{\alpha_j+1}=\chi_{j,-u_j}$ for any $j\in\{1,2,...,n\}$ and $\alpha\in E$ satisfying $d_{\alpha}\not=0$, where $\chi_{j,z_j}$ and $\chi_{j,-u_j}$ are the characters associated to functions $G_{\Omega_j}(\cdot,z_j)$ and $-u_j$ respectively.
\end{Theorem}

When $I=\mathcal{I}(\psi)_{z_0}$, statement $(1)$ in Theorem \ref{thm:main1-2} is equivalent to $(h_0-\sum_{\alpha\in E}d_{\alpha}w^{\alpha},z_0)\in\mathcal{I}(\psi)_{z_0}$, where $d_{\alpha}\in\mathbb{C}$ such that $\sum_{\alpha\in E}|d_{\alpha}|\not=0$.

Let $I$ be the maximal ideal of $\mathcal{O}_{z_0}$ and $h_0(z_0)=1$, then Theorem \ref{thm:main1-2} is Theorem \ref{thm:main1-1}  when $n>1$ and Theorem \ref{thm:main1-2} can be referred to \cite{GY-weightedsaitoh} (see also Theorem \ref{thm:saitoh-1d}) when $n=1$.

\begin{Remark}\label{product 1-2}When the three statements in Theorem \ref{thm:main1-2} hold, denote that 
	$$F=\sum_{\alpha\in E}\tilde d_{\alpha}\prod_{1\le j\le n}g_j ((P_j)_*(f_{z_j}))'(P_j)_*(f_{u_j}f_{z_j}^{\alpha_j}),$$
	where $\tilde d_{\alpha}$ is a constant for any $\alpha\in E$ such that $(F-h_0,z_0)\in \mathcal{I}(\psi)_{z_0}$, $P_j$ is the universal covering from unit disc $\Delta$ to $ D_j$, $f_{u_j}$ is a holomorphic function on $\Delta$ such that $|f_{u_j}|=P_j^*(e^{u_j})$, and $f_{z_j}$ is a holomorphic function on $\Delta$ such that $|f_{z_j}|=P_j^*(e^{G_{D_j}(\cdot,z_j)})$ for any $1\le j\le n$.
	Then we have
	$$\frac{1}{B_{\tilde \rho}^{I,h_0}(z_0)}=\|F\|_{M,\tilde\rho}^2,$$
	and
	 there is $f\in H^2_{\rho}(M,\partial M)$ such that $P_{\partial M}(f)=F$ and 
	$$\frac{1}{K_{\partial M,\rho}^{I,h_0}(z_0)}=\|f\|^2_{\partial M,\rho}.$$
\end{Remark}

\begin{Remark}
	Let $\hat\rho$ be any Lebesgue measurable function on $\overline M$, which satisfies that $\inf_{\overline{M}}\hat\rho>0$, $-\log\hat\rho$ is plurisubharmonic on $M$ and $\hat\rho(w_j,\hat w_j)\leq \liminf_{w\rightarrow w_j}\hat\rho(w,\hat w_j)$ for any $(w_j,\hat w_j)\in \partial D_j\times M_j\subset \partial M$ and any $1\le j\le n$, where $M_j=\prod_{l\not=j}D_l$. Let $\rho(w_1,\ldots,w_n)=\frac{1}{p_j}\left(\frac{\partial G_{D_j}(w_j,z_j)}{\partial v_{w_j}}\right)^{-1}\hat\rho$
on $\partial D_j\times {M_j}$, and let $\tilde \rho=c(-\psi)\hat\rho$
on $M$. Inequality \eqref{eq:0831a} in Theorem \ref{thm:main1-2} also holds for this case (see Proposition \ref{p:inequality}). 
\end{Remark}

\subsection{Hardy space $H^2_{\lambda}(M,S)$ over $S$}

Let $D_j$ be a planar regular region with finite boundary components which are analytic Jordan curves   for any $1\le j\le n$. Let $M=\prod_{1\le j\le n}D_j$ be a bounded domain in $\mathbb{C}^n$.  

Denote that $S:=\prod_{1\le j\le n}\partial D_j$. Let $\lambda$ be a Lebesgue measurable function on $S$ such that $\inf_{S}\lambda>0$.
\begin{Definition}\label{def2}
	Let $f\in L^2(S,\lambda d\sigma)$, where $d\sigma:=\frac{1}{(2\pi)^n}|dw_1|\ldots|dw_n|$. We call $f\in H^2_{\lambda}(M,S)$ if there exists $\{f_m\}_{m\in\mathbb{Z}_{\ge0}}\subset\mathcal{O}(M)\cap C(\overline M)\cap L^2(S,\lambda d\sigma)$ such that $\lim_{m\rightarrow+\infty}\|f_m-f\|_{S,\lambda}^2=0$, where $\|g\|_{S,\lambda}:=\left(\int_{S}|g|^2\lambda d\sigma\right)^{\frac{1}{2}}$ for any $g\in L^2(S,\lambda d\sigma)$.
\end{Definition}
We give some properties about $H^2_{\lambda}(M,S)$ in Section \ref{sec:S}. Denote that 
$$\ll f,g\gg_{S,\lambda}=\frac{1}{(2\pi)^n}\int_S f\overline g\lambda |dw_1|\ldots|dw_n|$$
for any $f,g\in L^2(S,\lambda d\sigma),$
then $H_{\lambda}^2(M,S)$ is a Hilbert space equipped with the inner product $\ll \cdot,\cdot\gg_{S,\lambda}$. There exists a  linear injective map $P_S:H^2_{\lambda}(M,S)\rightarrow\mathcal{O}(M)$ satisfying that $P_S(f)=f$ for any $f\in\mathcal{O}(M)\cap C(\overline M)\cap L^2(S,\lambda d\sigma)$ (see Lemma \ref{l:a2} and Lemma \ref{l:a2-injective}). 

We define a kernel function $K_{S,\lambda}(z,\overline w)$ as follows: 
$$K_{S,\lambda}(z,\overline w):=\sum_{m=1}^{+\infty}P_{S}(e_m)(z)\overline{P_{S}(e_m)(w)}$$
for $(z,w)\in M\times M\subset \mathbb{C}^{2n}$, where $\{e_m\}_{m\in\mathbb{Z}_{\ge1}}$ is a complete orthonormal basis of $H^2_{\lambda}(M,S)$. The definition of $K_{S,\lambda}(z,\overline w)$ is independent of the choices of $\{e_m\}_{m\in\mathbb{Z}_{\ge1}}$ (see Lemma \ref{l:a5}). Denote that $K_{S,\lambda}(z):=K_{S,\lambda}(z,\overline z)$.

Let $z_0=(z_1,\ldots,z_n)\in M,$ and let  
$$\psi(w_1,\ldots,w_n)=\max_{1\le j\le n}\{2p_jG_{D_j}(w_j,z_j)\}$$ be a plurisubharmonic function on $M$, where $G_{D_j}(\cdot,z_j)$ is the Green function on $D_j$ and $p_j>0$ is a constant for any $1\le j\le n$.

Let $\varphi_j$ be a subharmonic function on $D_j$, which satisfies that  $\varphi_j$ is continuous at $z_j$ for any $z_j\in \partial D_j$.  Let $\rho$ be a Lebesgue measurable function on $\partial M$ such that 
$$\rho(w_1,\ldots,w_n)=\frac{1}{p_j}\left(\frac{\partial G_{D_j}(w_j,z_j)}{\partial v_{w_j}}\right)^{-1}\times\prod_{1\le l\le n}e^{-\varphi_l(w_l)}$$
on $\partial D_j\times {M_j}$, thus we have $\inf_{\partial M}\rho>0$.
 Let 
$$\lambda (w_1,\ldots,w_n)=\prod_{1\le j\le n}\left(\frac{\partial G_{D_j}(w_j,z_j)}{\partial v_{w_j}}\right)^{-1}e^{-\varphi_j(w_j)}$$
on $S=\prod_{1\le j\le n}\partial D_j$, thus $\lambda$ is continuous on $S$. Let $c$ be a positive function on $[0,+\infty)$, which satisfies that $c(t)e^{-t}$ is decreasing on $[0,+\infty)$, $\lim_{t\rightarrow0+0}c(t)=c(0)=1$ and $\int_{0}^{+\infty}c(t)e^{-t}dt<+\infty$.
Denote that $$ \tilde\rho=c(-\psi)\prod_{1\le j\le n}e^{-\varphi_j}$$
on $M$.

When $n=1$, by definitions we know $K_{S,\lambda}(z_0)=\frac{1}{p_1}K_{\partial M,\rho}(z_0)$.
 We give a relation between $K_{S,\lambda}(z_0)$ and $K_{\partial M,\rho}(z_0)$ when $n>1$. 
\begin{Theorem}
	\label{thm:main2-1}
	Assume that $n>1$ and $K_{\partial M,\rho}(z_0)>0$. Then 
	\begin{equation}
		\label{eq:0831b}K_{S,\lambda}(z_0)\geq\left(\sum_{1\le j\le n}\frac{1}{p_j}\right)\pi^{n-1} K_{\partial M,\rho}(z_0)
	\end{equation}
	holds, and equality holds if and only if 
		the following statements hold:
	
	$(1)$ $\varphi_j$ is harmonic  on $D_j$ for any $1\le j\le n$;
	
	$(2)$ $\chi_{j,z_j}=\chi_{j,-\frac{\varphi_j}{2}}$, where $\chi_{j,-\frac{\varphi_j}{2}}$ and $\chi_{j,z_j}$ are the  characters associated to the functions $-\frac{\varphi_j}{2}$ and $G_{D_j}(\cdot,z_j)$ respectively.
\end{Theorem}

\begin{Remark}
	For any $1\le j\le n$, let $\varphi_j$ be any upper semicontinuous function on $\overline{D_j}$, which satisfies that $\sup_{\overline{D_j}}\varphi_j<+\infty$ and $\varphi_j$ is subharmonic on $D_j$. Then inequality \eqref{eq:0831b} also holds for this case (see Proposition \ref{p:inequality2}). 
\end{Remark}

Using Theorem \ref{thm:main1-2} and Theorem \ref{thm:main2-1}, we  obtain the following product version of Saitoh's conjecture.

\begin{Corollary}
	\label{thm:main2-2}Assume that $B_{\tilde\rho}(z_0)>0$ and  $\sum_{1\le j\le n}\frac{1}{p_j}\le 1$. Then 
	$$K_{S,\lambda}(z_0)\prod_{1\le j\le n}(\tilde\beta_j+1)\geq\left(\sum_{1\le j\le n}\frac{1}{p_j}\right)\left(\int_0^{+\infty}c(t)e^{-t}dt\right) \pi^n B_{\tilde\rho}(z_0)$$
	 holds, and the equality holds if and only if the following statements hold:
	
	$(1)$ $\sum_{1\le j\le n}\frac{1}{p_j}=1$;
		
	$(2)$ $\varphi_j$ is harmonic  on $D_j$ for any $1\le j\le n$.
	
	$(3)$ $\chi_{j,z_j}=\chi_{j,-\frac{\varphi_j}{2}}$, where $\chi_{j,-\frac{\varphi_j}{2}}$ and $\chi_{j,z_j}$ are the  characters associated to the functions $-\frac{\varphi_j}{2}$ and $G_{D_j}(\cdot,z_j)$ respectively.
\end{Corollary}

When $n=1$, Corollary \ref{thm:main2-2} is a weighted version of Saitoh's conjecture, which can be referred to \cite{GY-weightedsaitoh} (see also Theorem \ref{thm:saitoh-1d}).

Let $h_j$ be a holomorphic function on a neighborhood of  $z_j$, and let $$h_0=\prod_{1\le j\le n}h_j.$$
 Denote that $\beta_j:=ord_{z_j}(h_j)$ and $\beta:=(\beta_1,\cdots,\beta_n)$. Let $\tilde\beta=(\tilde\beta_1,\ldots,\tilde\beta_n)\in\mathbb{Z}_{\ge0}^n$ satisfy $\tilde\beta_j\ge\beta_j$ for any $1\le j\le n$.
Let 
$$I=\left\{(g,z_0)\in\mathcal{O}_{z_0}:g=\sum_{\alpha\in\mathbb{Z}_{\ge0}^n}b_{\alpha}(w-z_0)^{\alpha}\text{ near  }z_0\text{ s.t. $b_{\alpha}=0$ for $\alpha\in L_{\tilde\beta}$}\right\},$$
where $L_{\tilde\beta}=\{\alpha\in\mathbb{Z}_{\ge0}^n:\alpha_j\le\tilde\beta_j$ for any $1\le j\le n\}$.
It is clear that $(h_0,z_0)\not\in I$. Denote that 
\begin{displaymath}
	K_{S,\lambda}^{I,h_0}(z_0):=\frac{1}{\inf\left\{\|f\|^2_{S,\lambda}:f\in H^2_{\lambda}(M,S)\,\&\,(P_{S}(f)-h_0,z_0)\in I\right\}}.
\end{displaymath}

\begin{Theorem}
\label{thm:main2-3}
	Assume that $n>1$ and $K^{I,h_0}_{\partial M,\rho}(z_0)>0$. Then 
	\begin{equation}
		\label{eq:0831c}\left(\prod_{1\le j\le n}(\tilde\beta_j+1)\right)K_{S,\lambda}^{I,h_0}(z_0)\geq\left(\sum_{1\le j\le n}\frac{\tilde\beta_j+1}{p_j}\right)\pi^{n-1} K^{I,h_0}_{\partial M,\rho}(z_0)
	\end{equation}
holds, and equality holds if and only if 
the following statements hold:

$(1)$ $\beta=\tilde\beta$;

$(2)$ $\varphi_j$ is harmonic  on $D_j$ for any $1\le j\le n$;

$(3)$ $\chi_{j,z_j}^{\tilde\beta_j+1}=\chi_{j,-\frac{\varphi_j}{2}}$, where $\chi_{j,-\frac{\varphi_j}{2}}$ and $\chi_{j,z_j}$ are the  characters associated to the functions $-\frac{\varphi_j}{2}$ and $G_{D_j}(\cdot,z_j)$ respectively.
\end{Theorem}

\begin{Remark}
	For any $1\le j\le n$, let $\varphi_j$ be any upper semicontinuous function on $\overline{D_j}$, which satisfies that $\sup_{\overline{D_j}}\varphi_j<+\infty$ and $\varphi_j$ is subharmonic on $D_j$. Then inequality \eqref{eq:0831c} also holds for this case (see Proposition \ref{p:inequality3}). 
\end{Remark}

When $n=1$, by definitions we know that $K_{S,\lambda}^{I,h_0}(z_0)=\frac{1}{p_1}K_{\partial M,\rho}^{I,h_0}(z_0)$. Using
Theorem \ref{thm:main1-2} and Theorem \ref{thm:main2-3}, we obtain the following result.

\begin{Corollary}
	\label{thm:main2-4}
	Assume that  $B_{\tilde \rho}^{I,h_0}(z_0)>0$ and  $\sum_{1\le j\le n}\frac{\tilde\beta_j+1}{p_j}\le1$. Then 
	$$\left(\prod_{1\le j\le n}(\tilde\beta_j+1)\right)K_{S,\lambda}^{I,h_0}(z_0)\geq\left(\sum_{1\le j\le n}\frac{\tilde\beta_j+1}{p_j}\right) \left(\int_{0}^{+\infty}c(t)e^{-t}dt\right)\pi^{n} B_{\tilde \rho}^{I,h_0}(z_0)$$
	holds, and equality holds if and only if 
	the following statements hold:
	
	$(1)$ $\beta=\tilde\beta$ and $\sum_{1\le j\le n}\frac{\beta_j+1}{p_j}=1$;
	
	$(2)$ $\varphi_j$ is harmonic  on $D_j$ for any $1\le j\le n$;
	
	$(3)$ $\chi_{j,z_j}^{\tilde\beta_j+1}=\chi_{j,-\frac{\varphi_j}{2}}$, where $\chi_{j,-\frac{\varphi_j}{2}}$ and $\chi_{j,z_j}$ are the  characters associated to the functions $-\frac{\varphi_j}{2}$ and $G_{D_j}(\cdot,z_j)$ respectively.
\end{Corollary}
\begin{Remark}
	 Theorem \ref{thm:main2-3} and Corollary \ref{thm:main2-4} still hold without the condition that $h_0$ can be expressed as a product $\prod_{1\le j\le n}h_j$ (see Section \ref{sec:app}), and we also consider the case $I=\mathcal{I}(\psi)_{z_0}$ in Section \ref{sec:app}.
\end{Remark}

The rest parts of the article are organized as follows:

In Section \ref{sec:pre}, we recall some properties about the weighted Bergman spaces, the weighted Bergman kernels, the Hardy spaces $H^2(D)$ and the concavity property of minimal $L^2$ integrals.

In Section \ref{sec:partial m}, we discuss the weighted Hardy space $H^2_{\rho}(M,\partial M)$ over $\partial M$ and the kernel functions $K_{\partial M,\rho}$ and $K_{\partial M,\rho}^{I,h_0}$.

In Section \ref{sec:S}, we discuss the weighted Hardy space $H^2_{\lambda}(M,S)$ over $S$ and the kernel functions $K_{S,\lambda}$ and $K_{S,\lambda}^{I,h_0}$.

In Section \ref{sec:p}, we give some useful propositions, which will be used in the proofs of the main theorems.

In Section \ref{sec:proof1}, we prove Theorem \ref{thm:main1-1} and Remark \ref{r:product 1-1}.

In Section \ref{sec:proof2}, we prove Theorem \ref{thm:main1-2} and Remark \ref{product 1-2}. 

In Section \ref{sec:proof3}, we prove Theorem \ref{thm:main2-1} and Theorem \ref{thm:main2-3}.

In Section \ref{sec:app},  we do further discussion on the relations between $K_{S,\lambda}^{I,h_0}(z_0)$, $K_{\partial M,\rho}^{I,h_0}(z_0)$ and $B_{\tilde\rho}^{I,h_0}(z_0)$.

\section{Preparations}\label{sec:pre}

In this section, we do some preparations, which will be used in the discussion in Section \ref{sec:partial m}, Section \ref{sec:S} and Section \ref{sec:p}, and used in the proofs of the main theorems.

\subsection{Bergman space and Bergman kernel}\label{sec:B}

In this section, we recall and discuss some properties about the weighted Bergman spaces and weighted Bergman kernels.

Let $U\subset \mathbb{C}^n$  be an open set. Let us recall the definition of admissible weight given in \cite{Winiarski} and \cite{Winiarski2}.
\begin{Definition}
	A nonnegative measurable function $\rho$ on $U$ is called an admissible weight, if for any $z_0\in U$ the following condition is satisfied: there exists a neighborhood $V_{z_0}$ in $U$ and a constant $C_{z_0}>0$ such that
	$$|f(z)|^2\leq C_{z_0}\int_{U}|f|^2\rho$$ 
	holds for any $z\in V_{z_0}$ and any holomorphic function $f$ on  $U$. 
\end{Definition}

Let $\rho$  be an  admissible weight on $U$. The weighted Bergman space $A^2(U,\rho)$ is defined as follows:
$$A^2(U,\rho):=\left\{f\in\mathcal{O}(U):\int_U|f|^2\rho<+\infty\right\}.$$
Denote that $$\ll f,g\gg_{U,\rho}:=\int_Uf\overline g\rho$$
and
 $||f||_{U,\rho}:=(\int_U|f|^2\rho)^{\frac{1}{2}}$ for any $f,g\in A^2(U,\rho)$.

\begin{Lemma}[see \cite{Winiarski,Winiarski2}]
$A^2(U,\rho)$ is a separable Hilbert space equipped with the inner product $\ll\cdot,\cdot\gg_{U,\rho}$.
\end{Lemma}

The following lemma gives a sufficient condition for a weight to be an admissible weight.
\begin{Lemma}\label{l:admissible}
	 Let $\rho$ be a nonnegative Lebesgue measurable function on $U$, and let $S$ be an analytic subset of $U$. Assume  that for any $K\Subset U\backslash S$, there is $a>0$ such that $\int_{K}\rho^{-a}<+\infty$. Then $\rho$ is an admissible weight on $U$.
\end{Lemma}
\begin{proof}
	As $S$ is an analytic subset of $U$, by Local Parametrization Theorem (see \cite{demailly-book}) and Maximum Principle, for any compact set $K\subset\subset U$ there exists  $ K_1\subset\subset U\backslash S$ satisfying 	
	\begin{equation}
		\label{eq:210809a}\sup_{z\in K}|f(z)|^2\le C_1\sup_{z\in K_1}|f(z)|^2
			\end{equation}
	for  any holomorphic function $f$ on $U$, where $C_1$ is a constant.
Then there exists a compact set $K_2\subset\subset U\backslash S$ satisfying $K_1\subset K_2$ and
\begin{equation}
	\label{eq:210809b}\begin{split}
	\sup_{z\in K_1}|f(z)|^{2r}\le C_2\int_{K_2}|f|^{2r}\le C_2\left(\int_{K_2}|f|^2\rho\right)^r\left(\int_{K_2}\rho^{-\frac{r}{1-r}}\right)^{1-r}	
	\end{split}\end{equation}
for any holomorphic function $f$, where $r\in(0,1)$ and $C_2$ is a constant.  As we can choose $r$ small enough such that $\int_{K_2}\rho^{-\frac{r}{1-r}}<+\infty$, it follows from inequality \eqref{eq:210809a} and inequality \eqref{eq:210809b} that 
$$\sup_{z\in K}|f(z)|^2\le C_1\sup_{z\in K_1}|f(z)|^2\le C_3\int_{K_2}|f|^2\rho,$$
where $C_3=C_1C_2^{\frac{1}{r}}\left(\int_{K_2}\rho^{-\frac{r}{1-r}}\right)^{\frac{1-r}{r}}<+\infty$. Then by definition we know  $\rho$ is an admissible weight on $U$.
\end{proof}

We recall the following well-known property of holomorphic functions of several complex variables.
\begin{Lemma}[see \cite{hormander}]
	\label{l:hartogs}Let $\Omega$ be an open subset of $\mathbb{C}^n$, and let $(w_1,\ldots,w_n)$ be the coordinate on $\mathbb{C}^n$. If $f$ is a complex valued function on $\Omega$ and $f$ is holomorphic in each variable $w_j$ when the other  variables are given arbitrary fixed values, then $f$ is holomorphic in $\Omega$.
\end{Lemma}

For any $z\in U$, we define the evaluation functional $E_z$ on $A^2(U,\rho)$ by the formula 
$$E_z(f)=f(z),\text{   $f\in A^2(U,\rho)$.}$$
As $\rho$ is an admissible weight on $U$, we know $E_z$ is a bounded linear functional on $A^2(U,\rho)$. By Riesz representation theorem, for any $z\in D$, there exists a unique $e_{z,\rho}\in A^2(U,\rho)$ such that 
$$E_z(f)=\ll f,e_{z,\rho}\gg_{U,\rho}$$
for any $f\in A^2(U,\rho)$.
The weighted Bergman kernel $B_{U,\rho}(z,\overline{w})$ (see \cite{Winiarski2}) is defined by
$$B_{U,\rho}(z,\overline{w}):=e_{w,\rho}(z).$$
\begin{Lemma}
	[see \cite{Winiarski2}]$B_{U,\rho}(z,\overline{w})=\sum_{m=1}^{+\infty}e_m(z)\overline{e_{m}(w)},$
where $\{e_m\}_{m\in\mathbb{Z}_{\ge1}}$ is any complete orthonormal basis of $A^2(U,\rho)$.
\end{Lemma}
The above lemma shows that $B_{U,\rho}(z,\overline{w})=\overline{B_{U,\rho}(w,\overline{z})}$, and
 Lemma \ref{l:hartogs} shows that $B_{U,\rho}$ is holomorphic on $U\times U$.
Denote that 
$B_{U,\rho}(z):=B_{U,\rho}(z,\overline{z})$ ($B_{\rho}(z)$ for short without misunderstanding), and it is clear that 
$$B_{\rho}(z)=\sup_{f\in A^2(U,\rho)}\frac{|f(z)|^2}{\|f\|^2_{U,\rho}}=\frac{1}{\inf\left\{\|f\|^2_{U,\rho}:f\in\mathcal{O}(U)\,\&\,f(z)=1\right\}}.$$

Now, we consider a generalized Bergman kernel. Let $z_0\in U$. Denote that 
$$B_{U,\rho}^{I,h_0}(z_0):=\frac{1}{\inf\left\{\int_{U}|f|^2 \rho:f\in\mathcal{O}(U) \,\&\,(f-h_0,z_0)\in I\right\}}$$
for any holomorphic function $h_0$ on a neighborhood of $z_0$ and any ideal $I$ of $\mathcal{O}_{U,z_0}$. Denote $B_{U,\rho}^{I,h_0}(z_0)$ by $B_{\rho}^{I,h_0}(z_0)$ without misunderstanding. when $h_0(z_0)=1$ and $I$ takes the maximal ideal of $\mathcal{O}_{U,z_0}$, $B_{\rho}^{I,h_0}(z_0)$ is the Bergman kernel $B_{\rho}(z_0)$.

\begin{Lemma}[see \cite{G-R}]
\label{l:closedness}
Let $N$ be a submodule of $\mathcal O_{\mathbb C^n,o}^q$, $1\leq q<+\infty$, and let $f_j\in\mathcal O_{\mathbb C^n}(U)^q$ be a sequence of $q-$tuples holomorphic in an open neighborhood $U$ of the origin $o$. Assume that the $f_j$ converge uniformly in $U$ towards  a $q-$tuple $f\in\mathcal O_{\mathbb C^n}(U)^q$, assume furthermore that all germs $(f_{j},o)$ belong to $N$. Then $(f,o)\in N$.	
\end{Lemma}

The closedness of submodules will be used in the following discussion.

\begin{Lemma}
	\label{l:bergman<+infty}Assume that $(h_0,z_0)\not\in I$ and $B_{\rho}^{I,h_0}(z_0)>0$. Then $B_{\rho}^{I,h_0}(z_0)<+\infty$ and there is a unique holomorphic function $f\in A^2(U,\rho)$ such that $(f-h_0,z_0)\in I$ and 
	$$B_{\rho}^{I,h_0}(z_0)=\frac{1}{\|f\|_{U,\rho}^2}.$$ Furthermore,
for any   $\hat{f}\in A^2(U,\rho)$ such that $(\hat f-h_0,z_0)\in I$,
we have the following equality
\begin{equation}\label{eq:0826a}
\begin{split}
\|\hat f\|_{U,\rho}^2=\|f\|_{U,\rho}^2+\|\hat f-f\|_{U,\rho}^2.
\end{split}
\end{equation}
\end{Lemma}
\begin{proof}
	We prove $B_{\rho}^{I,h_0}(z_0)<+\infty$ by contradiction: if not, there is $\{{f}_{j}\}_{j\in\mathbb{Z}_{>0}}\subset  A^2(U,\rho)$
such that 
$$\lim_{j\to+\infty}\|{f}_{j}\|_{U,\rho}=0$$
 and
$(f_{j}-h_0,z_0)\in I$ for any $j$. As $\rho$ is an admissible weight, then  we know that $\{f_j\}$ uniformly converges to $0$ on any compact subset of $U$.
It follows from Lemma \ref{l:closedness} and $(f_j-h_0,z_0)\in I$
that $h_0\in I$, which contradicts to the assumption $h_0\not\in I$.
Thus, we have $B_{\rho}^{I,h_0}(z_0)<+\infty.$

Firstly, we prove the existence of $f$.
As $B_{\rho}^{I,h_0}(z_0)>0,$
then there is $\{f_{j}\}_{j\in\mathbb{Z}_{>0}}\subset A^2(U,\rho)$  such that
$$\lim_{j\rightarrow+\infty}\|f_j\|^2_{U,\rho}=\frac{1}{B_{\rho}^{I,h_0}(z_0)}<+\infty$$ 
and $(f_{j}-h_0,z_0)\in I$ for any $j$. 
As $\rho$ is an admissible weight on $U$, then we know that there is a subsequence of $\{f_j\}_{j\in\mathbb{Z}_{>0}}$ denoted also by $\{f_j\}_{j\in\mathbb{Z}_{>0}}$, which uniformly converges to a holomorphic function $f$ on $U$ on any compact subset of $U$. By Lemma \ref{l:closedness}, we have 
$$(f-h_0,z_0)\in I.$$
Using Fatou's lemma, we know that 
$$\|f\|_{U,\rho}^2\le\liminf_{j\rightarrow+\infty}\|f_j\|_{U,\rho}^2=\frac{1}{B_{\rho}^{I,h_0}(z_0)}.$$
By definition of $B_{\rho}^{I,h_0}(z_0)$, we have 
$$\|f\|_{U,\rho}^2=\frac{1}{B_{\rho}^{I,h_0}(z_0)}.$$
Thus,  we obtain the existence of $f$.

Secondly, we prove the uniqueness of $f$ by contradiction:
if not, there exist two different $g_{1}\in A^2(U,\rho)$ and $g_{2}\in A^2(U,\rho)$ satisfying that $\|g_1\|_{U,\rho}^2=\|g_1\|_{U,\rho}^2=\frac{1}{B_{\rho}^{I,h_0}(z_0)}$, 
$(g_{1}-h_0,z_0)\in I$ and $(g_{2}-h_0,z_0)\in I$. It is clear that 
$$(\frac{g_{1}+g_{2}}{2}-h_0,z_0)\in I.$$
Note that
\begin{equation}\nonumber
\begin{split}
\|\frac{g_1+g_2}{2}\|_{U,\rho}^2+\|\frac{g_1-g_2}{2}\|_{U,\rho}^2=
\frac{\|g_1\|_{U,\rho}^2+\|g_2\|_{U,\rho}^2}{2}=\frac{1}{K_{U,\rho}^{I,h_0}(z_0)},
\end{split}
\end{equation}
then we obtain that
$$\|\frac{g_1+g_2}{2}\|_{U,\rho}^2<\frac{1}{B_{\rho}^{I,h_0}(z_0)},$$
 which contradicts the definition of $B_{\rho}^{I,h_0}(z_0)$.

Finally, we prove equality \eqref{eq:0826a}.
It is clear that
for any complex number $\alpha$,
$f+\alpha (\hat f-f)\in A^2(U,\rho)$, $(f+\alpha (\hat f-f),z_0)\in I$,
and 
$$\|f\|^2_{U,\rho}\le \|f+\alpha (\hat f-f)\|_{U,\rho}<+\infty.$$
Thus we have 
$$\ll f,\hat f-f\gg_{U,\rho}=0,$$
which implies that 
$$\|\hat f\|_{U,\rho}^2=\|f\|_{U,\rho}^2+\|\hat f-f\|_{U,\rho}^2.$$

Thus, Lemma \ref{l:bergman<+infty} has been proved.
\end{proof}

In the following lemma, we assume that
 $U\Subset \mathbb{C}$ is an open set. Let $\rho$ be a nonnegative Lebesgue measurable function on $U$.  Assume that for any relatively compact set $U_1\Subset U$, there exists a real number $a>0$  such that $\rho^{-a}$  is integrable on $U_1$.  Lemma \ref{l:admissible} shows that $\rho$ is an admissible weight on $U$.
\begin{Lemma}
[see \cite{GMY-boundary4}]
\label{construction of basis in dim one} Let $z_0\in U$. There exists a countable complete orthonormal basis $\{{f}_i(z)\}_{i\in \mathbb{Z}_{\ge 0}}$  of $A^2(U,\rho)$ such that $k_i:=\text{ord}_{z_0}({f}_i)$ is strictly increasing with respect to $i$.
\end{Lemma}
\begin{Remark}\label{r:construction of basis in dim one}
	If there is $h\in A^2(U,\rho)$ such that $ord_{z_0}(h)=m$ for some $m\in\mathbb{Z}_{\ge0}$, then $m\in\{ord_{z_0}(f_i):i\in\mathbb{Z}_{\ge0}\},$ where $\{{f}_i(z)\}_{i\in \mathbb{Z}_{\ge 0}}$ is the complete orthonormal basis  of $A^2(U,\rho)$ in Lemma \ref{construction of basis in dim one}.
\end{Remark}

In the following part, we consider the weighted Bergman spaces and weighted Bergman kernels on product spaces.

Let $U\subset \mathbb{C}^n$ and $W\subset \mathbb{C}^m$ be two open sets. Let $\rho_1$ and  $\rho_2$ be two nonnegative  Lebesgue measurable functions on $U$ and $W$ respectively.  Assume that for any relatively compact set $U_1\Subset U$ ($W_2\Subset W$), there exists a real number $a_1>0$ ($a_2>0$) such that $\rho_1^{-a_1}$ ($\rho_2^{-a_2}$) is integrable on $U_1$ ($W_2$).
Let $M:=U\times W$ and $\rho=\rho_1\times\rho_2$. By Lemma \ref{l:admissible}, we know that $\rho_1$, $\rho_2$ and $\rho$ are admissible weights on $U$,  on $W$ and on $M$ respectively.

The following lemma implies that the product of bases of $A^2(U,\rho_1)$ and $A^2(W,\rho_2)$ makes a basis of $A^2(M,\rho)$.

\begin{Lemma}
[see \cite{GMY-boundary4}]
\label{basis of product}Let $\{f_i(z)\}_{i\in \mathbb{Z}_{\ge 0}}$ and $\{g_j(w)\}_{j\in \mathbb{Z}_{\ge 0}}$ be the complete orthonormal basis of $A^2(U,\rho_1)$ and $A^2(W,\rho_2)$ respectively. Then $\{f_i(z) g_j(w)\}_{i,j\in\mathbb{Z}_{\ge 0}}$ is a complete orthonormal basis of $A^2(M,\rho)$.
\end{Lemma}

Let $U_j$ be an open subset of $\mathbb{C}$ for any $1\le j\le n$, and let $M=\prod_{1\le j\le n}U_j$.
For any $1\le j\le n$, let $\rho_j$ be a nonnegative Lebesgue measurable functions on $U_j$, and  assume that for any relatively compact set $K\Subset U_j$, there exists a real number $a>0$ such that $\rho_j^{-a}$  is integrable on $K$. Let 
$$\rho=\prod_{1\le j\le n}\rho_j$$
 on $M$. For any $1\le j\le n$, let $h_j$ be a holomorphic function on a neighborhood of  $z_j$, and let $$h_0=\prod_{1\le j\le n}h_j.$$
 
Denote that $\beta_j:=ord_{z_j}(h_j)$ and $\beta:=(\beta_1,\cdots,\beta_n)$. Let $\tilde\beta=(\tilde\beta_1,\ldots,\tilde\beta_n)\in\mathbb{Z}_{\ge0}^n$ satisfy $\tilde\beta_j\ge\beta_j$ for any $1\le j\le n$.
Let $z_0=(z_1,\ldots,z_n)\in M$, and let 
$$I=\left\{(g,z_0)\in\mathcal{O}_{z_0}:g=\sum_{\alpha\in\mathbb{Z}_{\ge0}^n}b_{\alpha}(w-z_0)^{\alpha}\text{ near  }z_0\text{ s.t. $b_{\alpha}=0$ for $\alpha\in L_{\tilde\beta}$}\right\},$$
where $L_{\tilde\beta}=\{\alpha\in\mathbb{Z}_{\ge0}^n:\alpha_j\le\tilde\beta_j$ for any $1\le j\le n\}$.
It is clear that $(h_0,z_0)\not\in I$. Assume that $A^2(M,\rho)\not=\{0\}$.

\begin{Lemma}
	\label{l:Bergman-prod2}$B_{M,\rho}^{I,h_0}(z_0)=\prod_{1\le j\le n}B_{U_j,\rho_j}^{I_{\tilde\beta_j,z_j},h_j}(z_j)$ holds, where $I_{\tilde\beta_j,z_j}=((w_j-z_j)^{\tilde\beta_j+1})$ is an ideal of $\mathcal{O}_{U_j,z_j}$.
\end{Lemma}
\begin{proof}
As $A^2(M,\rho)\not=\{0\}$, it is clear that $A^2(U_j,\rho_j)\not=\emptyset$ for any $1\le j\le n$.
 Using Lemma \ref{construction of basis in dim one}, for any $1\le j\le n$,  there is 
	a complete orthonormal basis $\{e_{j,m}\}_{m\in\mathbb{Z}_{\ge0}}$   for $A^2(U_j,\rho_j)$ satisfying that $k_{j,m}:=ord_{z_j}(e_{j,m}^*)$ is strictly increasing with respect to $m$. It follows from Lemma \ref{basis of product} that $\{\prod_{1\le j\le n}e_{j,\sigma_j}\}_{\sigma\in \mathbb{Z}_{\ge0}^n}$ is a complete orthonormal basis for $A^2(M,\rho)$.
	
If $f_j\in A^2(U_j,\rho_j)$ satisfies $(f_j-h_j,z_j)\in I_{\tilde\beta_j,z_j}$ for $1\le j\le n$, then $\prod_{1\le j\le n}f_j\in A^2(M,\rho)$ and  $(\prod_{1\le j\le n}f_j-h_0,z_0)\in I$. Thus, we have 
\begin{equation}
	\label{eq:0811c}
	\begin{split}
		\frac{1}{B_{M,\rho}^{I,h_0}(z_0)}&=\inf\left\{\|f\|_{M,\rho}^2:f\in A^2(M,\rho)\,\&\,(f-h_0,z_0)\in I\right\}\\
		&\le\inf\left\{\|\prod_{1\le j\le n}f_j\|_{M,\rho}^2:f_j\in A^2(U_j,\rho_j)\,\&\,(f_j-h_j,z_j)\in I_{\tilde\beta_j,z_j}\right\}\\
		&=\frac{1}{\prod_{1\le j\le n}B_{U_j,\rho_j}^{I_{\tilde\beta_j,z_j},h_j}(z_j)}.
	\end{split}
\end{equation}

Next, we prove the following statement:
 If  $f\in A^2(M,\rho)$ satisfies  $(f-h_0,z_0)\in I$, then there is $f_j\in A^2(U_j,\rho_j)$ satisfies $(f_j-h_j,z_j)\in I_{\tilde\beta_j,z_j}$ for $1\le j\le n$ and $\|f\|_{M,\rho}^2\ge\prod_{1\le j\le n}\|f_j\|^2_{U_j,\rho_j}$.
 
  We prove the statement by induction on $n$. 
 When $n=1$, the statement is clear. Assume that the statement holds for   $n=k-1$ for integer $k\ge2$. Now, we prove the case $n=k$. Denote that 
 $$M_1:=\prod_{2\le j\le k}U_j\text{ and }\hat\rho_1:=\prod_{2\le j\le k}\rho_j.$$
   Let  $f\in A^2(M,\rho)$ satisfies  $(f-h_0,z_0)\in I$, then we have $$f=\sum_{m\ge0}e_{1,m}f_{1,m},$$ 
 where $f_{1,m}\in A^2(M_1,\hat\rho_1)$.  Let $m_1$ be the largest integer such that $k_{1,m_1}\le\tilde\beta_1$, then there is a constant $c_m$ such that 
 $$(f_{1,m}-c_m\prod_{2\le j\le k}h_j,\hat z_1)\in I'$$
 for $0\le m\le m_1$,
 where $\hat z_1=(z_2,\ldots,z_k)$ and 
 \begin{displaymath}
 	\begin{split}
 		I'=\Bigg\{(g,\hat z_1)\in\mathcal{O}_{M_1, \hat z_1}:g=\sum_{\alpha=(\alpha_2,\ldots,\alpha_k)\in\mathbb{Z}_{\ge0}^{k-1}}b_{\alpha}&\prod_{2\le j\le k}(w_j-z_j)^{\alpha_j}\text{ near  }\hat z_1\\
 		&\text{s.t. $b_{\alpha}=0$ for any $\alpha\in L_1$}\Bigg\},
 	\end{split}
 \end{displaymath}
where $L_1=\{\alpha=(\alpha_2,\ldots,\alpha_k)\in\mathbb{Z}_{\ge0}^{k-1}:\alpha_j\le\tilde\beta_j$ for any $2\le j\le k\}$.
Note that $\sum_{0\le m\le m_1}|c_m|>0$, then there is $\tilde f_1\in A^2(M_1,\hat\rho_1)$ such that $(\tilde f_{1}-\prod_{2\le j\le k}h_j,\hat z_1)\in I'$ and 
$$\|\tilde f_{1}\|^2_{M_1,\hat\rho_1}=\inf\left\{\|f\|^2_{M_1,\hat\rho_1}:f\in A^2(M_1,\hat\rho_1)\,\&\,(f-\prod_{2\le j\le k}h_j,\hat z_1)\in I'\right\}.$$
Thus, we have $(f-\sum_{0\le m\le m_1}c_me_{1,m}\tilde f_1,z_0)\in I$ and
\begin{equation}
	\label{eq:0811d}
	\begin{split}
			\|f\|^2_{M,\rho}&\ge\|\sum_{0\le m\le m_1}e_{1,m}f_{1,m}\|^2_{M,\rho}\\
			&=\sum_{0\le m\le m_1}\|f_{1,m}\|_{M_1,\hat\rho_1}^2\\
			&\ge\sum_{0\le m\le m_1}|c_m|^2\|\tilde f_{1}\|_{M_1,\hat\rho_1}^2.
	\end{split}
\end{equation}
 Denote that $f_1:=\sum_{0\le m\le m_1}c_me_{1,m}$, then we have $(f_1-h_1,z_1)\in I_{\tilde\beta_1,z_1}$ and $f_1\in A^2(U_1,\rho_1)$.
By assumption, there is $f_j\in A^2(U_j,\rho_j)$ satisfies $(f_j-h_j,z_j)\in I_{\tilde\beta_j,z_j}$ for $2\le j\le k$ and 
$$\|\tilde f_1\|_{M_1,\hat\rho_1}^2\ge\prod_{2\le j\le k}\|f_j\|^2_{U_j,\rho_j},$$
 Inequality \eqref{eq:0811d} implies that 
$$\|f\|^2_{M,\rho}\ge \|f_1\|_{U_1,\rho_1}^2\|\tilde f_{1}\|_{M_1,\hat\rho_1}^2\ge\prod_{1\le j\le k}\|f_j\|_{U_j,\rho_j}^2.$$
Thus, we have proved the case $n=k$. 

Now, we have 
\begin{equation}
	\label{eq:0811e}
	\begin{split}
		&\frac{1}{\prod_{1\le j\le n}B_{U_j,\rho_j}^{I_{\tilde\beta_j,z_j},h_j}(z_j)}\\
		=&\inf\left\{\|\prod_{1\le j\le n}f_j\|_{M,\rho}^2:f_j\in A^2(U_j,\rho_j)\,\&\,(f_j-h_j,z_j)\in I_{\tilde\beta_j,z_j}\right\}\\
		\le&\inf\left\{\|f\|_{M,\rho}^2:f\in A^2(M,\rho)\,\&\,(f-h_0,z_0)\in I\right\}\\
		=&\frac{1}{B_{M,\rho}^{I,h_0}(z_0)}.
	\end{split}
\end{equation}
It follows from inequality \eqref{eq:0811c} and inequality \eqref{eq:0811e} that 
$$B_{M,\rho}^{I,h_0}(z_0)=\prod_{1\le j\le n}B_{U_j,\rho_j}^{I_{\tilde\beta_j,z_j},h_j}(z_j),$$ thus Lemma \ref{l:Bergman-prod2} holds.
\end{proof}

Let $z_0=(z_1,\ldots,z_n)\in M$, and let $f$ be a holomorphic function near $z_0$. Let 
$$\psi=\max_{1\le j\le n}\{2p_j\log|w_j-z_j|\},$$ 
where $p_j>0$ is a constant.

We recall a characterization for $(f,z_0)\in\mathcal{I}(\psi)_{z_0}$.
 Let $f=\sum_{\alpha\in\mathbb{Z}_{\ge0}^n}b_{\alpha}(w-z_0)^{\alpha}$ (Taylor expansion) be a holomorphic function  on $\{w\in\mathbb{C}^n:|w_j-z_j|<r_0$ for any $j\in\{1,2,...,n\}\}$, where $r_0>0$. 	
\begin{Lemma}[see \cite{guan-20}, see also \cite{GY-concavity4}]\label{l:psi}
$(f,z_0)\in\mathcal{I}(\psi)_{z_0}$ if and only if $\sum_{1\le j\le n}\frac{\alpha_j+1}{p_j}>1$ for any $\alpha\in\mathbb{Z}_{\ge0}^n$ satisfying $b_{\alpha}\not=0$.
\end{Lemma}
Denote that $E_1:=\{\alpha\in\mathbb{Z}_{\ge0}^n:\sum_{1\le j\le n}\frac{\alpha_j+1}{p_j}\leq 1\}$. 
We call $\alpha>\beta$ for any $\alpha,\beta\in \mathbb{Z}_{\ge0}^n$ if $\alpha_j\ge \beta_j$ for any $1\le j\le n$ but $\alpha\not=\beta$.
Let $L\not=\emptyset$ be a subset of $E_1$ satisfying that if $\alpha\in L$ then $\beta\not\in E_1$ for any $\beta >\alpha$.
Let 
$$f(w)=\sum_{\alpha\in L}b_{\alpha}(w-z_0)^{\alpha}$$ be a holomorphic function on $M$, where $b_{\alpha}\not=0$ is a constant.

\begin{Lemma}
	\label{l:Bergman-prod} 
	$$\frac{1}{B_{M,\rho}^{\mathcal{I}(\psi)_{z_0},f}(z_0)}=\sum_{\alpha \in L}\frac{|b_{\alpha}|^2}{\prod_{1\le j\le n}B_{U_j,\rho_j}^{I_{j,\alpha_j},(w_j-z_j)^{\alpha_j}}(z_j)}$$
 holds, where $I_{j,\alpha_j}=((w_j-z_j)^{\alpha_j+1})$ is an ideal of $\mathcal{O}_{U_j,z_j}$ and $w_j$ is the coordinate on $U_j$.
\end{Lemma}
\begin{proof}As $(f,z_0)\not\in \mathcal{I}(\psi)_{z_0}$ and $((w_j-z_j)^{\alpha_j},z_j)\not\in I_{j,\alpha_j}$, it follows from Lemma \ref{l:bergman<+infty} that 
$B_{M,\rho}^{\mathcal{I}(\psi)_{z_0},f}(z_0)<+\infty$ and $B_{U_j,\rho_j}^{I_{j,\alpha_j},(w_j-z_j)^{\alpha_j}}(z_j)<+\infty$.

If $B_{U_j,\rho_j}^{I_{j,\alpha_j},(w_j-z_j)^{\alpha_j}}(z_j)>0$ for any $\alpha\in L$ and $1\le j\le n$, it follows from Lemma \ref{l:bergman<+infty} that there exists $g_{j,\alpha_j}\in A^2(U_j,\rho_j)$ such that $(g_{j,\alpha_j}-(w_j-z_j)^{\alpha_j},z_j)\in I_{j,\alpha_j}$. Note that $\{\beta>\alpha:\alpha\in L \}\cap E_1=\emptyset$,  then it follows from Lemma \ref{l:psi} that 
$$(\sum_{\alpha\in L}b_{\alpha}\prod_{1\le j\le n}g_{j,\alpha}-f,z_0)\in\mathcal{I}(\psi)_{z_0},$$
which implies $B_{M,\rho}^{\mathcal{I}(\psi)_{z_0},f}(z_0)\geq\frac{1}{\|\sum_{\alpha\in L}b_{\alpha}\prod_{1\le j\le n}g_{j,\alpha}\|^2_{M,\rho}}>0$. Hence, we obtain that $B_{M,\rho}^{\mathcal{I}(\psi)_{z_0},f}(z_0)=0$ implies $\sum_{\alpha \in L}\frac{|b_{\alpha}|^2}{\prod_{1\le j\le n}B_{U_j,\rho_j}^{I_{j,\alpha_j},(w_j-z_j)^{\alpha_j}}(z_j)}=+\infty$.

In the following, we consider the case $B_{M,\rho}^{\mathcal{I}(\psi)_{z_0},f}(z_0)>0$.
	Using Lemma \ref{construction of basis in dim one}, let $\{f_{j,l}\}_{l\in\mathbb{Z}_{\ge0}}$ be a complete orthonormal basis of $A^2(U_j,\rho_j)$ such that $k_{j,l}:=ord_{z_j}(f_{j,l})$ is strictly increasing with respect to $l$ for any $1\le j\le n$. There is a constant $c_{j,l}\not=0$ such that 
	$$ord_{z_j}(f_{j,l}(w_j)-c_{j,l}(w_j-z_j)^l)>k_{j,l}$$
	 for any $1\le j\le n$ and $l\in\mathbb{Z}_{\ge0}$.
	By Lemma \ref{l:bergman<+infty} there is $f_0\in A^2(M,\rho)$ such that $(f_0-f,z_0)\in I$ and   
	\begin{equation}
		\label{eq:0826b}\|f_0\|_{M,\rho}^2=\frac{1}{B_{M,\rho}^{\mathcal{I}(\psi)_{z_0},f}(z_0)}.
	\end{equation}
	Lemma \ref{basis of product} shows that $\{\prod_{1\le j\le n}f_{j,\alpha_j}\}_{\alpha\in\mathbb{Z}_{\ge0}^n}$ is a complete orthonormal basis of $A^2(M,\rho)$. Then we have 
	$$f_0=\sum_{\alpha\in \mathbb{Z}_{\ge0}^n}d_{\alpha}\prod_{1\le j\le n}f_{j,\alpha_j},$$
	where $d_{\alpha}$ is a constant for any $\alpha$. Since $(f_0-f,z_0)\in I$, $f(w)=\sum_{\alpha\in L}b_{\alpha}(w-z_0)^{\alpha}$.and $\{\beta>\alpha:\alpha\in L \}\cap E_1=\emptyset$,  then it follows from Lemma \ref{l:psi} that for any $\alpha\in L$, there exists $\beta_{\alpha}\in\mathbb{Z}_{\ge0}^n$ such that  
	$$(d_{\beta_{\alpha}}\prod_{1\le j\le n}f_{j,\beta_{\alpha,j}}-b_{\alpha}(w-z_0)^{\alpha},z_0)\in\mathcal{I}(\psi)_{z_0},$$ which implies that 
	$$ord_{z_j}f_{j,\beta_{\alpha,j}}=\alpha_j.$$
	Note that 
	$\int_{U_j}f_{j,l}\overline g\rho_j=0$
	holds for any $g\in A^2(U_j,\rho_j)$ satisfying $ord_{z_j}(g)>ord_{z_j}(f_{j,l})$.
	 Then we have 
	$$\inf\left\{\int_{U_j}|\tilde f|^2\rho_j:\tilde f\in\mathcal{O}(U_j)\,\&\,(\tilde f-(w_j-z_j)^{\alpha_j},z_j)\in I_{j,\alpha_j}\right\}=\int_{U_j}\left|\frac{f_{j,\beta_{\alpha,j}}}{c_{j,\alpha_j}}\right|^2\rho_j,$$
	i.e., 
	\begin{equation}
		\label{eq:220803b}B_{U_j,\rho_j}^{I_{j,\alpha_j},(w_j-z_j)^{\alpha_j}}(z_j)=\frac{|c_{j,\alpha_j}|^2}{\int_{U_j}|f_{j,\beta_{\alpha,j}}|^2\rho_j},
	\end{equation}
	where $c_{j,\alpha_j}=\lim_{w_j\rightarrow z_j}\frac{f_{j,\beta_{\alpha,j}}(w_j)}{(w_j-z_j)^{\alpha_j}}$. Following from $(d_{\beta_{\alpha}}\prod_{1\le j\le n}f_{j,\beta_{\alpha,j}}-b_{\alpha}(w-z_0)^{\alpha},z_0)\in\mathcal{I}(\psi)_{z_0},$ we know that 
	\begin{equation}
		\label{eq:0826c}d_{\beta_{\alpha}}\prod_{1\le j\le n}c_{j,\alpha_j}=b_{\alpha}
	\end{equation}
	holds for any $\alpha\in L$.
 As $\{\prod_{1\le j\le n}f_{j,\alpha_j}\}_{\alpha\in\mathbb{Z}_{\ge0}^n}$ is a complete orthonormal basis of $A^2(M,\rho)$, it follows from equality \eqref{eq:0826b}, equality \eqref{eq:220803b} and equality \eqref{eq:0826c} that
	$$f_0=\sum_{\alpha\in L}d_{\beta_{\alpha}}\prod_{1\le j\le n}f_{j,\beta_{\alpha,j}}$$
	and  
	\begin{equation}
		\nonumber
		\begin{split}
			\frac{1}{B_{M,\rho}^{\mathcal{I}(\psi)_{z_0},f}(z_0)}&=\|f_0\|_{M,\rho}^2\\
			&=\sum_{\alpha\in L}|d_{\beta_{\alpha}}|^2\|\prod_{1\le j\le n}f_{j,\beta_{\alpha,j}}\|_{M,\rho}\\
			&=\sum_{\alpha\in L}|d_{\beta_{\alpha}}|^2\prod_{1\le j\le n}\frac{|c_{j,\alpha_j}|^2}{B_{U_j,\rho_j}^{I_{j,\alpha_j},(w_j-z_j)^{\alpha_j}}(z_j)}\\
			&=\sum_{\alpha \in L}\frac{|b_{\alpha}|^2}{\prod_{1\le j\le n}B_{U_j,\rho_j}^{I_{j,\alpha_j},(w_j-z_j)^{\alpha_j}}(z_j)}.
		\end{split}
	\end{equation}
	
Thus, Lemma \ref{l:Bergman-prod}	 holds.
		 \end{proof}

\subsection{Hardy space $H^2(D)$}\label{sec:2.1}
Let $D$ be a planar regular region with finite boundary components which are analytic Jordan curves  (see \cite{saitoh}, \cite{yamada}). In this section, we recall some properties related to Hardy space $H^2(D)$ and its corresponding kernel functions.

Let $H^2(D)$ (see \cite{saitoh}) denote the analytic Hardy class on $D$ defined as the set of all analytic functions $f(z)$ on $D$ such that the subharmonic functions $|f(z)|^2$ have harmonic majorants $U(z)$: 
$$|f(z)|^2\le U(z)\,\, \text{on}\,\,D.$$
Then each function $f(z)\in H^2(D)$ has Fatou's nontangential boundary value a.e. on $\partial D$ belonging to $L^2(\partial D)$ (see \cite{duren}). It is well know (see \cite{rudin55}) that if a subharmonic function has a harmonic majorant in $D$, then there exists a least harmonic majorant. Denote the least harmonic majorant of $|f|^2$ by $u_f$.

Let $z_0\in D$. Let $L^2(\partial D,\rho)$ be the space of complex valued measurable fucntion $h$ on $\partial D$, normed by 
$$\|h\|_{\partial D,\rho}^2=\frac{1}{2\pi}\int_{\partial D}|h|^2\rho |dz|,$$
where $\rho=\frac{\partial G_D(z,z_0)}{\partial v_z}$ is a positive continuous function on $\partial D$ by the analyticity of $\partial D$, $G_D(z,z_0)$ is the Green function on $D$, and $\partial/\partial v_z$ denotes the derivative along the outer normal unit vector $v_z$.

The following lemma gives some properties related to the Hardy space $H^2(D)$.
\begin{Lemma}[\cite{rudin55}]
\label{l:0-1}
	$(a)$ If $f\in H^2(D)$, there is a function $f_*$ on $\partial D$ such that $f$  has nontangential boundary value $f_*$ almost everywhere on $\partial D$. The map $\gamma:f\mapsto f_*$ is an injective linear map from $H^2(D)$ into $L^2(\partial D,\rho)$ and
	$$\|f_*\|_{\partial D,\rho}^2=u_f(z_0)$$
	holds for any $f\in H^2(D)$, where $u_f$ is the least harmonic majorant of $|f|^2$. 
	
	$(b)$ $g\in \gamma(H^2(D))$ if and only if 
	$$\int_{\partial D}g(z)\phi(z)dz=0$$
	holds for any holomorphic function $\phi$ on a neighborhood of  $\overline D$.
	 
	 $(c)$ The inverse of $\gamma $ is given by 
	\begin{equation}
		\label{eq:0728a}f(w)=\frac{1}{2\pi\sqrt{-1}}\int_{\partial D}\frac{f_*(z)}{z-w}dz
	\end{equation}
	for any $z\in D$.
\end{Lemma}

Lemma \ref{l:0-1} implies the following
\begin{Lemma}
	\label{l:0-1b}If $\lim_{n\rightarrow+\infty}\|\gamma(f_n)\|_{\partial D,\rho}=0$ for $f_n\in H^2(D)$, then $f_n$ uniformly converges to $0$ on any compact subset of $D$.
\end{Lemma}

\begin{Lemma}[\cite{rudin55}] \label{l:0-2}$ H^2(D)$ is a Hilbert space equipped with the inner product
	$$\ll f,g\gg_{\partial D,\rho}=\frac{1}{2\pi}\int_{\partial D} f_*\overline {g_*}\rho|dz|,$$
	where $\rho=\frac{\partial G_D(z,z_0)}{\partial v_z}$. Moreover,	$\mathcal{O}(D)\cap C(\overline D)$ is dense in $ H^2(D)$.
\end{Lemma}

\begin{Lemma}
	\label{l:0-3}Let $f_n\in H^2(D)$ for any $n\in \mathbb{Z}_{>0}$. Assume that $f_n$ uniformly converges to $0$ on any compact subset of $D$ and   there exists $f\in L^2(\partial D,\rho)$ such that  $\lim_{n\rightarrow+\infty}\|\gamma(f_n)-f\|_{\partial D,\rho}=0$. Then we have $f=0$.  
\end{Lemma}
\begin{proof}
	It follows from Lemma \ref{l:0-1} and Lemma \ref{l:0-2} that there exists $f_0\in H^2(D)$ such that $\gamma(f_0)=f$. Using Lemma \ref{l:0-1b}, we get that $f_n-f_0$ uniformly converges to $0$ on any compact subset of $D$, i.e. $f_0=0$, which implies that $f=0$.
\end{proof}

Let $\{D_k\}_{k\in\mathbb{Z}_{>0}}$ be an increasing sequence of domains with analytic boundaries, such that $z_0\in D_1$ and $\cup_{k=1}^{+\infty} D_k=D$. Let $G_{D_k}(\cdot,z_0)$ be the Green function of $D_k$.
\begin{Lemma}[see \cite{rudin55}] \label{l:0-4}
	$\|f_*\|^2_{\partial D,\rho}=\lim_{k\rightarrow+\infty}\frac{1}{2\pi}\int_{\partial D_k}|f|^2\frac{\partial G_{D_k}(z,z_0)}{\partial v_z}|dz|$
	holds for any $f\in H^2(D)$.
\end{Lemma}

In the following, we recall a weighted version of Saitoh's conjecture on $D$ and a higher derivatives version of it, which can be referred to \cite{GY-weightedsaitoh}.

Let $\rho$ be a positive continuous function on $\partial D$. 
The kernel function  $K_{\rho}(z,\overline w)$ (see \cite{nehari}) is defined as follows: 
$$f(w)=\frac{1}{2\pi}\int_{\partial D}f(z)\overline{K_{\rho}(z,\overline w)}\rho(z)|dz|$$
holds for any $f\in H^{2}(D)$, and we have $K_{\rho}(z,\overline w)=\sum_{k=1}^{+\infty}e_k(z)\overline{e_k(w)}$ (see \cite{nehari}), where $\{e_k\}_{k\in\mathbb{Z}_{\ge1}}$ is a complete orthonormal basis of the Hilbert space $H^2(D)$ equipped with the norm $\ll f,g\gg_{\partial D,\rho}:=\frac{1}{2\pi}\int_{\partial D}f\overline{g}\rho|dw|$.  Denote that $K_{\rho}(z):=K_{\rho}(z,\overline{z})$.

 Let $G_D(p,t)$ be the Green function on $D$, and let $\partial/\partial v_p$ denote the derivative along the outer normal unit vector $v_p$. When $\rho(p)=\left(\frac{\partial G_D(p,t)}{\partial v_p}\right)^{-1}$ on $\partial D$, $\hat K_t(z,\overline w)$ denotes $K_{\rho}(z,\overline w)$, which is the so-called conjugate Hardy $H^2$ kernel on $D$ (see \cite{saitoh}).

Let $z_0\in D,$ and let $\psi=p_0G_{D}(\cdot,z_0)$, where $p_0>0$ is a constant. 
Let $\varphi$ be a Lebesgue measurable function on $\overline D$ such that $\varphi+2\psi$ is subharmonic on $D$, the Lelong number 
$$v(dd^c(\varphi+2\psi),z_0)\ge2$$ and $\varphi$ is continuous at $z$ for any $z\in\partial D$. Let $c$ be a positive Lebesgue measurable function on $[0,+\infty)$ satisfying that $c(t)e^{-t}$ is decreasing on $[0,+\infty)$, $\lim_{t\rightarrow0+0}c(t)=c(0)=1$ and $\int_0^{+\infty}c(t)e^{-t}dt<+\infty$.

Denote that  
$$\rho:=\frac{1}{p_0}\left(\frac{\partial G_D(z,t)}{\partial v_z}\right)^{-1}e^{-\varphi}c(-2\psi)$$
on $\partial D$, and that 
$$\tilde\rho:=e^{-\varphi}c(-2\psi)$$
on $D$. 
By Lemma \ref{l:admissible}, we know that $\tilde\rho$ is an admissible weight on $D$.

We recall a weighted version of Saitoh's conjecture as follows:
\begin{Theorem}[\cite{GY-weightedsaitoh}]
\label{thm:saitoh-1d}Assume that $B_{\tilde\rho}(z_0)>0$. Then
$$K_{\rho}(z_0)\ge \left(\int_0^{+\infty}c(t)e^{-t}dt\right)\pi B_{\tilde\rho}(z_0)$$ holds, and the equality holds if and only if the following statements hold:

$(1)$ $\varphi+2\psi=2G_{D}(\cdot,z_0)+2u$, where $u$ is a harmonic function on $D$;

$(2)$ $\chi_{z_0}=\chi_{-u}$, where $\chi_{-u}$ and $\chi_{z_0}$ are the  characters associated to the functions $-u$ and $G_{D}(\cdot,z_0)$ respectively.
\end{Theorem} 

\begin{Remark}[\cite{GY-weightedsaitoh}]\label{rem:function}
Let $p$ be the universal covering from unit disc $\Delta$ to $ D$. 
When the two statements  in  Theorem \ref{thm:saitoh-1d} hold, 
$$K_{\rho}(\cdot,\overline{z_0})=\left(\int_0^{+\infty}c(t)e^{-t}dt\right)\pi B_{\tilde\rho}(\cdot,\overline{z_0})=c_1(p_*(f_{z_0}))'p_*(f_u),$$
where  $c_1$ is a constant, $f_u$ is a holomorphic function on $\Delta$ such that $|f_u|=p^*(e^u)$, and $f_{z_0}$ is a holomorphic function on $\Delta$ such that $|f_{z_0}|=p^*(e^{G_D(\cdot,z_0)})$.
\end{Remark}

Let $k$ be a nonnegative integer.   
Let $\rho$ be a continuous positive function on $\partial D$. Denote that 
\begin{displaymath}
	\begin{split}
		K_{\rho}^{(k)}(z_0):=\sup\Bigg\{\left|\frac{f^{(k)}(z_0)}{k!}\right|^2&:f\in H^{2}(D),\\
		&\int_{\partial D}|f|^2\rho|dz|\le1\,\,\&\,\,f(z_0)=\ldots=f^{(k-1)}(z_0)=0 \Bigg\}.
	\end{split}
\end{displaymath}
Especially, when $k=0$, $K_{\rho}^{(k)}(z_0)$ is the weighted  Szeg\"o kernel $K_{\rho}(z_0)$.

Let $\psi=(k+1)G_{D}(\cdot,z_0)$.   
Let $\varphi$ be a Lebesgue measurable function on $\overline D$ satisfying that $\varphi$ is subharmonic on $D$ 
 and $\varphi$ is continuous at $z$ for any $z\in\partial D$. Let $c$ be a positive Lebesgue measurable function on $[0,+\infty)$ satisfying that $c(t)e^{-t}$ is decreasing on $[0,+\infty)$, $\lim_{t\rightarrow0+0}c(t)=c(0)=1$ and $\int_0^{+\infty}c(t)e^{-t}dt<+\infty$. Denote that
$$\rho:=\frac{1}{k+1}\left(\frac{\partial G_D(z,t)}{\partial v_z}\right)^{-1}e^{-\varphi}c(-2\psi)$$
on $\partial D$, and 
$$\tilde \rho:=e^{-\varphi}c(-2\psi)$$
on $D$. 

Denote that $B^{(k)}_{\tilde\rho}(z_0):=B^{I_k,(w-z_0)^k}_{\tilde\rho}(z_0)$ (the definition of $B^{I_k,(w-z_0)^k}_{\tilde\rho}(z_0)$ can be seen in Section \ref{sec:B}), where $I_k=((w-z_0)^{k+1})$ is an ideal of $\mathcal{O}_{z_0}$.
It is clear that
\begin{displaymath}
	\begin{split}
		&B^{(k)}_{\tilde\rho}(z_0)\\
		&=\sup\left\{\left|\frac{f^{(k)}(z_0)}{k!}\right|^2:f\in\mathcal{O}(D),\,\,\int_D|f|^2\tilde\rho\le1\,\,\&\,\,f(z_0)=\ldots=f^{(k-1)}(z_0)=0 \right\}.
	\end{split}
\end{displaymath}
When $\tilde\rho\equiv1$, $B^{(k)}_{\tilde\rho}(z_0)$ is the Bergman kernel for higher derivatives (see \cite{Berg70}).

We recall a weighted version of Saitoh's conjecture for higher derivatives as follows:

\begin{Corollary}[\cite{GY-weightedsaitoh}]Assume that $B^{(k)}_{\tilde\rho}(z_0)>0$. Then
	\label{c:saitoh-higher jet}
	$$K_{\rho}^{(k)}(z_0)\ge\left(\int_0^{+\infty}c(t)e^{-t}dt\right) \pi B^{(k)}_{\tilde\rho}(z_0)$$
	 holds, and the equality holds if and only if  the following statements hold:
	
	$(1)$ $\varphi$ is  harmonic  $D$;
	 	
	$(2)$ $\chi_{z_0}^{k+1}=\chi_{-\frac{\varphi}{2}}$, where $\chi_{-\frac{\varphi}{2}}$ and $\chi_{z_0}$ are the  characters associated to the functions $-\frac{\varphi}{2}$ and $G_{D}(\cdot,z_0)$ respectively.
\end{Corollary}
\begin{Remark}
	\label{rem:saitoh2}Let $p$ be the universal covering from unit disc $\Delta$ to $ D$. 
	Assume that the two statements in  Corollary \ref{c:saitoh-higher jet} hold, and denote that 
	$$F= c_2(p_*(f_{z_0}))'p_*(f_{\frac{\varphi}{2}}f_{z_0}^{k}),$$ 
	where   $f_{\frac{\varphi}{2}}$ is a holomorphic function on $\Delta$ such that $|f_{\frac{\varphi}{2}}|=p^*(e^{\frac{\varphi}{2}})$,  $f_{z_0}$ is a holomorphic function on $\Delta$ such that $|f_{z_0}|=p^*(e^{G_D(\cdot,z_0)})$, and $c_2$ is a constant such that $F^{(k)}(z_0)=k!$.
	Then we have
	$$\frac{1}{K_{\rho}^{(k)}(z_0)}=\frac{1}{2\pi}\int_{\partial D}|F|^2\rho|dz|\text{ and } \frac{1}{B^{(k)}_{\tilde\rho}(z_0)}=\int_{D}|F|^2\tilde\rho.$$
\end{Remark}
\begin{proof}
	Let $\varphi_1=\varphi-2k\log|z-z_0|$. It is clear that $\varphi_1+2\psi=\varphi+2(k+1)G_D(\cdot,z_0)-2k\log|z-z_0|$ is subharmonic on $D$ and $v(dd^c(\varphi_1+2\psi),z_0)\ge2$.  Denote that 
	$$\tilde\rho_1:=e^{-\varphi_1}c(-2\psi)=|z-z_0|^{2k}\tilde\rho$$
	 on $D$ and 
	 $$\rho_1:=\frac{1}{k+1}\left(\frac{\partial G_D(z,t)}{\partial v_z}\right)^{-1}e^{-\varphi_1}c(-2\psi)=|z-z_0|^{2k}\rho$$ 
	 on $\partial D$.
	Note that 
	\begin{displaymath}
		\begin{split}
			&B^{(k)}_{\tilde\rho}(z_0)\\
			=&\sup\left\{\left|\frac{f^{(k)}(z_0)}{k!}\right|^2:f\in\mathcal{O}(D),\,\,\int_D|f|^2\tilde\rho\le1\,\,\&\,\,f(z_0)=\ldots=f^{(k-1)}(z_0)=0 \right\}\\
			=&\sup\left\{|g(z_0)|^2:g\in\mathcal{O}(D),\,\,\int_D|g|^2|z-z_0|^{2k}\tilde\rho\le1 \right\}\\
			=&B_{\tilde\rho_1}(z_0),
		\end{split}
	\end{displaymath}
	and
	\begin{displaymath}
		\begin{split}
			&K_{\rho}^{(k)}(z_0)\\
			=&\sup\Bigg\{\left|\frac{f^{(k)}(z_0)}{k!}\right|^2:f\in H^2(D),\,\, \int_{\partial D}|f|^2\rho|dz|\le1\\
			&\&\,\, f(z_0)=\ldots=f^{(k-1)}(z_0)=0 \Bigg\}\\
			=&\sup\Bigg\{|g(z_0)|^2:g\in H^2(D)\,\,\&\,\,\int_{\partial D}|g|^2|z-z_0|^{2k}\rho|dz|\le1\Bigg\}\\
			=&K_{\rho_1}(z_0).
		\end{split}
	\end{displaymath} 
Denote that $2u:=\varphi_1+2\psi-2G_D(\cdot,z_0)=\varphi+2k(G_D(\cdot,z_0)-\log|z-z_0|)$ and $\tilde F:=c_3(p_*(f_{z_0}))'p_*(f_u)$, where $c_3$ is a constant such that $\tilde F(z_0)=1$.
It follows from Remark \ref{rem:function} that $B_{\tilde\rho_1}(z_0)=\frac{1}{\int_D|\tilde F|^2\tilde\rho_1}$ and $K_{\rho_1}(z_0)=\frac{1}{\frac{1}{2\pi}\int_D|\tilde F|^2\rho_1|dz|}$. Thus, we have 
$$B^{(k)}_{\tilde\rho}(z_0)=\frac{1}{\int_D|(z-z_0)^k\tilde F|^2\tilde\rho}$$
and 
$$K_{\rho}^{(k)}(z_0)=\frac{1}{\frac{1}{2\pi}\int_{\partial D}|(z-z_0)^k\tilde F|^2\rho|dz|}.$$
As $2u=\varphi+2k(G_D(\cdot,z_0)-\log|z-z_0|)$, we know $f_u=f_{\frac{\varphi}{2}}\left(\frac{f_{z_0}}{p^*(z-z_0)}\right)^k$. 
Note that $\left((z-z_0)^k\tilde F\right)^{(k)}(z_0)=k!$ and 
$$(z-z_0)^k\tilde F=c_3(z-z_0)^k(p_*(f_{z_0}))'p_*(f_u)=c_3(p_*(f_{z_0}))'p_*(f_{\frac{\varphi}{2}}f_{z_0}^{k}).$$ Thus, we have $F=(z-z_0)^k\tilde F$.
\end{proof}

\subsection{Concavity of minimal $L^2$ integrals}

In this section, we recall some results about the concavity property of minimal $L^2$ integrals (see \cite{GY-concavity,GY-concavity4}).

Let $M$ be an $n-$dimensional Stein manifold,  and let $K_{M}$ be the canonical (holomorphic) line bundle on $M$.
Let $\psi$ be a plurisubharmonic function on $M$,
and let  $\varphi$ be a Lebesgue measurable function on $M$,
such that $\varphi+\psi$ is a plurisubharmonic function on $M$. Take $T=-\sup_{M}\psi>-\infty$.

\begin{Definition}
\label{def:gain}
We call a positive measurable function $c$  on $(T,+\infty)$ in class $\mathcal{P}_T$ if the following two statements hold:

$(1)$ $c(t)e^{-t}$ is decreasing with respect to $t$;

$(2)$ there is a closed subset $E$ of $M$ such that $E\subset \{z\in Z:\psi(z)=-\infty\}$ and for any compact subset $K\subseteq M\backslash E$, $e^{-\varphi}c(-\psi)$ has a positive lower bound on $K$, where $Z$ is some analytic subset of $M$.	
\end{Definition}

Let $Z_{0}$ be a subset of $\{\psi=-\infty\}$ such that $Z_{0}\cap Supp(\{\mathcal{O}/\mathcal{I(\varphi+\psi)}\})\neq\emptyset$.
Let $U\supseteq Z_{0}$ be an open subset of $M$
and let $f$ be a holomorphic $(n,0)$ form on $U$.
Let $\mathcal{F}\supseteq\mathcal{I}(\varphi+\psi)|_{U}$ be an  analytic subsheaf of $\mathcal{O}$ on $U$.

Denote
\begin{equation*}
\begin{split}
\inf\Bigg\{\int_{\{\psi<-t\}}|\tilde{f}|^{2}e^{-\varphi}c(-\psi):(\tilde{f}-f)\in H^{0}(Z_0,&
(\mathcal{O}(K_{M})\otimes\mathcal{F})|_{Z_0})\\&\&{\,}\tilde{f}\in H^{0}(\{\psi<-t\},\mathcal{O}(K_{M}))\Bigg\},
\end{split}
\end{equation*}
by $G(t;c)$ (without misunderstanding, we denote $G(t;c)$ by $G(t)$),  where $t\in[T,+\infty)$,  $c\in\mathcal{P}_{T}$ satisfying $\int_T^{+\infty}c(l)e^{-l}dl<+\infty$,
$|f|^{2}:=\sqrt{-1}^{n^{2}}f\wedge\bar{f}$ for any $(n,0)$ form $f$ and $(\tilde{f}-f)\in H^{0}(Z_0,
(\mathcal{O}(K_{M})\otimes\mathcal{F})|_{Z_0})$ means $(\tilde{f}-f,z_0)\in(\mathcal{O}(K_{M})\otimes\mathcal{F})_{z_0}$ for all $z_0\in Z_0$.

We recall some results about the concavity for $G(t)$.
\begin{Theorem} [\cite{GY-concavity}]
\label{thm:general_concave}
 Assume that $G(T)<+\infty$. Then $G(h^{-1}(r))$ is concave with respect to $r\in(0,\int_{T}^{+\infty}c(l)e^{-l}dl)$, $\lim_{t\rightarrow T+0}G(t)=G(T)$ and $\lim_{t\rightarrow +\infty}G(t)=0$, where $h(t)=\int_{t}^{+\infty}c(l)e^{-l}dl$.
\end{Theorem}

The following corollary gives a necessary condition for the concavity of $G(h^{-1}(r))$ degenerating to linearity.

\begin{Corollary}[\cite{GY-concavity}]	\label{c:linear}
 Assume that $G(T)\in(0,+\infty)$. If $G( {h}^{-1}(r))$ is linear with respect to $r\in(0,\int_{T}^{+\infty}c(l)e^{-l}dl)$, where $ {h}(t)=\int_{t}^{+\infty}c(l)e^{-l}dl$,
 then there is a unique holomorphic $(n,0)$ form $F$ on $M$ satisfying $(F-f)\in H^{0}(Z_0,(\mathcal{O}(K_{M})\otimes\mathcal F)|_{Z_0})$ and $G(t;c)=\int_{\{\psi<-t\}}|F|^2e^{-\varphi}c(-\psi)$ for any $t\geq T$. Furthermore,
\begin{equation}
\nonumber	\int_{\{-t_1\leq\psi<-t_2\}}|F|^2e^{-\varphi}a(-\psi)=\frac{G(T_1;c)}{\int_{T_1}^{+\infty}c(l)e^{-l}dl}\int_{t_2}^{t_1} a(t)e^{-t}dt
\end{equation}
for any nonnegative measurable function $a$ on $(T,+\infty)$, where $+\infty\geq t_1>t_2\geq T$.
\end{Corollary}

The following lemma is a characterization of $G(t)= 0$, where $t\geq T$.

\begin{Lemma}[\cite{GY-concavity}]
The following two statements are equivalent:
\par
$(1)$ $(f) \in
H^0(Z_0,(\mathcal{O} (K_M) \otimes \mathcal{F})|_{Z_0} )$.
\par
$(2)$ $G(t) = 0$.
\label{l:G equal to 0}
\end{Lemma}

We recall the existence and uniqueness of the holomorphic $(n,0)$ form related to $G(t)$.
\begin{Lemma}[\cite{GY-concavity}]
\label{l:unique}
Assume that $G(t)<+\infty$ for some $t\in[T,+\infty)$.
Then there exists a unique holomorphic $(n,0)$ form $F_{t}$ on
$\{\psi<-t\}$ satisfying $(F_{t}-f)\in H^{0}(Z_0,(\mathcal{O}(K_{M})\otimes\mathcal{F})|_{Z_0})$ and $\int_{\{\psi<-t\}}|F_{t}|^{2}e^{-\varphi}c(-\psi)=G(t)$.
Furthermore,
for any holomorphic $(n,0)$ form $\hat{F}$ on $\{\psi<-t\}$ satisfying $(\hat{F}-f)\in H^{0}(Z_0,(\mathcal{O}(K_{M})\otimes\mathcal{F})|_{Z_0})$ and
$\int_{\{\psi<-t\}}|\hat{F}|^{2}e^{-\varphi}c(-\psi)<+\infty$,
we have the following equality
\begin{equation}
\nonumber \begin{split}
&\int_{\{\psi<-t\}}|F_{t}|^{2}e^{-\varphi}c(-\psi)+\int_{\{\psi<-t\}}|\hat{F}-F_{t}|^{2}e^{-\varphi}c(-\psi)
\\=&
\int_{\{\psi<-t\}}|\hat{F}|^{2}e^{-\varphi}c(-\psi).
\end{split}
\end{equation}
\end{Lemma}

In the following, we recall a characterization for the concavity degenerating to linearity  when $M$ is a product manifold of open Riemann surfaces.

Let $\Omega_j$  be an open Riemann surface, which admits a nontrivial Green function $G_{\Omega_j}$ for any  $1\le j\le n$. Let 
$$M=\prod_{1\le j\le n}\Omega_j$$ be an $n-$dimensional complex manifold, and let $\pi_j$ be the natural projection from $M$ to $\Omega_j$. Let $K_M$ be the canonical (holomorphic) line bundle on $M$. Let
 $Z_0=\{z_0\}=\{(z_1,\ldots,z_n)\}\subset M$. For any $j\in\{1,\ldots,n\}$, let $\varphi_j$ be a subharmonic function on $\Omega_j$, and let 
 $$\varphi=\sum_{1\le j\le n}\pi_j^*(\varphi_j).$$ 
Let 
$$\psi=\max_{1\le j\le n}\left\{2p_j\pi_j^{*}(G_{\Omega_j}(\cdot,z_j))\right\},$$
 where $p_j$ is positive real number.
Let $c$ be a positive function on $(0,+\infty)$ such that $\int_{0}^{+\infty}c(t)e^{-t}dt<+\infty$ and $c(t)e^{-t}$ is decreasing on $(0,+\infty)$. Let $\mathcal{F}_{z}=\mathcal{I}(\psi)_z$ for any $z\in Z_0$.

Let $w_j$ be a local coordinate on a neighborhood $V_{z_j}$ of $z_j\in\Omega_j$ satisfying $w_j(z_j)=0$. Denote that $V_0:=\prod_{1\le j\le n}V_{z_j}$, and $w:=(w_1,\ldots,w_n)$ is a local coordinate on $V_0$ of $z_0\in M$.
Let $f$ be a holomorphic $(n,0)$ form on $V_0$. Denote that 
$$E:=\left\{(\alpha_1,\ldots,\alpha_n):\sum_{1\le j\le n}\frac{\alpha_j+1}{p_j}=1\,\&\,\alpha_j\in\mathbb{Z}_{\ge0}\right\}.$$

We recall a characterization of the concavity of $G(h^{-1}(r))$ degenerating to linearity. 
\begin{Theorem}[\cite{GY-concavity4}]
	\label{thm:linear-2d}
	Assume that $G(0)\in(0,+\infty)$ and $\varphi_j(z_j)>-\infty$ for any $1\le j\le n$.  $G(h^{-1}(r))$ is linear with respect to $r\in(0,\int_{0}^{+\infty}c(t)e^{-t}dt]$  if and only if the  following statements hold:
	
	$(1)$ $f=\left(\sum_{\alpha\in E}d_{\alpha}w^{\alpha}+g_0\right)dw_1\wedge\ldots\wedge dw_n$ on $V_0$, where  $d_{\alpha}\in\mathbb{C}$ such that $\sum_{\alpha\in E}|d_{\alpha}|\not=0$ and $g_0$ is a holomorphic function on $V_0$ such that $(g_0,z_0)\in\mathcal{I}(\psi)_{z_0}$;
	
	$(2)$ $\varphi_j=2\log|g_j|+2u_j$, where $g_j$ is a holomorphic function on $\Omega_j$ such that $g_j(z_j)\not=0$ and $u_j$ is a harmonic function on $\Omega_j$ for any $1\le j\le n$;

    $(3)$ $\chi_{j,z_j}^{\alpha_j+1}=\chi_{j,-u_j}$ for any $j\in\{1,2,...,n\}$ and $\alpha\in E$ satisfying $d_{\alpha}\not=0$, $\chi_{j,z_j}$ and $\chi_{j,-u_j}$ are the characters associated to functions $G_{\Omega_j}(\cdot,z_j)$ and $-u_j$ respectively.
\end{Theorem}

The following Lemma will be used in the proof of Remark \ref{r:var=-infty}.

\begin{Lemma}[see \cite{GY-concavity4}] \label{l:m2}
Let $\psi=\max_{1\le j\le n}\{2p_j\log|w_j|\}$ be a plurisubharmonic function on $\mathbb{C}^n$, where $p_j>0$.
	Let $f=\sum_{\alpha\in \mathbb{Z}_{\ge0}^n}b_{\alpha}w^{\alpha}$ (Taylor expansion) be a holomorphic function on $\{\psi<-t_0\}$, where $t_0>0$. Then
	$$\int_{\{\psi<-t\}}|f|^2d\lambda_n=\sum_{\alpha\in\mathbb{Z}_{\ge0}^n}e^{-\sum_{1\le j\le n}\frac{\alpha_j+1}{p_j}t}\frac{|b_{\alpha}|^2\pi^n}{\Pi_{1\le j\le n}(\alpha_j+1)}$$
	holds for any $t\ge t_0$. \end{Lemma}

\begin{Remark}
	\label{r:var=-infty}The requirement ``$\varphi_j(z_j)>-\infty$ for any $1\le j\le n$" in Theorem \ref{thm:linear-2d} can be removed.
\end{Remark}
\begin{proof}It suffices to prove that the linearity of $G(h^{-1}(r))$ can deduce $\varphi(z_j)>-\infty$ for any $1\le j\le n$. 

Assume that $G(h^{-1}(r))$ is linear with respect to $r\in(0,\int_{0}^{+\infty}c(t)e^{-t}dt]$.
	It follows from Corollary \ref{c:linear} that there is a holomorphic $(n,0)$ form $F$ on $M$ such that $(F-f,z_0)\in\mathcal{I}(\psi)_{z_0}$ and
	\begin{equation}
		\label{eq:220807a}\int_{\{\psi<-t\}}|F|^2e^{-\varphi}=\frac{G(0)}{\int_0^{+\infty}c(t)e^{-t}dt}e^{-t}
	\end{equation}
	for any $t\ge0$. Without loss of generality, assume that $|w_j(z)|=e^{G_{\Omega_j}(z,z_j)}$ on $V_{z_j}$, hence $\psi=\max_{1\le j\le n}\{2p_j\log|w_j|\}$ on $V_0$. As $G_{\Omega_j}(z,z_j)$ is the Green function on $\Omega_j$,  there is $t_0>0$ such that
	$$\{\psi<-t_0\}\Subset V_0.$$
	 Denote that 
	 $$c_t=\sup_{\{\psi<-t\}}\varphi<+\infty$$
	for any $t\ge t_0$. As $\varphi=\sum_{1\le j\le n}\varphi_j$ is plurisubharmonic, we know that $\lim_{t\rightarrow+\infty}c_t=\varphi(z_0)=\sum_{1\le j\le n}\varphi_j(z_j)$.
	Let $F=\sum_{\alpha\in \mathbb{Z}_{\geq0}^n}b_{\alpha}w^{\alpha}dw_1\wedge\ldots\wedge dw_n$ near $z_0$. Denote that $E_1=\{\alpha\in \mathbb{Z}_{\geq0}^n:\sum_{1\le j\le n}\frac{\alpha_j+1}{p_j}\le 1\}$. Since $(F,z_0)\not\in\mathcal{I}(\psi)_{z_0}$, we have 
	$$\sum_{\alpha\in E_1}|d_{\alpha}|^2>0.$$
	  Lemma \ref{l:m2} tells us that 
	  \begin{equation}
	  	\label{eq:220807b}
	  	\begin{split}
	  			  	\int_{\{\psi<-t\}}|F|^2e^{-\varphi}&\ge e^{-c_t}\int_{\{\psi<-t\}}|F|^2\\
	  			  	&=e^{-c_t}\sum_{\alpha\in\mathbb{Z}_{\ge0}^n}e^{-\sum_{1\le j\le n}\frac{\alpha_j+1}{p_j}t}\frac{|d_{\alpha}|^2(2\pi)^n}{\Pi_{1\le j\le n}(\alpha_j+1)}\\
	  			  	&\geq e^{-c_t}\sum_{\alpha\in E_1}e^{-\sum_{1\le j\le n}\frac{\alpha_j+1}{p_j}t}\frac{|d_{\alpha}|^2(2\pi)^n}{\Pi_{1\le j\le n}(\alpha_j+1)}
	  	\end{split}
	  \end{equation} 
	  for any $t\ge t_0$. It follows from equality \eqref{eq:220807a} and inequality \eqref{eq:220807b} that 
	  \begin{equation}\label{eq:220807c}
	  \begin{split}
	  		  	\frac{G(0)}{\int_0^{+\infty}c(t)e^{-t}dt}&=\lim_{t\rightarrow+\infty}e^t \int_{\{\psi<-t\}}|F|^2e^{-\varphi}\\
	  		  	&\geq \lim_{t\rightarrow+\infty}e^{-c_t}\sum_{\alpha\in E_1}e^{\left(1-\sum_{1\le j\le n}\frac{\alpha_j+1}{p_j}\right)t}\frac{|d_{\alpha}|^2(2\pi)^n}{\Pi_{1\le j\le n}(\alpha_j+1)},
	  \end{split}
	  \end{equation}
	  Note that $\frac{G(0)}{\int_0^{+\infty}c(t)e^{-t}dt}\in(0,+\infty)$ and $1-\sum_{1\le j\le n}\frac{\alpha_j+1}{p_j}\ge0$ for any $\alpha\in E_1$, inequality \eqref{eq:220807c} shows that 
	  $$\lim_{t\rightarrow+\infty}c_t>-\infty,$$
	  hence we have $\varphi_j(z_j)>-\infty$ for any $1\le j\le n$.
\end{proof}

Let $c_j(z_j)$ be the logarithmic capacity (see \cite{S-O69}) on $\Omega_j$, which is locally defined by
$$c_j(z_j):=\exp\lim_{z\rightarrow z_j}(G_{\Omega_j}(z,z_j)-\log|w_j(z)|).$$
\begin{Remark}[\cite{GY-concavity4}]
	\label{r:1.1}When the three statements in Theorem \ref{thm:linear-2d} hold,
$$\sum_{\alpha\in E}\tilde{d}_{\alpha}\wedge_{1\le j\le n}\pi_j^*\left(g_j(P_j)_*\left(f_{u_j}f_{z_j}^{\alpha_j}df_{z_j}\right)\right)$$
 is the unique holomorphic $(n,0)$ form $F$ on $M$ such that $(F-f,z_0)\in(\mathcal{O}(K_{M}))_{z_0}\otimes\mathcal{I}(\psi)_{z_0}$ and
	$$G(t)=\int_{\{\psi<-t\}}|F|^2e^{-\varphi}c(-\psi)=\left(\int_t^{+\infty}c(l)e^{-l}dl\right)\sum_{\alpha\in E}\frac{|d_{\alpha}|^2(2\pi)^ne^{-\varphi(z_{0})}}{\prod_{1\le j\le n}(\alpha_j+1)c_{j}(z_j)^{2\alpha_{j}+2}}$$
	 for any $t\ge0$, where $P_j:\Delta \rightarrow\Omega_j$ is the universal covering, $f_{u_j}$ is a holomorphic function on $\Delta$ such that $|f_{u_j}|=P_j^*(e^{u_j})$ for any $j\in\{1,\ldots,n\}$, $f_{z_j}$ is a holomorphic function on $\Delta$ such that $|f_{z_j}|=P_j^*\left(e^{G_{\Omega_j}(\cdot,z_j)}\right)$ for any $j\in\{1,\ldots,n\}$ and $\tilde{d}_{\alpha}$ is a constant such that $\tilde{d}_{\alpha}=\lim_{z\rightarrow z_0}\frac{d_{\alpha}w^{\alpha}dw_1\wedge\ldots\wedge dw_n}{\wedge_{1\le j\le n}\pi_j^*\left(g_j(P_j)_*\left(f_{u_j}f_{z_j}^{\alpha_j}df_{z_j}\right)\right)}$ for any $\alpha\in E$.
\end{Remark}

\section{Hardy space over $\partial M$}\label{sec:partial m}

Let $D_j$ be a planar regular region with finite boundary components which are analytic Jordan curves   for any $1\le j\le n$. Let $$M=\prod_{1\le j\le n}D_j$$ be a bounded domain in $\mathbb{C}^n$. 

 Let  $M_j=\prod_{1\le l\le n,l\not=j}D_l$, then $M=D_j\times M_j$. Note that $\partial M=\cup_{j=1}^n\partial D_j\times \overline{M_j}$. Let $z_j\in D_j$ for any $1\le j\le n$. Recall that $H^2(D_j)$ denotes  the Hardy space on $D_j$ and there exists a norm-preserving linear map  $\gamma_j:H^2(D_j)\rightarrow L^2(\partial D_j,\frac{\partial G_{D_j}(z,z_j)}{\partial v_z})$ (see Section \ref{sec:2.1}) satisfying that $\gamma_j(f)$ denotes the nontangential boundary value of $f$ a.e. on $\partial D_j$ for any $f\in H^2(D_j)$.

Let $d\mu_j$ be the Lebesgue measure on $M_j$ for any $1\le j\le n$ and $d\mu$ is a measure on $\partial M$ defined by
$$\int_{\partial M}hd\mu=\sum_{1\le j\le n}\frac{1}{2\pi}\int_{M_j}\int_{\partial D_j}h(w_j,\hat w_j)|dw_j|d\mu_j(\hat w_j)$$
 for any $h\in L^1(\partial M)$. For simplicity, denote $d\mu|_{\partial D_j\times M_j}$ by $d\mu$. Before defining the Hardy spaces over $\partial M$, we consider a space over $\partial D_j\times M_j$.
Denote 
\begin{displaymath}\begin{split}
	\{f\in L^2(\partial D_j\times M_j,d\mu):\exists f^*\in \mathcal{O}(M)\text{, s.t.  }&f^*(\cdot,\hat w_j)\in H^2(D_j) \text{ for any  $\hat w_j\in M_j$}\\
	&\,\&\, f=\gamma_j(f^*) \text{ a.e. on } \partial D_j\times M_j\}	
\end{split}\end{displaymath}
by  $H^2(M,\partial D_j\times M_j)$.  It follows from Lemma \ref{l:0-1} that $f^*$ is unique for any $f\in H^2(M,\partial D_j\times M_j).$
Thus, there exists a map $P_{\partial M,j}$ from $H^2(M,\partial D_j\times M_j)$ to $\mathcal{O}(M)$. When $n=1$, we consider $M_j$ as a single point set.

\begin{Lemma}
	\label{l:b0-p}There exists a unique linear injective map $P_{\partial M,j}$ from $H^2(M,\partial D_j\times M_j)$ to $\mathcal{O}(M)$ such that $P_{\partial M,j}(f)$ satisfies the following conditions for any $f\in H^2(M,\partial D_j\times M_j)$:
	
	$(1)$ $P_{\partial M,j}(f)(\cdot,\hat w_j)\in H^2(D_j)$ for any  $\hat w_j\in M_j$;

	$(2)$ $f=\gamma_j(P_{\partial M,j}(f))$ a.e. on $\partial D_j\times M_j$.
\end{Lemma}
\begin{proof}
	The uniqueness of $P_{\partial M,j}$ follows from Lemma \ref{l:0-1}. By the definition of $H^2(M,\partial D_j\times M_j)$, we know that there is a map $P_{\partial M,j}$ from $H^2(M,\partial D_j\times M_j)$ to $\mathcal{O}(M)$ such that $P_{\partial M,j}(f)$ satisfies the two conditions for any $f\in H^2(M,\partial D_j\times M_j)$, thus it suffices to prove that $P_{\partial M,j}$ is a linear injective map. 
	
	It is clear that $P_{\partial M,j}$ is a linear map from $H^2(M,\partial D_j\times M_j)$ to $\mathcal{O}(M)$. If $P_{\partial M,j}(f)=0$ for some $f\in H^2(M,\partial D_j\times M_j)$, noting that $f=\gamma_j(P_{\partial M,j}(f))$ a.e. on $\partial D_j\times M_j$ and $\gamma_j(P_{\partial M,j}(f))(\cdot,\hat w_1)$ denotes the nontangential boundary value of $P_{\partial M,j}(f)(\cdot,\hat w_j)$ a.e. on $\partial D_j$ for any $\hat w_j\in M_j$, then we have $f=0$ a.e. on $\partial D_j\times M_j$.
	
	Thus, Lemma \ref{l:b0-p} holds.
\end{proof}

It is clear that $P_{\partial M,1}=\gamma_1^{-1}$ when $n=1$. Denote that $f^*=P_{\partial M,j}(f)$  for any $f\in H^2(M,\partial D_j\times M_j)$. Denote that $\ll\cdot,\cdot\gg_{\partial D_j\times M_j}$ is the inner product of the Hilbert space $L^2(\partial D_j\times M_j ,d\mu)$, i.e.,
$$\ll f,g\gg_{\partial D_j\times M_j}=\frac{1}{2\pi}\int_{M_j}\int_{\partial D_j}f(w_j,\hat w_j)\overline{g(w_j,\hat w_j)} |dw_j|d\mu_j(\hat w_j)$$
for any $f,g\in L^2(\partial D_j\times M_j,d\mu)$. Denote that $\|f\|_{\partial D_j\times M_j}:=\left(\ll f,f\gg_{\partial D_j\times M_j}\right)^{\frac{1}{2}}$.
\begin{Lemma}
	\label{l:b1-p}For any compact subset $K$ of $M$, there exists a positive constant $C_K$ such that 
	$$|f^*(z)|\le C_K\|f\|_{\partial D_j\times M_j}$$
	holds for any $z\in K$ and $f\in H^2(M,\partial D_j\times M_j)$.
\end{Lemma}

\begin{proof}
Without loss of generality, assume that $j=1$.
There are a compact subset $K_1$ of $D_1$ and a compact subset $K_2$ of $M_1$ such that $K\subset K_1\times K_2$.
	Note that $f^*(w_1,\cdot)\in \mathcal{O}(M_1)$ for any $w_1\in D_1$, then there is a positive constant $C_{K_2}$  such that 
		\begin{equation}
		\label{eq:220730a}\sup_{\hat w_1\in K_2}|f^*(w_1,\hat w_1)|^2\le C_{K_2}\int_{M_1}|f^*(w_1,\cdot)|^2d\mu_1,
	\end{equation}
	where $C_{K_2}$ is independent of $w_1\in D_1$. 
As $f=\gamma_1(f^*)$ a.e. on $\partial D_1\times M_1$, it follows from Lemma \ref{l:0-1} that 
	\begin{displaymath}
		f^*(w_1,\hat w_1)=\frac{1}{2\pi\sqrt{-1}}\int_{\partial D_1}\frac{f(z_1,\hat w_1)}{z_1-w_1}dz_1
	\end{displaymath}
	holds for a.e. $\hat w_1\in M_1$ and for any $w_1\in D_1$, which implies that 
	\begin{equation}
		\label{eq:220730b}
		\sup_{w_1\in K_1}|f^*(w_1,\hat w_1)|^2\le C_{K_1}\frac{1}{2\pi} \int_{\partial D_1}|f(z_1,\hat w_1)|^2|dz_1|,
	\end{equation}
	holds for a.e. $\hat w_1\in M_1$, where $C_{K_1}>0$ is independent of $\hat w_1$. Following from inequality \eqref{eq:220730a} and inequality \eqref{eq:220730b}, we get that 
	\begin{displaymath}
		\begin{split}
			|f^*(z)|^2=&|f^*(w_1,\hat w_1)|^2\\
			\le & C_{K_2}\int_{M_1}|f^*(w_1,\hat z_1)|^2d\mu_1(\hat z_1)\\
			\le&C_{K_2}\int_{M_1}C_{K_1}\frac{1}{2\pi} \int_{\partial D_1}|f(z_1,\hat z_1)|^2|dz_1|d\mu_1(\hat z_1)\\
			=&C_{K_1}C_{K_2}\|f\|^2_{\partial D_j\times M_j}
		\end{split}
	\end{displaymath}
	holds for any $z=(w_1,\hat w_1)\in K$, hence Lemma \ref{l:b1-p} holds.
\end{proof}

We recall the following basic formula.
\begin{Lemma}[see \cite{GY-weightedsaitoh}]
	\label{l:v} $\frac{\partial G_{D_1}(\cdot,z_1)}{\partial v_z}=\left(\left(\frac{\partial G_{D_1}(\cdot,z_1)}{\partial x}\right)^2+\left(\frac{\partial G_{D_1}(\cdot,z_1)}{\partial y}\right)^2\right)^{\frac{1}{2}}$ on $\partial D_1$, where $z_1\in D_1$ and $\partial/\partial v_z$ denotes the derivative along the outer normal unit vector $v_z$. 
\end{Lemma}

The following lemma shows the requirement ``$f^*(\cdot,\hat w_j)\in H^2(D_j)$ for any $\hat w_j\in M_j$" in the definition of $H^2(M,\partial D_j\times M_j)$ can be reduced to ``$f^*(\cdot,\hat w_j)\in H^2(D_j)$ for a.e. $\hat w_j\in M_j$".

\begin{Lemma}
	\label{l:b3-p} Let $f\in L^2(\partial D_j\times M_j,d\mu)$. Assume that  there is $ f^*\in \mathcal{O}(M)$ such that  $f^*(\cdot,\hat w_j)\in H^2(D_j)$  for a.e.  $\hat w_j\in M_j$ and $  f=\gamma_j(f^*)$ a.e. on $ \partial D_j\times M_j$. Then we have $f^*(\cdot,\hat w_j)\in H^2(D_j)$ for any $\hat w_j\in M_j$, and for any compact subset $K$ of $M_j$, there exists a positive constant $C_K$ such that
	\begin{displaymath}
		\sup_{\hat w_j\in K}\frac{1}{2\pi}\int_{\partial D_j}|\gamma_j(f^*(\cdot,\hat w_j))|^2d|w_j|\le C_K\|f\|^2_{\partial D_j\times M_j},
	\end{displaymath}
	holds for any $f\in H^2(M,\partial D_j\times M_j)$.
\end{Lemma}
\begin{proof}
	Without loss of generality, assume that $j=1$. For fixed $z_0\in D_1$, denote that 
	$$D_{1,r}:=\{z\in D_1:G_{D_1}(z,z_0)<\log r\},$$
	where $G_{D_1}(\cdot,z_0)$ is the Green function on $D_1$ and $r\in(0,1)$. It is well-known that $G_{D_1}(\cdot,z_0)-\log r$ is the Green function on $D_{1,r}$.
	By the analyticity of the boundary of $D_1$, we have $G_{D_1}(z,w)$ has an analytic extension on $U\times V\backslash\{z=w\}$ and $\frac{\partial G_{D_1}(z,z_0)}{\partial v_z}$ is positive and smooth on $\partial D_1$, where  $\partial/\partial v_z$ denotes the derivative along the outer normal unit vector $v_z$, $U$ is a neighborhood of $\overline{D_1}$ and $V\Subset D_1$. Then there exist $r_0\in(0,1)$ and $C_1>0$ such that  $\frac{1}{C_1}\le|\bigtriangledown G_{D_1}(\cdot,z_0)|\le C_1$ on $\{z\in D_1:G_D(z,z_0)>\log r_0\}$, which implies 
	\begin{equation}
		\label{eq:0709b}\frac{1}{C_1}\le\frac{\partial G_{D_1}(z,z_0)}{\partial v_z}\le C_1
	\end{equation}
	holds on $\{z\in D:G_{D_1}(z,z_0)>\log r_0\}$ (by using Lemma \ref{l:v}).
Denote that 
$$U_{r,\hat w_1}(w_1):=\frac{1}{2\pi}\int_{\partial D_{1,r}}|f^*(z,\hat w_1)|^2\frac{\partial G_{D_{1,r}}(z,w_1)}{\partial v_z} |dz|$$
for any $\hat w_1\in M_1,$	
	where $r\in (r_0,1)$ and $G_{D_{1,r}}(\cdot,\cdot)$ is the Green function on $D_{1,r}$. Then we know that $U_{r,\hat w_1}(w_1)\in C(\overline{D_{1,r}})$, $U_{r,\hat w_1}(w_1)|_{\partial D_{1,r}}=|f^*(z,\hat w_1)|^2$ and $U_{r,\hat w_1}(w_1)$ is harmonic on $D_{1,r}$ for any $\hat w_1\in M_1$.
	As $|f^*(z,\hat w_1)|^2$ is subharmonic on $D_1$ for any $\hat w_1\in M_1$, then $U_{r,\hat w_1}$ is increasing with respect ro $r$. Note that $f^*(\cdot,\hat w_1)\in H^2(D_1)$ holds for a.e. $\hat w_1\in M_1$ and $G_{D_{1,r}}(\cdot,z_0)=G_{D_1}(\cdot,z_0)-\log r$, then it follows from Lemma \ref{l:0-4} that 
	\begin{equation}
		\label{eq:220730c}\lim_{r\rightarrow1-0}U_{r,\hat w_1}(z_0)=\frac{1}{2\pi}\int_{\partial D_1}|\gamma_1(f^*(\cdot,\hat w_1))|^2\frac{\partial G_{D_1}(z,z_0)}{\partial v_z}|dz|
	\end{equation}
	for a.e. $\hat w_1\in M_1$.
	As $f^*(z,\cdot)\in \mathcal{O}(M_1)$ for any $z\in D_1$, there is a constant $\tilde C_K>0$ such that 
	\begin{equation}
		\label{eq:220730d}\begin{split}
			U_{r,\hat w_1}(z_0)&=\frac{1}{2\pi}\int_{\partial D_{1,r}}|f^*(z,\hat w_1)|^2\frac{\partial G_{D_{1}}(z,z_0)}{\partial v_z} |dz|\\
			&\le\frac{1}{2\pi}\int_{\partial D_{1,r}}\tilde C_{K}\left(\int_{M_1}|f^*(z,\hat z_1)|^2d\mu_1(\hat z_1) \right)\frac{\partial G_{D_{1}}(z,z_0)}{\partial v_z} |dz|\\
			&=\tilde C_K\int_{M_1}\left(\frac{1}{2\pi}\int_{\partial D_{1,r}}|f^*(z,\hat z_1)|^2\frac{\partial G_{D_{1}}(z,z_0)}{\partial v_z} |dz|\right)d\mu_1(\hat z_1)\\
			&=\tilde C_K\int_{M_1}U_{r,\hat z_1}(z_0)d\mu_1(\hat z_1)
		\end{split}
	\end{equation}
	for any $\hat w_1\in K$ and any $r\in(r_0,1)$.
	Note that $U_{r,\hat z_1}(z_0)$ is increasing with respect to $r$ for any $\hat z_1\in M_1$ and As $\gamma_1(f^*)=f$ a.e. on $\partial D_1\times M_1$, then following from equality \eqref{eq:220730c}, we have 
\begin{equation}
	\label{eq:220730f}
	\begin{split}
			&\liminf_{r\rightarrow1-0}\int_{M_1}U_{r,\hat z_1}(z_0)d\mu_1(\hat z_1)\\
			=&\int_{M_1}\lim_{r\rightarrow1-0}U_{r,\hat z_1}(z_0)d\mu_1(\hat z_1)\\
			=&\int_{M_1}\frac{1}{2\pi}\left(\int_{\partial D_1}|f(\cdot,\hat z_1)|^2\frac{\partial G_{D_1}(z,z_0)}{\partial v_z}|dz|\right)d\mu_1(\hat z_1).
	\end{split}
\end{equation}
Inequality \eqref{eq:220730d} and equality \eqref{eq:220730f} show that $\lim_{r\rightarrow1-0}U_{r,\hat w_1}(z_0)<+\infty$ for any $\hat w_1\in M_1$. Note that $U_{r,\hat w_1}$ is increasing with respect to $r$. By Harnack's principle (see \cite{ahlfors}), the sequence $U_{r,\hat w_1}$ converges to a harmonic function $U_{\hat w_1}$ on $D_1$ when $r\rightarrow1-0$ for any $\hat w_1\in M_1$, which satisfies that $|f^*(\cdot,\hat w_1)|^2\le U_{\hat w_1}$. Thus, we have 
	$$f^*(\cdot,\hat w_1)\in H^2(D_1)$$
	 for any $\hat w_1\in M_1$. 	Thus, equality \eqref{eq:220730c} holds for any $\hat w_1\in M_1$.
	
Following from inequality \eqref{eq:0709b}, equality \eqref{eq:220730c} and inequality \eqref{eq:220730d}, we obtain that 	
\begin{equation}
	\label{eq:220730e}\begin{split}
			&\sup_{\hat w_1\in K}\frac{1}{2\pi}\int_{\partial D_1}|\gamma_j(f^*(\cdot,\hat w_1))|^2d|z|\\
			\le& C_1\sup_{\hat w_1\in K}\frac{1}{2\pi}\int_{\partial D_1}|\gamma_1(f^*(\cdot,\hat w_1))|^2\frac{\partial G_{D_1}(z,z_0)}{\partial v_z}d|z|\\
			=& C_1\sup_{\hat w_1\in K}\lim_{r\rightarrow1-0}\frac{1}{2\pi}\int_{\partial D_{1,r}}|f^*(z,\hat w_1)|^2\frac{\partial G_{D_{1}}(z,z_0)}{\partial v_z} |dz|\\
			\le&C_1\tilde C_K\liminf_{r\rightarrow1-0}\int_{M_1}U_{r,\hat z_1}(z_0)d\mu_1(\hat z_1).
		\end{split}
\end{equation}
It follows from inequality \eqref{eq:0709b}, inequality \eqref{eq:220730e} and inequality \eqref{eq:220730f} that 	
	\begin{displaymath}
		\begin{split}
			&\sup_{\hat w_1\in K}\frac{1}{2\pi}\int_{\partial D_1}|\gamma_j(f^*(\cdot,\hat w_1))|^2d|z|\\
			\le&C_1\tilde C_K\liminf_{r\rightarrow1-0}\int_{M_1}U_{r,\hat z_1}(z_0)d\mu_1(\hat z_1)\\
						=&C_1^2\tilde C_K\int_{M_1}\frac{1}{2\pi}\int_{\partial D_1}|f(z,\hat z_1)|^2|dz|d\mu_1(\hat z_1)\\
			=& C_1^2\tilde C_K\|f\|^2_{\partial D_1\times M_1}.
		\end{split}
	\end{displaymath}
	Thus, Lemma \ref{l:b3-p} holds.
	\end{proof}

The following lemma shows that  $H^2(M,\partial D_j\times M_j)$ is a Hilbert space equipped with the inner product $\ll \cdot,\cdot\gg_{\partial D_j\times M_j}$.
\begin{Lemma}
	\label{l:b2-p}Let $\{f_m\}_{m\in\mathbb{Z}_{>0}}\subset H^2(M,\partial D_j\times M_j)$ satisfy that 
	$\lim_{m\rightarrow+\infty}\|f_m-f\|_{\partial D_j\times M_j}=0,$ where $f\in L^2(\partial D_j\times M_j,d\mu)$. Then we have $f\in H^2(M,\partial D_j\times M_j)$ and $f_m^*$ uniform converges  $f^*$ on any compact subset of $M$. Especially, $H^2(M,\partial D_j\times M_j)$ is a closed subset of $L^2(\partial D_j\times M_j,d\mu)$ under the topology deduced by the norm $\|\cdot\|_{\partial D_j\times M_j}$.
\end{Lemma}

\begin{proof}Without loss of generality, assume that $j=1$.
	It follows from Lemma \ref{l:b1-p} that $f_m^*$ uniformly converges to a holomorphic function $g$ on $M$ on any compact subset of $M$. It follows from Lemma \ref{l:b3-p} that $\{\gamma_1(f_m^*(\cdot,\hat w_1))\}$ is a Cauchy sequence in $H^2(D_1,\partial D_1)$ for any $\hat w_1\in M_1$. Using Lemma \ref{l:0-1}, Lemma \ref{l:0-1b} and $f_m^*$ converges to $g$, we know that $g(\cdot,\hat w_1)\in H^2(D_1)$ and 
	\begin{equation}
		\label{eq:220730g}\lim_{m\rightarrow +\infty}\int_{\partial D_1}|\gamma_1 (f_m^*(\cdot,\hat w_1))-\gamma_1(g(\cdot,\hat w_1))|^2|dz|=0
	\end{equation}
	for any $\hat w_1\in M_1$.
	
	As 
	\begin{displaymath}
		\lim_{m\rightarrow+\infty}\int_{M_1}\frac{1}{2\pi}\int_{\partial D_1}|f_m(z,\hat w_1)-f(z,\hat w_1)|^2|dz|d\mu_1(\hat w_1) =0,
	\end{displaymath}
	then there is a subsequence of $\{f_m\}_{m\in\mathbb{Z}_{>0}}$ denoted also by $\{f_m\}_{m\in\mathbb{Z}_{>0}}$ such that 
	$$\lim_{m\rightarrow +\infty}\int_{\partial D_1}|f_m(z,\hat w_1)-f(z,\hat w_1)|^2|dz|=0$$
	holds for a.e. $\hat w_1\in M_1$. As  $\gamma_1(f_m^*)=f_m$ a.e. on $\partial D_1\times M_1$, equality \eqref{eq:220730g} shows that $\gamma_1(g)=f$ a.e. on $\partial D_1\times M_1$, which implies that $f\in H^2(M,\partial D_j\times M_j)$ and $f^*=g$. 
	
	Thus, Lemma \ref{l:b2-p} holds.
\end{proof}

Now, we consider the Hardy space $H^2(M,\partial M)$ over $\partial M$. We recall Definition \ref{def:1} for the case $\rho\equiv1$:

\emph{ For any $f\in L^2(\partial M,d\mu)$, we call $f\in H^2_{\rho}(M,\partial M)$ if there exists $f^*\in\mathcal{O}(M)$ such that for any $1\le j\le n$, $f^*(\cdot,\hat w_j)\in H^2(D_j)$  for any  $\hat w_j\in M_j$ and $f=\gamma_j(f^*)$ a.e. on $\partial D_j\times M_j$.} 

For any $f\in L^2(\partial M,d\mu)$, it is clear that $f\in H^2(M,\partial M)$ if and only if $f|_{\partial D_j\times M_j}\in H^2(M,\partial D_j\times M_j)$ for any $1\le j\le n$ and $P_{\partial M,j}(f|_{\partial D_j\times M_j})=P_{\partial M,k}(f|_{\partial D_j\times M_k})$ for any $j\not=k$.

Define $P_{\partial M}(f)=P_{\partial M,j}(f|_{\partial D_j\times M_j})$ for any $f\in H^2(M,\partial M)$. Following from Lemma \ref{l:b0-p}, we have the following lemma. 

\begin{Lemma}
	\label{l:b0}There exists a unique linear injective map $P_{\partial M}$ from $H^2(M,\partial M)$ to $\mathcal{O}(M)$ such that $P_{\partial M}(f)$ satisfies the following conditions for any $f\in H^2(M,\partial M)$:
	
	$(1)$ $P_{\partial M}(f)(\cdot,\hat w_1)\in H^2(D_j)$ for any $1\le j\le n$ and $\hat w_1\in M_j$;

	$(2)$ $f=\gamma_j(P_{\partial M}(f))$ a.e. on $\partial D_j\times M_j$.
\end{Lemma}

 Denote that $f^*=P_{\partial M}(f)$  for any $f\in H^2(M,\partial M)$.
  Denote that $\ll\cdot,\cdot\gg_{\partial M}$ is the inner product of the Hilbert space $L^2(\partial M,d\mu)$, i.e.,
$$\ll f,g\gg_{\partial M}=\sum_{1\le j\le n}\frac{1}{2\pi}\int_{M_j}\int_{\partial D_j}f(w_j,\hat w_j)\overline{g(w_j,\hat w_j)} |dw_j|d\mu_j(\hat w_j)$$
for any $f,g\in L^2(\partial M,d\mu)$. Denote that $\|f\|_{\partial M}:=\left(\ll f,f\gg_{\partial M}\right)^{\frac{1}{2}}$.

Following from Lemma \ref{l:b2-p}, we have the following lemma, which shows that $H^2(M,\partial M)$ is a Hilbert space equipped with the inner product $\ll f,g\gg_{\partial M}$.

\begin{Lemma}
	\label{l:b2}Let $\{f_m\}_{m\in\mathbb{Z}_{>0}}\subset H^2(M,\partial M)$ satisfy that $\lim_{m\rightarrow+\infty}\|f_m-f\|_{\partial M}=0$, where $f\in L^2(\partial M,d\mu)$. Then we have $f\in H^2(M,\partial M)$ and $f_m^*$ uniform converges  $f^*$ on any compact subset of $M$. Especially, $H^2(M,\partial M)$ is a close subset of $L^2(\partial M,d\mu)$ under the topology deduced by the norm $\|\cdot\|_{\partial M}$.
\end{Lemma}
\begin{proof}
	Lemma \ref{l:b2-p} shows that  $f|_{\partial D_j\times M_j}\in H^2(M,\partial D_j\times M_j)$  and $f_m^*$ uniform converges  $P_{\partial M,j}(f|_{\partial D_j\times M_j})$ on any compact subset of $M$ for any $1\le j\le n$. By definition of $H^2(M,\partial M)$, we have  $f\in H^2(M,\partial M)$. 
\end{proof}

 Let $\rho$ be a Lebesgue measurable function on $\partial M$ such that $\inf_{\partial M}\rho>0$. Denote that  
$$\ll f,g\gg_{\partial M,\rho}:=\sum_{1\le j\le n}\frac{1}{2\pi}\int_{M_j}\int_{\partial D_j}f(w_j,\hat w_j)\overline{g(w_j,\hat w_j)}\rho |dw_j|d\mu_j(\hat w_j)$$
for any $f,g\in L^2(\partial M,\rho d\mu)\subset L^2(\partial M,d\mu)$.
 The weighted Hardy space over $\partial M$ is defined as follows:
 \begin{displaymath}
 	H^2_{\rho}(M,\partial M):=\{f\in H^2(M,\partial M):\|f\|_{\partial M,\rho}<+\infty \}.
 \end{displaymath}
 As $\|f\|_{\partial M}\le\frac{1}{\inf_{\partial M}\rho>0} \|f\|_{\partial M,\rho}$ and $H^2(M,\partial M)$ is a Hilbert space equipped with the inner product $\ll \cdot,\cdot\gg_{\partial M}$, then $H^2_{\rho}(M,\partial M)$ is a Hilbert space equipped with the inner product $\ll \cdot,\cdot\gg_{\partial M,\rho}$.

  Let $\{e_m\}_{m\in\mathbb{Z}_{>0}}$ be a complete orthonormal basis for $H^2_{\rho}(M,\partial M)$.

\begin{Lemma}
	\label{l:b4}$\sup_{z\in K}\sum_{m=1}^{+\infty}|e^*_m(z)|^2<+\infty$ holds for any compact subset $K$ of $M$.
\end{Lemma}

\begin{proof}
	For any $z_0\in K$, let $f_{l,z_0}=\sum_{m=1}^{l}\overline{e_m^*(z_0)}e_m\in H^2_{\rho}(M,\partial M)$. As $\{e_m\}_{m\in\mathbb{Z}_{>0}}$ is a complete orthonormal basis for $H^2_{\rho}(M,\partial M)$, we have 
	$$\|f_{l,z_0}\|_{\partial M,\rho}^2=\sum_{m=1}^{l}|e_m^*(z_0)|^2.$$ It is clear that $f_{l,z_0}^*=\sum_{m=1}^{l}\overline{e_m^*(z_0)}e_m^*$, hence 
	$$f_{l,z_0}^*(z_0)=\sum_{m=1}^{l}|e_m^*(z_0)|^2.$$ It follows from Lemma \ref{l:b1-p} and $\inf_{\partial M}\rho>0$ that there exists a positive constant $C_K$ such that
	\begin{displaymath}
		\begin{split}
			\sum_{m=1}^{l}|e_m^*(z_0)|^2=|f_{l,z_0}^*(z_0)|\le C_K\|f_{l,z_0}\|_{\partial M,\rho}=C_K\left(	\sum_{m=1}^{l}|e_m^*(z_0)|^2\right)^{\frac{1}{2}},
		\end{split}
	\end{displaymath}
	which implies that $\sup_{z\in K}\sum_{m=1}^{+\infty}|e^*_m(z)|^2<+\infty.$
\end{proof}

We define a kernel function as follows: 
$$K_{\partial M,\rho}(z,\overline w):=\sum_{m=1}^{+\infty}e_m^*(z)\overline{e_m^*(w)}$$
for $(z,w)\in M\times M\subset \mathbb{C}^{2n}$. When $n=1$, $K_{\partial M,\rho}(z,\overline w)=K_{\rho}(z,\overline w)$ (the definition of $K_{\rho}(z,\overline w)$ can be seen in Section \ref{sec:2.1}). It follows from Lemma \ref{l:b4} and Lemma \ref{l:hartogs} that $K_{\partial M,\rho}$ is a holomorphic function on $M\times M$. Denote that 
$$K_{\partial M,\rho}(z):=K_{\partial M,\rho}(z,\overline z)$$
for $z\in M$.
\begin{Lemma}
	\label{l:b5}$\ll f,\sum_{m=1}^{+\infty}e_m\overline{e_m^*(w)}\gg_{\partial M,\rho}=f^*(w)$ holds for any $w\in M$ and any $f\in H_{\rho}^2(M,\partial M)$. 
\end{Lemma}
\begin{proof}
	As  $\{e_m\}_{m\in\mathbb{Z}_{>0}}$ is a complete orthonormal basis for $H^2_{\rho}(M,\partial M)$, we have $f=\sum_{m=1}^{+\infty}a_me_m$ (convergence under the norm $\|\cdot\|_{\partial M,\rho}$), where $a_m$ is a constant for any $m$. Lemma \ref{l:b2} shows that $f^*=\sum_{m=1}^{+\infty}a_me^*_m$ (uniform convergence on any compact subset of $M$). Then we have 
	$$\ll f,\sum_{m=1}^{+\infty}e_m\overline{e_m^*(w)}\gg_{\partial M,\rho}=\sum_{m=1}^{+\infty}a_me_m^*(w)=f^*(w)$$
	for any $w\in M$.
\end{proof}

Lemma \ref{l:b5} implies that the definition of $K_{\partial M,\rho}(z,\overline w)$ is independent of the choices of the orthonormal basis. 
\begin{Lemma}
	\label{l:sup-b}
	$K_{\partial M,\rho}(z_0)=\sup_{f\in H^2_{\rho}(M,\partial M)}\frac{|f^*(z_0)|^2}{\|f\|_{\partial M,\rho}^2}$ holds for any $z_0\in M$.
\end{Lemma}

\begin{proof}
	It follows from Lemma \ref{l:b5} that $f\mapsto f^*(z_0)$ is a bounded linear map from $H^2_{\rho}(M,\partial M)$ and the norm of the map equals 
	$$\sup_{f\in H^2_{\rho}(M,\partial M)}\frac{|f^*(z_0)|}{\|f\|_{\partial M,\rho}}=\|\sum_{m=1}^{+\infty}e_m\overline{e_m^*(z_0)}\|_{\partial M,\rho}.$$ Thus, we have 
	\begin{displaymath}
		\begin{split}
			K_{\partial M,\rho}(z_0)&=\sum_{m=1}^{+\infty}e_m^*(z_0)\overline{e_m^*(z_0)}\\
			&=\ll \sum_{m=1}^{+\infty}e_m\overline{e_m^*(z_0)},\sum_{m=1}^{+\infty}e_m\overline{e_m^*(z_0)}\gg_{\partial M,\rho}\\
			&
			=\sup_{f\in H^2_{\rho}(M,\partial M)}\frac{|f^*(z_0)|^2}{\|f\|_{\partial M,\rho}^2}.
		\end{split}
	\end{displaymath}
\end{proof}

Let $h_0$ be a holomorphic function on a neighborbood $V_0$ of $z_0$, and let $I$ be a proper ideal of $\mathcal{O}_{z_0}$.   We generalize the kernel function $K_{\partial M,\rho}(z_0)$ as follows: 
$$K_{\partial M,\rho}^{I,h_0}(z_0):=\frac{1}{\inf\left\{\|f\|_{\partial M,\rho}^2:f\in H^2_{\rho}(M,\partial M)\,\&\,(f^*-h_0,z_0)\in I\right\}}.$$ 
\begin{Lemma}
	\label{l:b7}Assume that $(h_0,z_0)\not\in I$ and $K_{\partial M,\rho}^{I,h_0}(z_0)>0$. Then $K_{\partial M,\rho}^{I,h_0}(z_0)<+\infty$ and there is a unique holomorphic function $f\in H^2_{\rho}(M,\partial M)$ such that $(f^*-h_0,z_0)\in I$ and $K_{\partial M,\rho}^{I,h_0}(z_0)=\frac{1}{\|f\|_{\partial M,\rho}^2}$. Furthermore,
for any   $\hat{f}\in H^2_{\rho}(M,\partial M)$ such that $(\hat f^*-h_0,z_0)\in I$,
we have the following equality
\begin{equation}\label{eq:0803b}
\begin{split}
\|\hat f\|_{\partial M,\rho}^2=\|f\|_{\partial M,\rho}^2+\|\hat f-f\|_{\partial M,\rho}^2.
\end{split}
\end{equation}
\end{Lemma}
\begin{proof}
	We prove $K_{\partial M,\rho}^{I,h_0}(z_0)<+\infty$ by contradiction: if not, there is $\{{f}_{j}\}_{j\in\mathbb{Z}_{>0}}\subset  H^2_{\rho}(M,\partial M)$
such that $\lim_{j\to+\infty}\|{f}_{j}\|_{\partial M,\rho}=0$ and
$(f^*_{j}-h_0,z_0)\in I$ for any $j$. 
It follows from Lemma \ref{l:b1-p} that $f^*_j$ uniformly converges to $0$ on any compact subset of $M$.
It follows from Lemma \ref{l:closedness}
that $h_0\in I$, which contradicts to the assumption $h_0\not\in I$.
Thus, we have $K_{\partial M,\rho}^{I,h_0}(z_0)<+\infty.$

Firstly, we prove the existence of $f$.
As $K_{\partial M,\rho}^{I,h_0}(z_0)>0,$
then there is $\{f_{j}\}_{j\in\mathbb{Z}_{>0}}\subset H^2_{\rho}(M,\partial M)$  such that
$\lim_{j\rightarrow+\infty}\|f_j\|^2_{\partial M,\rho}=\frac{1}{K_{\partial M,\rho}^{I,h_0}(z_0)}<+\infty$ 
and $(f^*_{j}-h_0,z_0)\in I$ for any $j$. 
Then there is a subsequence of $\{f_j\}_{j\in\mathbb{Z}_{>0}}$ denoted also by $\{f_j\}_{j\in\mathbb{Z}_{>0}}$, which weakly converges to an element $f\in H_{\rho}^2(M,\partial M)$, i.e.,
	\begin{equation}\label{eq:0803a}
		\lim_{j\rightarrow+\infty}\ll f_j,g\gg_{\partial M,\rho}=\ll f,g\gg_{\partial M,\rho}
	\end{equation}
holds for any $g\in H^2_{\rho}(M,\partial M)$. Hence we have 
\begin{equation}
	\label{eq:220807g}\|f\|_{\partial M,\rho}^2\le \lim_{j\rightarrow+\infty}\|f_j\|^2_{\partial M,\rho}=\frac{1}{K_{\partial M,\rho}^{I,h_0}(z_0)}.
\end{equation}
It follows from Lemma \ref{l:b1-p} that there is a subsequence of $\{f_j\}_{j\in\mathbb{Z}_{>0}}$ denoted also by $\{f_j\}_{j\in\mathbb{Z}_{>0}}$, which satisfies that $f_j^*$ uniformly converges to a holomorphic function $g_0$ on $M$ on any compact subset of $M$.
By Lemma \ref{l:b5} and equality \eqref{eq:0803a}, we get that 
$$\lim_{j\rightarrow+\infty}f_j^*(z)=f^*(z)$$
for any $z\in M$, hence we know that $f^*=g_0$ and $f_j^*$ uniformly converges to $f^*$ on any compact subset of $M$. Following from Lemma \ref{l:closedness} and $(f^*_{j}-h_0,z_0)\in I$ for any $j$, we get 
$$(f^*-h_0,z_0)\in I.$$
By definition of $K_{\partial M,\rho}^{I,h_0}(z_0)$ and inequality \eqref{eq:220807g}, we have 
$$\|f\|_{\partial M,\rho}^2=\frac{1}{K_{\partial M,\rho}^{I,h_0}(z_0)}.$$
Thus,  we obtain the existence of $f$.

Secondly, we prove the uniqueness of $f$ by contradiction:
if not, there exist two different $g_{1}\in H^2_{\rho}(M,\partial M)$ and $g_{2}\in H^2_{\rho}(M,\partial M)$ satisfying that $\|g_1\|_{\partial M,\rho}^2=\|g_1\|_{\partial M,\rho}^2=\frac{1}{K_{\partial M,\rho}^{I,h_0}(z_0)}$, 
$(g_{1}^*-h_0,z_0)\in I$ and $(g_{2}^*-h_0,z_0)\in I$. It is clear that 
$$(\frac{g^*_{1}+g^*_{2}}{2}-h_0,z_0)\in I.$$
Note that
\begin{equation}\nonumber
\begin{split}
\|\frac{g_1+g_2}{2}\|_{\partial M,\rho}^2+\|\frac{g_1-g_2}{2}\|_{\partial M,\rho}^2=
\frac{\|g_1\|_{\partial M,\rho}^2+\|g_2\|_{\partial M,\rho}^2}{2}=\frac{1}{K_{\partial M,\rho}^{I,h_0}(z_0)},
\end{split}
\end{equation}
then we obtain that
$$\|\frac{g_1+g_2}{2}\|_{\partial M,\rho}^2<\frac{1}{K_{\partial M,\rho}^{I,h_0}(z_0)},$$
 which contradicts the definition of $K_{\partial M,\rho}^{I,h_0}(z_0)$.

Finally, we prove equality \eqref{eq:0803b}.
It is clear that
for any complex number $\alpha$,
$f+\alpha (\hat f-f)\in H^2_{\rho}(M,\partial M)$ and $(f^*+\alpha (\hat f^*-f^*),z_0)\in I$,
and 
$$\|f\|^2_{\partial M,\rho}\le \|f+\alpha (\hat f-f)\|_{\partial M,\rho}<+\infty.$$
Thus we have 
$$\ll f,\hat f-f\gg_{\partial M,\rho}=0,$$
which implies that 
$$\|\hat f\|_{\partial M,\rho}^2=\|f\|_{\partial M,\rho}^2+\|\hat f-f\|_{\partial M,\rho}^2.$$

Thus, Lemma \ref{l:b7} has been proved.
\end{proof}

\begin{Lemma}
	\label{l:b8}Let $h_l$ be a holomorphic function on a neighborhood of $z_0$ for $l\in\{1,\ldots,m\}$. Let $f_l\in H^2_{\rho}(M,\partial M)$ satisfy that $(f_l^*-h_l,z_0)\in I$ and $K_{\partial M,\rho}^{I,h_l}(z_0)=\frac{1}{\|f_l\|_{\partial M,\rho}^2}$ for any $1\le l\le m$. Then we have 
	$(\sum_{1\le l\le m}f_l^*-\sum_{1\le l\le m}h_l,z_0)\in I$ and 
	$$K_{\partial M,\rho}^{I,\sum_{1\le l\le m}h_l}(z_0)=\frac{1}{\|\sum_{1\le l\le m}f_l\|_{\partial M,\rho}^2}.$$
\end{Lemma}
\begin{proof}Following from $(f_l^*-h_l,z_0)\in I$  for any 
$1\le l\le m$, we have 
$$(\sum_{1\le l\le m}f_l^*-\sum_{1\le l\le m}h_l,z_0)\in I.$$
	It follows from equality \eqref{eq:0803b} that 
	\begin{displaymath}
		\ll f_l,g\gg_{\partial M,\rho}=0
	\end{displaymath}
	holds for any $1\le l\le m$ and any $g\in H^2_{\rho}(M,\partial M)$ satisfying that $(g^*,z_0)\in I$. Hence we have that  	\begin{displaymath}
		\ll \sum_{1\le l\le m}f_l,g\gg_{\partial M,\rho}=0
	\end{displaymath}
	holds for any $g\in H^2_{\rho}(M,\partial M)$ satisfying that $(g^*,z_0)\in I$, which implies that 
	\begin{equation}\nonumber
\|\hat f\|_{\partial M,\rho}^2=\|\sum_{1\le l\le m}f_l\|_{\partial M,\rho}^2+\|\hat f-\sum_{1\le l\le m}f_l\|_{\partial M,\rho}^2
	\end{equation}
	holds for any
$\hat{f}\in H^2_{\rho}(M,\partial M)$ such that $(\hat f^*-\sum_{1\le l\le m}h_l ,z_0)\in I$. 
Thus, we obtain that $K_{\partial M,\rho}^{I,\sum_{1\le l\le m}h_l}(z_0)=\frac{1}{\|\sum_{1\le l\le m}f_l\|_{\partial M,\rho}^2}$.
\end{proof}

In the following part, we consider the case $n>1$.

Assume that there is  a holomorphic function $g\not\equiv0$ on a neighborhood of $\overline D_2$ such that $g$ has no zero point on  $\partial D_2$ and $\inf_{\partial M}\rho_*>0$, where $\rho_*:=\rho|g|^2$.

Let $I$ be a proper ideal of $\mathcal{O}_{z_0}$, and let  $h_0$ be a holomorphic function on a neighborbood $V_0$ of $z_0$ such that $(h_0,z_0)\not\in I.$ 

\begin{Lemma}
	\label{l:divide g}Assume that $K_{\partial M,\rho}^{I,h_0}(z_0)>0$. There exists a norm-preserving linear isomorphism $P:H^2_{\rho}(M,\partial M)\rightarrow H^2_{\rho_*}(M,\partial M)$ such that $P(f)=\frac{f}{g}$ and $P(f)^*=\frac{f^*}{g}$.
\end{Lemma}
\begin{proof}
	It is clear that $\|f\|_{\partial M,\rho}=\|\frac{f}{g} \|_{\partial M,\rho_*}$, thus it suffices to prove that $f\in H^2_{\rho}(M,\partial M)$ if and only if $\frac{f}{g}\in H^2_{\rho_*}(M,\partial M).$
	
	Take any $f\in H^2_{\rho}(M,\partial M).$ As $f^*(\cdot,\hat w_1)\in H^2(D_1)$ for any $\hat w_1\in M_1$ and $\gamma_1(f^*)=f$ a.e. on $\partial D_1\times M_1$, it follows from Lemma \ref{l:0-1} that for any $K\Subset D_1$, there is $C_K>0$ such that 
	\begin{equation}
		\label{eq:0806a}
		\sup_{w_1\in K_1}\left|f^*(w_1,\hat w_1)\right|^2\le C_{K_1}\frac{1}{2\pi} \int_{\partial D_1}\left|f(z_1,\hat w_1)\right|^2|dz_1|,
	\end{equation}
	holds for a.e. $\hat w_1\in M_1$. Since $g\not\equiv0$ and $\inf_{\partial M}\rho_*>0$, inequality \eqref{eq:0806a} implies that 
	\begin{equation}
\nonumber		\begin{split}&\int_{K_1}\int_{M_1}\left|\frac{f}{g}^*(w_1,\hat w_1)\right|^2d\mu_1(\hat w_1)d\mu_1(\hat w_1) \\
		\le&		C_1\sup_{w_1\in K_1}\int_{M_1}\left|\frac{f}{g}^*(w_1,\hat w_1)\right|^2d\mu_1(\hat w_1)\\
		\le& C_1C_{K_1}\int_{M_1}\left(\frac{1}{2\pi} \int_{\partial D_1}\left|\frac{f}{g}(z_1,\hat w_1)\right|^2|dz_1|\right)d\mu_1(\hat w_1)\\
					\leq& \frac{C_1C_{K_1}}{\inf_{\partial M}\rho_*}\int_{M_1}\left(\frac{1}{2\pi} \int_{\partial D_1}\left|f\right|^2\rho|dz_1|\right)d\mu_1(\hat w_1)\\
					<&+\infty,
		\end{split}
	\end{equation}
	which implies that $\frac{f^*}{g}$ is holomorphic on $M$. For $j\not=2$, as $g\not\equiv0$ and $f\in H^2_{\rho}(M,\partial M)$, we know that $\frac{f^*}{g}(\cdot,\hat w_j)\in H^2(D_j)$ for a.e. $\hat w_j\in M_j$ and $\gamma_j(\frac{f^*}{g})=\frac{f}{g}$ a.e. on $\partial D_j\times M_j$. Since $g$ has no zero point on $\partial D_2$ and $f\in H^2_{\rho}(M,\partial M)$, we know that  $\frac{f^*}{g}(\cdot,\hat w_2)\in H^2(D_2)$ for any $\hat w_2\in M_2$ and $\gamma_2(\frac{f^*}{g})=\frac{f}{g}$ a.e. on $\partial D_2\times M_2$.
It is clear that $\frac{f}{g}\in L^2(\partial M,\rho_*d\mu)$, then it follows from the difinition of $H^2_{\rho_*}(M,\partial M)$ and Lemma \ref{l:b3-p} that $\frac{f}{g}\in H^2_{\rho_*}(M,\partial M)$ and $\left(\frac{f}{g}\right)^*=\frac{f^*}{g}$.

Take any $h\in H^2_{\rho_*}(M,\partial M)$. Note that $g$ is  holomorphic on a neighborhood of $\overline{D_2}$. By definition of $H^2(D_j)$ and $h^*(\cdot,\hat w_j)\in H^2(D_j)$ for any $\hat w_j\in M_j$, we know that $gh^*(\cdot,\hat w_j)\in H^2(D_j)$ for any $\hat w_j\in M_j$ for any $1\le j\le n$. As $\gamma_j(f)$ denotes the nontangential boundary value of $f$ a.e. on $\partial D_j$ for any $f\in H^2(D_j)$, we have  $\gamma_j(gh^*)=gh$ a.e. on $\partial D_j\times M_j$ for any $1\le j\le n$. It is clear that $gh\in L^2(\partial M,\rho d\mu)$, then it follows from the definition of $H^*_{\rho}(M,\partial M)$ that $gh\in H^2_{\rho}(M,\partial M)$ and $(gh)^*=gh^*$.

Thus, Lemma \ref{l:divide g} holds.	
\end{proof}

 Denote that  
$$\ll f,g\gg_{\partial D_j\times M_j,\rho}:=\frac{1}{2\pi}\int_{M_j}\int_{\partial D_j}f(w_j,\hat w_j)\overline{g(w_j,\hat w_j)}\rho |dw_j|d\mu_j(\hat w_j)$$
for any $f,g\in L^2(\partial D_j\times M_j,\rho d\mu)\subset L^2(\partial D_j\times M_j,d\mu)$.
 Denote that
 \begin{displaymath}
 	H^2_{\rho}(M,\partial D_j\times M_j):=\{f\in H^2(M,\partial D_j\times M_j):\|f\|_{\partial D_j\times M_j,\rho}<+\infty \}.
 \end{displaymath}
Since $H^2(M,\partial D_j\times M_j)$ is a Hilbert space equipped with the inner product $\ll \cdot,\cdot\gg_{\partial D_j\times M_j}$,  we have that $H^2_{\rho}(M,\partial D_j\times M_j)$ is a Hilbert space equipped with the inner product $\ll \cdot,\cdot\gg_{\partial D_j\times M_j,\rho}$.

Assume that $\rho|_{\partial D_1\times M_1}=\rho_1\times\lambda_1$,  where $\rho_1$ is a positive Lebesgue measurable function on $\partial D_1$ and $\lambda_1$ is a positive Lebesgue measurable function on $M_1$.

\begin{Lemma}\label{l:prod-d1xm1}
	Assume that $H^2_{\rho}(M,\partial D_1\times M_1)\not=\{0\}$. Then we have $H^2_{\rho_1}(D_1,\partial D_1)\not=\{0\}$ and $A^2(M_1,\lambda_1)\not=\{0\}$. Furthermore,
	$\{e_m(z)\tilde e_l(w)\}_{m,l\in\mathbb{Z}_{>0}}$ is a complete orthonormal basis for $H^2_{\rho}(M,\partial D_1\times M_1)$, where $\{e_m\}_{m\in\mathbb{Z}_{>0}}$ is a complete orthonormal basis for $H^2_{\rho_1}(D_1,\partial D_1)$, and $\{\tilde e_m\}_{m\in\mathbb{Z}_{>0}}$ is a complete orthonormal basis for $A^2(M_1,\lambda_1)$. 
\end{Lemma}
\begin{proof}
	Let $f\in H^2_{\rho}(M,\partial D_1\times M_1)$ satisfying $f\not\equiv0$. By definition, we know that $f(\cdot,\hat w_1)\in H^2_{\rho_1}(D_1,\partial D_1)$ for a.e. $\hat w_1\in M_1$, hence $H^2_{\rho_1}(D_1,\partial D_1)\not=\{0\}$. Note that $f^*(\cdot,\hat w_1)\in H^2(D_1)$, then it follows from Lemma \ref{l:0-1} that 
	\begin{equation}\nonumber
		|f^*(z_1,\hat w_1)|^2\le C_{z_1}\frac{1}{2\pi}\int_{\partial D_1}|f(w_1,\hat w_1)|^2|dw_1|
	\end{equation}
	holds for any $\hat w_1\in M_1$, where $z_1\in D_1$ and $C_{z_1}>0$ is a constant independent of $\hat w_1$. Then we have 
	\begin{displaymath}
		\begin{split}
			\int_{M_1}|f^*(z_1,\cdot)|^2\lambda_1 d\mu_1&\le \frac{C_{w_1}}{\inf_{\partial D_1}\rho_1}\int_{M_1}\left(\frac{1}{2\pi}\int_{\partial D_1}|f(w_1,\hat w_1)|^2\rho_1|dw_1|\right)\lambda_1d\mu_1\\
			&=\frac{C_{w_1}}{\inf_{\partial D_1}\rho_1}\|f\|_{\partial D_1\times M_1,\rho}\\
			&<+\infty,
		\end{split}
	\end{displaymath}
	which shows that $f^*(z_1,\cdot)\in A^2(M_1,\lambda_1)$ for any $z_1\in D_1$. Hence $A^2(M_1,\lambda_1)\not=\{0\}.$
	
	Denote that $e_{m,l}:=e_m\tilde e_l$, and it is clear that $e_{m,l}\in H^2_{\rho}(M,\partial D_1\times M_1)$.
	It follows from Fubini's theorem that 
	$$\ll e_{m,l},e_{m',l'} \gg_{\partial D_1\times M_1,\rho}=0$$
	for any $(m,l)\not=(m',l')$, and $\|e_{m,l}\|_{\partial D_1\times M_1,\rho}=1$ for any $m,l\in \mathbb{Z}_{>0}$. Thus, it suffices to prove that, if $f\in H^2_{\rho}(M,\partial D_1\times M_1)$ and $\ll f,e_{m,l}\gg_{\partial D_1\times M_1,\rho}=0$ for any $m,l\in \mathbb{Z}_{>0}$ then $f=0$.
	
	Let $f\in H^2_{\rho}(M,\partial D_1\times M_1)$ satisfy $\ll f,e_{m,l}\gg_{\partial D_1\times M_1,\rho}=0$ for any $m,l\in \mathbb{Z}_{>0}$. Denote that  
	$$f_l(z):=\int_{M_1}f(z,\cdot)\overline{\tilde e_l}\lambda_1 d\mu_1$$ on $\partial D_1$ for any $l\ge0$. 
Using Fubini's theorem, we know that $f_l\in L^2(\partial D_1,\rho_1)$. 
	 For any holomorphic function $\phi$ on a neighborhood of $\overline{D_1}$, it follows from Lemma \ref{l:0-1} that 
	 \begin{equation}\nonumber
	 	\begin{split}
	 			 	\int_{\partial D_1}f_l(z)\phi(z)dz&=	\int_{\partial D_1}\left(\int_{M_1}f(z,\cdot)\overline{\tilde e_l}\lambda_1 d\mu_1\right)\phi(z)dz\\
	 			 	&=\int_{M_1}\left(\int_{\partial D_1}f(z,\hat w_1)\phi(z)dz\right)\overline{\tilde e_l}\lambda_1d\mu_1(\hat w_1)\\
	 			 	&=0,
	 	\end{split}
	 \end{equation}
	  which shows that $f_l\in H^2_{\rho_1}(D_1 ,\partial D_1)$. 
	  
	Note that $\ll f,e_{m,l}\gg_{\partial D_1\times M_1,\rho}=0$ for any $m,l\in \mathbb{Z}_{>0}$ and $ f_l(z)=\ll f(z,\cdot), \tilde e_l\gg_{M_1,\lambda_1}$, then  it follows from Fubini's theorem that $\ll f_l,e_m\gg_{\partial D_1,\rho_1}=0$ for any $m,l\in \mathbb{Z}_{>0}$. As $\{e_m\}_{m\in\mathbb{Z}_{>0}}$ is a complete orthonormal basis for $H^2_{\rho_1}(D_1,\partial D_1)$, then we have $ f_l=0$ a.e. on $\partial D_1$ for any $l$. As $\{\tilde e_m\}_{m\in\mathbb{Z}_{>0}}$ is a complete orthonormal basis for $A^2(M_1,\lambda_1)$, then we have $f=0$ a.e. on $\partial D_1\times M_1$. 
	
	Thus, Lemma \ref{l:prod-d1xm1} has been proved.
\end{proof}

Let $z_0=(z_1,\ldots,z_n)\in M$, and let  
$$\psi=\max_{1\le j\le n}\{2p_j\log|w_j-z_j|\}$$ be a plurisubharmonic function on $\mathbb{C}^n$, where $p_j>0$ is a real number for any $1\le j\le n$.
Let $\gamma$ be a nonnegative integer such that $\frac{\gamma+1}{p_1}<1$, and let 
$$\tilde{\psi}_{\gamma}=\max_{2\le j\le n}\{2p_j(1-\frac{\gamma+1}{p_1})\log|w_j-z_j|\}$$
 be a plurisubharmonic function on $\mathbb{C}^{n-1}$. Denote that $\hat z_1:=\{z_2,\ldots,z_n\}\in M_1$.
 
The following lemma will be used in the proof of Lemma \ref{l:b9}.
\begin{Lemma}[see \cite{GY-concavity4}]\label{l:b9-0} Let $f$ be any holomorphic function on a neighborhood of $z_0$. If $(f,z_0)\in\mathcal{I}(\psi)_{z_0}$, then $((\frac{\partial}{\partial w_1})^{\sigma}f(z_1,\cdot),\hat z_1)\in\mathcal{I}(\tilde{\psi}_{\gamma})_{\hat z_1}$ for any integer $\sigma$ satisfying $0\le \sigma\le \gamma$.
	\end{Lemma}

\begin{Lemma}
	\label{l:b9}Let $f_1\in H^2_{\rho_1}(D_1,\partial D_1)$ such that 
	$$\frac{1}{2\pi}\int_{\partial D_1} f_1\overline g\rho_1|dw_1|=0$$ for any $g\in H^2_{\rho_1}(D_1,\partial D_1)$ satisfying $ord_{z_1}(g^*)\ge\gamma+1$. Let $f_2\in A^2_{\lambda_1}(M_1)$ such that 
	$$\int_{M_1}f_2\overline g\lambda_1 d\mu_1=0$$
	for any $g\in A^2_{\lambda_1}(M_1)$ satisfying $(g,\hat z_1)\in\mathcal{I}(\tilde\psi_{\gamma})_{z_1}$. Then we have 
	$$\ll f_1f_2,g\gg_{\partial D_1\times M_1,\rho}=0$$
	for any $g\in H^2_{\rho}(M,\partial D_1\times M_1)$ satisfying $(g^*,z_0)\in\mathcal{I}(\psi)_{z_0}$.
\end{Lemma}
\begin{proof}
Let $g\in H^2_{\rho}(M,\partial D_1\times M_1)$ satisfying $(g^*,z_0)\in\mathcal{I}(\psi)_{z_0}$. It follows from Lemma \ref{l:0-1} and $\inf_{\partial D_1}\rho_1>0$ that for any compact subset $K$ of $D_1$, there exists a constant $C_K>0$ such that 
	\begin{equation}\label{eq:0804a}
	\sup_{w_1\in K}|g^*(w_1,\hat w_1)|^2\le \frac{C_K}{2\pi}\int_{\partial D_1}|g(\cdot,\hat w_1)|^2\rho_1|dw_1|
	\end{equation}
	holds for any $\hat w_1\in M_1$, which shows that $g^*(w_1,\cdot)\in L^2(M_1,d\mu_1)$ for any $w_1\in D_1$.  Denote that  
	$$g_1(z):=\int_{M_1}g(z,\cdot)\overline{f_2}\lambda_1 d\mu_1$$ on $\partial D_1$ and 
	$$\tilde g_1(z):=\int_{M_1}g^*(z,\cdot)\overline{f_2}\lambda_1 d\mu_1$$ on $D_1$. 
Using Fubini's theorem, we know that $g_1\in L^2(\partial D_1,\rho_1)$. 
	 For any holomorphic function $\phi$ on a neighborhood of $\overline{D_1}$, it follows from Lemma \ref{l:0-1} that 
	 \begin{equation}\nonumber
	 	\begin{split}
	 			 	\int_{\partial D_1}g_1(z)\phi(z)dz&=	\int_{\partial D_1}\left(\int_{M_1}g(z,\cdot)\overline{f_2}\lambda_1 d\mu_1\right)\phi(z)dz\\
	 			 	&=\int_{M_1}\left(\int_{\partial D_1}g(z,\hat w_1)\phi(z)dz\right)\overline{f_2}\lambda_1d\mu_1(\hat w_1)\\
	 			 	&=0,
	 	\end{split}
	 \end{equation}
	  which shows that $g_1\in H^2_{\rho_1}(D_1 ,\partial D_1)$. As $g\in H^2_{\rho}(M,\partial D_1\times M_1)$, it follows from Lemma \ref{l:0-1}, Lemma \ref{l:b0} and Fubini's theorem that 
	 \begin{displaymath}
	 	\begin{split}
	 	&\frac{1}{2\pi\sqrt{-1}}\int_{\partial D_1}\frac{g_1(w_1)}{w_1-w}|dw_1|\\
	 		=&\frac{1}{2\pi\sqrt{-1}}\int_{\partial D_1}\frac{1}{w_1-w}\left(\int_{M_1}g(w_1,\cdot)\overline{f_2}\lambda_1 d\mu_1\right)d|w_1|\\
	 		=&\int_{M_1}\left(\frac{1}{2\pi\sqrt{-1}}\int_{\partial D_1}\frac{g(w_1,\hat w_1)}{w_1-w}|dw_1|\right)\overline{f_2}\lambda_1d\mu_1(\hat w_1)\\
	 		=&\int_{M_1}g^*(w,\cdot)\overline{f_2}\lambda_1d\mu_1(\hat w_1)\\
	 		=&\tilde g_1(w),
	 	\end{split}
	 \end{displaymath}
	 which implies that $g_1\in\mathcal{O}(D_1)$ and $g_1^*=\tilde g_1$. 
	 	
	Now, we prove that $ord_{z_1}(\tilde g_1)\ge \gamma+1$.
	It follows from inequality \eqref{eq:0804a} and $g\in H^2_{\rho}(M,\partial M)$ that $\left(\frac{\partial}{\partial w_1}\right)^{\sigma}g^*(w_1,\cdot)\in A^2_{\lambda_1}(M_1)$ for any $w_1\in D_1$ and any nonnegative integer $\sigma$. Denote that 
	$$\tilde g_{1,\sigma}(w_1):=\int_{M_1}\left(\frac{\partial}{\partial w_1}\right)^{\sigma} g^*(w_1,\cdot)\overline{f_2}\lambda_1 d\mu_1$$ on $D_1$ for any nonnegative integer $\sigma$.
	Note that $g^*(\cdot,\hat w_1)\in \mathcal{O}(D_1)$ for any $\hat w_1\in M_1$, then it follows from equality \eqref{eq:0804a} that  
	\begin{equation}
		\label{eq:0804b}
		\begin{split}
			&\sup_{|w_1-z_1|<\frac{r}{2}}\left|\left(\frac{\partial}{\partial w_1}\right)^{\sigma}g^*(w_1,\hat w_1)\right|\\
			\le& C_{1,\sigma}\sup_{|w_1-z_1|<r}|g^*(w_1,\hat w_1)|\\
			\le& \frac{C_{2,\sigma}}{2\pi}\int_{\partial D_1}|g(\cdot,\hat w_1)|^2\rho_1|dw_1|,
		\end{split}
			\end{equation}
	where  $r>0$ such that $\{z:|w_1-z_1|<r\}\Subset D_1$, and $C_{i,\sigma}$ is a constant independent of $\hat w_1 \in M_1$ for $i=1,2$. 
	By inequality \eqref{eq:0804b}, we have  
\begin{equation}
	\label{eq:0804c}\begin{split}
		&|\tilde g_{1,\sigma}(w_1)-\tilde g_{1,\sigma}(\tilde w_1)-(w_1-\tilde w_1)\tilde g_{1,\sigma+1}(w_1)|\\
		=&\bigg|\int_{M_1}\left(\left(\frac{\partial}{\partial w_1}\right)^{\sigma} g^*(w_1,\cdot)-\left(\frac{\partial}{\partial w_1}\right)^{\sigma} g^*(\tilde w_1,\cdot)-(w_1-\tilde w_1)\left(\frac{\partial}{\partial w_1}\right)^{\sigma+1} g^*(w_1,\cdot)\right)\\
		&\times\overline{f_2}\lambda_1d \mu_1\bigg|\\
		\le&|w_1-\tilde w_1|^2\cdot\left(\int_{M_1}\sup_{|z-z_1|<\frac{r}{2}}\left|\left(\frac{\partial}{\partial z}\right)^{\sigma+2}g^*(z,\cdot)\right|\lambda_1d\mu_1\right)^{\frac{1}{2}}\cdot\left(\int_{M_1}|f_2|^2\lambda_1d\mu_1 \right)^{\frac{1}{2}}\\
		\le&C_{2,\sigma}|w_1-\tilde w_1|^2\cdot\|g\|_{\partial D_1\times M_1,\rho}\cdot\left(\int_{M_1}|f_2|^2\lambda_1d\mu_1 \right)^{\frac{1}{2}},
	\end{split}
\end{equation}	
	where $|w_1-z_1|<\frac{r}{2}$ and $|\tilde w_1-z_1|<\frac{r}{2}$. Inequality \eqref{eq:0804c} shows that 
	$$\frac{\partial}{\partial w_1}\tilde g_{1,\sigma}=\tilde g_{1,\sigma+1}$$
	holds for any nonnegative integer $\sigma$.  It follows from $(g^*,z_0)\in\mathcal{I}(\psi)_{z_0}$ and Lemma \ref{l:b9-0} that $((\frac{\partial}{\partial w_1})^{\sigma} g^*(z_1,\cdot),\hat z_1)\in\mathcal{I}(\tilde\psi_{\gamma})_{\hat z_1}$ for any $\sigma\le \gamma$, which shows that 
	$$\tilde g_{1,\sigma}(z_1)=\int_{M_1}(\frac{\partial}{\partial w_1})^{\sigma} g^*(z_1,\cdot)\overline{f_2}\lambda_1d\mu_1=0$$
holds for any $\sigma\le \gamma$, i.e. $ord_{z_1}(\tilde g_1)\ge\gamma +1$. As $g_1^*=\tilde g_1$ and $g_1\in H^2_{\rho_1}(D_1,\partial D_1)$, then we have 
$$\frac{1}{2\pi}\int_{\partial D_1}g_1\overline{f_1}\rho_1|dw_1|=0.$$
	Thus, we have $\ll g,f_1f_2\gg_{\partial D_1\times M_1,\rho}=\frac{1}{2\pi}\int_{\partial D_1}\left(\int_{M_1}g\overline{f_2}\lambda_1d\mu_1 \right)\overline{f_1}\rho_1|dw_1|=0$.
\end{proof}

\section{Hardy space over $S$}\label{sec:S}

Let $D_j$ be a planar regular region with finite boundary components which are analytic Jordan curves  for any $1\le j\le n$. Let $M=\prod_{1\le j\le n}D_j$ be a bounded domain in $\mathbb{C}^n$.  

Denote that $S:=\prod_{1\le j\le n}\partial D_j$.
Denote 
\begin{displaymath}
	\left\{f\in L^2(S,d\sigma):\exists\{f_m\}_{m\in\mathbb{Z}_{\ge0}}\subset\mathcal{O}(M)\cap C(\overline M)\text{ s.t. } \lim_{m\rightarrow+\infty}\|f_m-f\|_{S}^2=0 \right\}
\end{displaymath} 
by $H^2(M,S)$, where $d\sigma:=\frac{1}{(2\pi)^n}|dw_1|\ldots|dw_n|$ and $\|f\|_S^2:=\frac{1}{(2\pi)^n}\int_S|f|^2|dw_1|\ldots|dw_n|$. Denote that 
$$\ll f,g\gg_S=\frac{1}{(2\pi)^n}\int_S f\overline g |dw_1|\ldots|dw_n|,$$
then $H^2(M,S)$ is a Hilbert space equipped with the inner product $\ll \cdot,\cdot\gg_S$.

\begin{Lemma}[see \cite{demailly-book}]
	\label{l:0}$f(z)=\frac{1}{2\pi\sqrt{-1}}\int_{\partial D_1}\frac{f(w)}{w-z}dw$ holds for $z\in D_1$ and any $f\in \mathcal{O}(D_1)\cap C(\overline{D_1})$.
\end{Lemma}

\begin{Lemma}
	\label{l:a1}For any compact subset $K$ of $M$, there exists a positive constant $C_K$ such that 
	$$|f(z)|\le C_K\|f\|_S$$
	holds for any $z\in K$ and $f\in\mathcal{O}(M)\cap C(\overline M)$.
\end{Lemma}
\begin{proof}
	For any $f\in\mathcal{O}(M)\cap C(\overline M)$ and any $j_0\in\{1,\ldots,n-1\}$, fixed $p\in\prod_{1\le j\le j_0}\partial D_j$, it is clear that $f(p,\cdot)\in \mathcal{O}(\prod_{j_0+1\le j\le n}D_j)$. For any $z=(z_1\ldots,z_n)\in K$, it follows from Lemma \ref{l:0} that  
	\begin{displaymath}
		\begin{split}
			|f(z)|^2&=|f(z_1,\ldots,z_n)|^2\\
			&\le\frac{1}{2\pi d(K,\partial M)} \int_{\partial D_1}|f(w_1,z_2,\ldots,z_n)|^2|dw_1|\\
			&\le \left(\frac{1}{2\pi d(K,\partial M)} \right)^2\int_{\partial D_1\times\partial D_2}|f(w_1,w_2,\ldots,z_n)|^2|dw_1||dw_2|\\
			&\le \left(\frac{1}{2\pi d(K,\partial M)} \right)^n\int_{S}|f(w_1,\ldots,w_n)|^2|dw_1|\ldots|dw_n|,
		\end{split}
	\end{displaymath}
	where $d(K,\partial M)=\inf\{|x-y|:x\in K\,\&\,y\in\partial M\}$.
	Thus, Lemma \ref{l:a1} holds.
\end{proof}

\begin{Lemma}
	\label{l:a2}There exists a unique linear map $P_S:H^2(M,S)\rightarrow\mathcal{O}(M)$ satisfying the following two conditions: 
	
$(1)$	$P_S(f)=f$ for any $f\in\mathcal{O}(M)\cap C(\overline M)$;

$(2)$ If sequence $\{f_m\}_{m\in\mathbb{Z}_{\ge0}}\subset H^2(M,S)$ satisfies $\lim_{m\rightarrow+\infty}\|f_m-f_0\|_S=0$, then $P_S(f_m)$ uniformly converges to $P_S(f_0)$ on any compact subset of $M$.
\end{Lemma}

\begin{proof}As $\mathcal{O}(M)\cap C(\overline M)$ is dense in $H^2(M,S)$ under the topology induced by the norm $\|\cdot\|_S$, the two conditions shows the uniqueness of $P_S$. Thus, it suffices to prove the existence of $P_S$.

	For any $f\in H^2(H,S)$, it follows from the definition of $H^2(M,S)$ that there exists $\{f_m\}_{m\in\mathbb{Z}_{>0}}\subset \mathcal{O}(M)\cap C(\overline M)$ such that $\lim_{m\rightarrow+\infty}\|f_m-f\|_S=0$. Using Lemma \ref{l:a1}, we know that $f_m$ uniformly converges to a holomorphic function $f^*$ on $M$ on any compact subset of $M$. If there exists another sequence $\{\hat f_m\}_{m\in\mathbb{Z}_{>0}}\subset \mathcal{O}(M)\cap C(\overline M)$ such that $\lim_{m\rightarrow+\infty}\|\hat f_m-f\|_S=0$, then $\lim_{m\rightarrow+\infty}\|f_m-\hat f_m\|_S=0$, which deduces that $\hat f_m$ uniformly converges to $f^*$  on any compact subset of $M$ by  Lemma \ref{l:a1}. Take 
	$$P_S(f)=f^*,$$
	   and $P_S$ is well defined. It is clear that $P_S:H^2(M,S)\rightarrow\mathcal{O}(M)$ is a linear map  and satisfies condition $(1)$.
	   
	   Now, we prove that $P_S$ satisfies condition $(2)$. Let sequence $\{g_m\}_{m\in\mathbb{Z}_{\ge0}}\subset H^2(M,S)$ satisfy $\lim_{m\rightarrow+\infty}\|g_m-g_0\|_S=0$. Choosing a sequence of open subsets $\{M_l\}_{l\in\mathbb{Z}_{>0}}$ satisfying $M_l\Subset M$  and $M_l\Subset M_{l+1}$ for any $l$, there exists $\hat g_m\in \mathcal{O}(M)\cap C(\overline M)$ such that 
	   $$\|\hat g_m-g_m\|_S<\frac{1}{m}\text{ and }\sup_{M_m}|\hat g_m-P_S(g_m)|<\frac{1}{m}$$
	   for any $m\in\mathbb{Z}_{>0}$.
	 Then  $\lim_{m\rightarrow+\infty}\|g_m-g_0\|_S=0$ implies that $\lim_{m\rightarrow+\infty}\|\hat g_m-g_0\|_S=0$, which shows that $\hat g_m$ uniformly converges to $P_S(g_0)$ on any compact subset of $M$ by the definition of $P_S$.
	 Note that $\sup_{M_m}|\hat g_m-P_S(g_m)|<\frac{1}{m}$, then we have $P_S(g_m)$ uniformly converges to $P_S(g_0)$ on any compact subset of $M$, which shows that condition $(2)$ holds.
\end{proof}

\begin{Remark}
	It follows from Lemma \ref{l:0}, Lemma \ref{l:a2} and the definition of $H^2(M,S)$ that 
	$$P_S(f)(z)=\left(\frac{1}{2\pi\sqrt{-1}}\right)^n\int_{\partial D_n}\cdot\cdot\cdot\int_{\partial D_1}\frac{f(w_1,\ldots,w_n)}{\prod_{j=1}^n(w_j-z_j)}dw_1\ldots dw_n$$
	holds for any $z\in M$ and any $f\in H^2(M,S)$.
\end{Remark}

\begin{Lemma}
	\label{l:a2-injective}The map $P_S:H^2(M,S)\rightarrow\mathcal{O}(M)$ is injective.
\end{Lemma}
\begin{proof}
	We prove Lemma \ref{l:a2-injective} by induction on $n$.

	Firstly, we consider the case $n=1$. Let $f\in H^2(M,S)$ satisfy $P_S(f)=0$. By definition of $P_S$,  there exists $\{f_m\}_{m\in\mathbb{Z}_{>0}}\subset \mathcal{O}(M)\cap C(\overline M)$ such that $\lim_{m\rightarrow+\infty}\|f_m-f\|_S=0$ and $f_m$ uniformly converges to $0$ on any compact subset of $M$. Note that $\gamma(f_m)=f_m$ and 
	$$\lim_{m\rightarrow+\infty}\|f_m-f\|_{\partial D,\rho}=0,$$
	where $\gamma$ and $\rho$ are as in Section \ref{sec:2.1}, 	
 it follows from Lemma \ref{l:0-1} and Lemma \ref{l:0-2} that there exists $f_0\in H^2(D_1)$ such that 
	$$\gamma (f_0)=f.$$ Using Lemma \ref{l:0-1b}, we have $f_m-f_0$ uniformly converges to $0$ on any compact subset of $M$, i.e. $f_0=0$, which implies $f=0$ a.e. on $\partial D_1$. Thus, Lemma \ref{l:a2-injective} holds for $n=1$.

	 We assume that Lemma \ref{l:a2-injective} holds for $n=k-1$, where $k\ge2$ is an integer.
	Now, we prove the case $n=k$. 
	
	Denote that $M_1:=\prod_{2\le j\le k}D_j$,  $S_1:=\prod_{2\le j\le k}\partial D_j$, and $P_{S_1}:H^2(M_1,S_1)\rightarrow \mathcal{O}(M_1)$ is the linear map defined in Lemma \ref{l:a2}. By assumption, we know $P_{S_1}$ is injective. Let $f\in H^2(M,S)$ satisfy $P_S(f)=0$. By definition of $P_S$,  there exists $\{f_m\}_{m\in\mathbb{Z}_{>0}}\subset \mathcal{O}(M)\cap C(\overline M)$ such that $\lim_{m\rightarrow+\infty}\|f_m-f\|_S=0$ and $f_m$ uniformly converges to $0$ on any compact subset of $M$. It suffices to prove that $f=0$ a.e. on $S$.
	Note that 
	$$\lim_{m\rightarrow+\infty}\int_{\partial D_1}\|f_m(w_1,\cdot)-f(w_1,\cdot)\|_{S_1}^2|dw_1| =\lim_{m\rightarrow+\infty}\|f_m-f\|^2_S=0$$
	and 
	$$\int_{\partial D_1}\|f(w_1,\cdot)\|_{S_1}^2|dw_1| =\|f\|^2_S<+\infty,$$
then there exists a subsequence of $\{f_m\}_{m\in\mathbb{Z}_{>0}}$ denoted also by $\{f_m\}_{m\in\mathbb{Z}_{>0}}$, which satisfies that 
$$\lim_{m\rightarrow+\infty}\|f_m(w_1,\cdot)-f(w_1,\cdot)\|_{S_1}^2=0\text{ and } \|f(w_1,\cdot)\|_{S_1}^2<+\infty$$
for any $w_1\in E_1$, where $\mu_1$ is the Lebesgue measure on $\partial D_1$ and $E_1\subset\partial D_1$ such that $\mu_1(E_1)=\mu_1(\partial D_1)$.
As $f_m\in\mathcal{O}(M)\cap C(\overline M)$, we have $f_m(w_1,\cdot)\in\mathcal{O}(M_1)\cap C(\overline{M_1})$ for any $w_1\in \overline{D_1}$. Hence we have $f(w_1,\cdot)\in H^2(M_1,S_1)$ for any $w_1\in E_1$. It follows from Lemma \ref{l:a2} that $P_{S_1}(f_m(w_1,\cdot))=f_m(w_1,\cdot)$ uniformly converges to $P_{S_1}(f(w_1,\cdot))$ on any compact subset of $M_1$.

For any $\tilde w_2\in M_1$, it follows from Lemma \ref{l:a1} that  there exists a positive constant $C_{\tilde w_2}$ such that 
\begin{displaymath}
	\begin{split}
		&\frac{1}{2\pi}\int_{\partial D_1}|f_m(w_1,\tilde w_2)-f_l(w_1,\tilde w_2)|^2|dw_1|\\
		\le& C_{\tilde w_2}\frac{1}{2\pi}\int_{\partial D_1}\|f_m(w_1,\cdot)-f_l(w_1,\cdot)\|_{S_1}^2|dw_1|\\
		=&C_{\tilde w_2}\|f_m(w_1,\cdot)-f_l(w_1,\cdot)\|_{S}^2,
	\end{split}
\end{displaymath}
$$$$
which implies that $\{f_m(\cdot,\tilde w_2)\}_{m\in \mathbb{Z}_{>0}}$ is a Cauchy sequence in $L^2(\partial D_1)$ for any $\tilde w_2\in M_1$. As $f_m(\cdot,\tilde w_2)$ uniformly converges to $0$ on any compact subset of $D_1$, it follows from Lemma \ref{l:0-3} that $f_m(\cdot,\tilde w_2)$ converges to $0$ in $L^2(\partial D_1)$. As  $f_m(w_1,\cdot)$ uniformly converges to $P_{S_1}(f(w_1,\cdot))$ on any compact subset of $M_1$ for any $w_1\in E_1\subset\partial D_1$, we get that for any $\tilde w_2\in M_1$, $f_m(\cdot,\tilde w_2)$ converges a.e. on $\partial D_1$. By the uniqueness of limit, we have that 
$$\lim_{m\rightarrow+\infty}f_m(\cdot,\tilde w_2)=0$$
a.e. on $\partial D_1$ for any $\tilde w_2\in M_1$.
Note that  $f_m(w_1,\cdot)$ uniformly converges to $P_{S_1}(f(w_1,\cdot))$ on any compact subset of $M_1$ for any $w_1\in E_1\subset\partial D_1$, then we have 
$$P_{S_1}(f(w_1,\cdot))=0$$
for any $w_1\in E_2\subset E_1\subset\partial D_1$, where $E_2$ satisfies $\mu_1(E_2)=\mu_1(E_1)=\mu_1(\partial D_1)$. By the assumption that Lemma \ref{l:a2-injective} holds for the case $n=k-1$, we have $f(w_1,\cdot)=0$ a.e. on $S_1$ for any $w_1\in E_2$, then $f=0$ a.e. on $S$. Hence, Lemma \ref{l:a2-injective} holds for $n=k$.

Lemma \ref{l:a2-injective} has been proved.
\end{proof}

When $n=1$, it follows from Lemma \ref{l:0-1} and Lemma \ref{l:0-2} that  $P_S$ is a linear bijection from $H^2(D_1,\partial D_1)$ to $H^2(D_1)$ and $P_S=\gamma_1^{-1}$, where $\gamma_1(f)$ denotes the nontangential boundary value of $f$ a.e. on $\partial D_1$ for any $f\in H^2(D_1)$.
Denote that $f^*=P_S(f)$ for any $f\in H^2(M,S)$.

\begin{Lemma}
	\label{l:a3}For any compact subset $K$ of $M$, there exists a positive constant $C_K$ such that 
	$$|f^*(z)|\le C_K\|f\|_S$$
	holds for any $z\in K$ and $f\in H^2(M,S)$.
\end{Lemma}
\begin{proof}
	It follows from the definition of $H^2(M,S)$ and Lemma \ref{l:a1} that there exists $\{f_m\}_{m\in\mathbb{Z}_{>0}}\subset \mathcal{O}(M)\cap C(\overline M)$ such that $\lim_{m\rightarrow+\infty}\|f_m-f\|_S=0$ and $f_m$ uniformly converges to $f^*$ on any compact subset of $M$. Lemma \ref{l:a1} shows that there exists a positive constant $C_K$ such that 
	$$|f_m^*(z)|\le C_K\|f_m\|_S$$
	holds for any $z\in K$. Letting $m\rightarrow+\infty$, we have $$|f^*(z)|\le C_K\|f\|_S$$
	holds for any $z\in K$.  Thus, Lemma \ref{l:a3} holds.
\end{proof}

Let $\lambda$ be a Lebesgue measurable function on $S$ satisfying $\inf_{S}\lambda>0$. Denote that
$$\ll f,g\gg_{S,\lambda}:=\frac{1}{(2\pi)^n} \int_Sf\overline g\lambda|dw_1|\ldots|dw_n|$$
for any $f,g\in L^2(S,\lambda d\sigma)\subset L^2(S, d\sigma)$. We recall Definition \ref{def2} as follows: 

\emph{For any $f\in L^2(S,\lambda d\sigma)$, we call $f\in H^2_{\lambda}(M,S)$ if there exists $\{f_m\}_{m\in\mathbb{Z}_{\ge0}}\subset\mathcal{O}(M)\cap C(\overline M)\cap L^2(S,\lambda d\sigma)$ such that $\lim_{m\rightarrow+\infty}\|f_m-f\|_{S,\lambda}^2=0$.}

Assume that $H^2_{\lambda}(M,S)\not=\{0\}$. Since $H^2(M,S)$ is a Hilbert space and $\inf_{S}\lambda>0$,
 it is clear that $H^2_{\lambda}(M,S)$ is a Hilbert space. Let $\{e_m\}_{m\in\mathbb{Z}_{>0}}$ be a complete orthonormal basis for $H^2_{\lambda}(M,S)$.

\begin{Lemma}
	\label{l:a4}$\sup_{z\in K}\sum_{m=1}^{+\infty}|e^*_m(z)|^2<+\infty$ holds for any compact subset $K$ of $M$.
\end{Lemma}

\begin{proof}
	For any $z_0\in K$, let $f_{l,z_0}=\sum_{m=1}^{l}\overline{e_m^*(z_0)}e_m\in H^2_{\lambda}(M,S)$. As $\{e_m\}_{m\in\mathbb{Z}_{>0}}$ is a complete orthonormal basis for $H^2_{\lambda}(M,S)$, we have 
	$$\|f_{l,z_0}\|_{S,\lambda}^2=\sum_{m=1}^{l}|e_m^*(z_0)|^2.$$ It is clear that $f_{l,z_0}^*=\sum_{m=1}^{l}\overline{e_m^*(z_0)}e_m^*$, hence 
	$$f_{l,z_0}^*(z_0)=\sum_{m=1}^{l}|e_m^*(z_0)|^2.$$ It follows from Lemma \ref{l:a3} and $\inf_{S}|\lambda|>0$ that there exists a positive constant $C_K$ such that
	\begin{displaymath}
		\begin{split}
			\sum_{m=1}^{l}|e_m^*(z_0)|^2=|f_{l,z_0}^*(z_0)|\le C_K\|f_{l,z_0}\|_{S,\lambda}=C_K\left(	\sum_{m=1}^{l}|e_m^*(z_0)|^2\right)^{\frac{1}{2}},
		\end{split}
	\end{displaymath}
	which implies that $\sup_{z\in K}\sum_{m=1}^{+\infty}|e^*_m(z)|^2<+\infty.$
\end{proof}

Denote that 
$$K_{S,\lambda}(z,\overline w):=\sum_{m=1}^{+\infty}e_m^*(z)\overline{e_m^*(w)}$$
for $(z,w)\in M\times M\subset \mathbb{C}^{2n}$. It follows from Lemma \ref{l:a4} and Lemma \ref{l:hartogs} that $K_{S,\lambda}$ is a holomorphic function on $M\times M$. Denote that 
$$K_{S,\lambda}(z):=K_{S,\lambda}(z,\overline z)$$
for $z\in M$.
\begin{Lemma}
	\label{l:a5}$\ll f,\sum_{m=1}^{+\infty}e_m\overline{e_m^*(w)}\gg_{S,\lambda}=f^*(w)$ holds for any $w\in M$ and any $f\in H_{\lambda}^2(M,S)$. 
\end{Lemma}
\begin{proof}
	As $\{e_m\}_{m\in\mathbb{Z}_{>0}}$ is a complete orthonormal basis for $H^2_{\lambda}(M,S)$, we have $f=\sum_{m=1}^{+\infty}a_me_m$ (convergence under the norm $\|\cdot\|_{S,\lambda}$), where $a_m$ is a constant for any $m$. Lemma \ref{l:a2} shows that $f^*=\sum_{m=1}^{+\infty}a_me^*_m$ (uniform convergence on any compact subset of $M$). Then we have 
	$$\ll f,\sum_{m=1}^{+\infty}e_m\overline{e_m^*(w)}\gg_{S,\lambda}=\sum_{m=1}^{+\infty}a_me_m^*(w)=f^*(w)$$
	for any $w\in M$.
\end{proof}

Lemma \ref{l:a5} implies that the definition of $K_{S,\lambda}(z,\overline w)$ is independent of the choices of the orthonormal basis. 
\begin{Lemma}
	\label{l:sup}$K_{S,\lambda}(z_0)=\sup_{f\in H^2_{\lambda}(M,S)}\frac{|f^*(z_0)|^2}{\|f\|_{S,\lambda}^2}$ holds for any $z_0\in M$.
\end{Lemma}

\begin{proof}
	It follows from Lemma \ref{l:a5} that $f\mapsto f^*(z_0)$ is a bounded linear map from $H^2_{\lambda}(M,S)$ and the norm of the map equals 
	$$\sup_{f\in H^2_{\lambda}(M,S)}\frac{|f^*(z_0)|}{\|f\|_{S,\lambda}}=\|\sum_{m=1}^{+\infty}e_m\overline{e_m^*(z_0)}\|_{S,\lambda}.$$ Thus, we have 
	\begin{displaymath}
		\begin{split}
			K_{S,\lambda}(z_0)&=\sum_{m=1}^{+\infty}e_m^*(z_0)\overline{e_m^*(z_0)}\\
			&=\ll \sum_{m=1}^{+\infty}e_m\overline{e_m^*(z_0)},\sum_{m=1}^{+\infty}e_m\overline{e_m^*(z_0)}\gg_{S,\lambda}\\
			&
			=\sup_{f\in H^2_{\lambda}(M,S)}\frac{|f^*(z_0)|^2}{\|f\|_{S,\lambda}^2}.
		\end{split}
	\end{displaymath}
\end{proof}

Let $h_0$ be a holomorphic function on a neighborbood $V_0$ of $z_0$, and let $I$ be a proper ideal of $\mathcal{O}_{z_0}$.  Denote that 
$$K_{S,\lambda}^{I,h_0}(z_0):=\frac{1}{\inf\left\{\|f\|_{S,\lambda}^2:f\in H^2_{\lambda}(M,S)\,\&\,(f^*-h_0,z_0)\in I\right\}}.$$ 
\begin{Lemma}
	\label{l:a K<+infty}Assume that $(h_0,z_0)\not\in I$ and $K_{S,\lambda}^{I,h_0}(z_0)>0$. Then $K_{S,\lambda}^{I,h_0}(z_0)<+\infty$ and there is a unique holomorphic function $f\in H^2_{\lambda}(M,S)$ such that $(f^*-h_0,z_0)\in I$ and $K_{S,\lambda}^{I,h_0}(z_0)=\frac{1}{\|f\|_{M,\lambda}^2}$. Furthermore,
for any   $\hat{f}\in H^2_{\lambda}(M,S)$ such that $(\hat f^*-h_0,z_0)\in I$,
we have the following equality
\begin{equation}\label{eq:220807d}
\begin{split}
\|\hat f\|_{S,\lambda}^2=\|f\|_{S,\lambda}^2+\|\hat f-f\|_{S,\lambda}^2.
\end{split}
\end{equation}
\end{Lemma}
\begin{proof} The proof follows from the the proof of Lemma \ref{l:b7} with small modifications. 

	We prove $K_{S,\lambda}^{I,h_0}(z_0)<+\infty$ by contradiction: if not, there is $\{{f}_{j}\}_{j\in\mathbb{Z}_{>0}}\subset  H^2_{\rho}(M,S)$
such that 
$$\lim_{j\to+\infty}\|{f}_{j}\|_{S,\lambda}=0$$ and
$(f^*_{j}-h_0,z_0)\in I$ for any $j$. 
It follows from Lemma \ref{l:a1} and $\inf_{S}\rho>0$ that $f^*_j$ uniformly converges to $0$ on any compact subset of $M$.
By Lemma \ref{l:closedness}, we have
 $h_0\in I$, which contradicts to the assumption $h_0\not\in I$.
Thus, we have $K_{S,\lambda}^{I,h_0}(z_0)<+\infty.$

Firstly, we prove the existence of $f$.
As $K_{S,\lambda}^{I,h_0}(z_0)>0,$
then there is $\{f_{j}\}_{j\in\mathbb{Z}_{>0}}\subset H^2_{\lambda}(M,S)$  such that
$$\lim_{j\rightarrow+\infty}\|f_j\|^2_{S,\lambda}=\frac{1}{K_{S,\lambda}^{I,h_0}(z_0)}<+\infty$$ 
and $(f^*_{j}-h_0,z_0)\in I$ for any $j$. 
Then there is a subsequence of $\{f_j\}_{j\in\mathbb{Z}_{>0}}$ denoted also by $\{f_j\}_{j\in\mathbb{Z}_{>0}}$, which weakly converges to an element $f\in H_{\lambda}^2(M,S)$, i.e.,
	\begin{equation}\label{eq:220807e}
		\lim_{j\rightarrow+\infty}\ll f_j,g\gg_{S,\lambda}=\ll f,g\gg_{S,\lambda}
	\end{equation}
holds for any $g\in H^2_{\lambda}(M,S)$. Hence we have 
\begin{equation}
	\label{eq:220807f}\|f\|_{S,\lambda}^2\le \lim_{j\rightarrow+\infty}\|f_j\|^2_{S,\lambda}=\frac{1}{K_{S,\lambda}^{I,h_0}(z_0)}.
\end{equation}
It follows from Lemma \ref{l:a1} that there is a subsequence of $\{f_j\}_{j\in\mathbb{Z}_{>0}}$ denoted also by $\{f_j\}_{j\in\mathbb{Z}_{>0}}$, which satisfies that $f_j^*$ uniformly converges to a holomorphic function $g_0$ on $M$ on any compact subset of $M$.
By Lemma \ref{l:a5} and equality \eqref{eq:220807e}, we get that 
$$\lim_{j\rightarrow+\infty}f_j^*(z)=f^*(z)$$
for any $z\in M$, hence we know that $f^*=g_0$ and $f_j^*$ uniformly converges to $f^*$ on any compact subset of $M$. Following from Lemma \ref{l:closedness} and $(f^*_{j}-h_0,z_0)\in I$ for any $j$, we get 
$$(f^*-h_0,z_0)\in I.$$
By definition of $K_{S,\lambda}^{I,h_0}(z_0)$ and inequality \eqref{eq:220807f}, we have 
$$\|f\|_{S,\lambda}^2=\frac{1}{K_{S,\lambda}^{I,h_0}(z_0)}.$$
Thus,  we obtain the existence of $f$.

Secondly, we prove the uniqueness of $f$ by contradiction:
if not, there exist two different $g_{1}\in H^2_{\lambda}(M,S)$ and $g_{2}\in H^2_{\lambda}(M,S)$ satisfying that $\|g_1\|_{S,\lambda}^2=\|g_1\|_{S,\lambda}^2=\frac{1}{K_{S,\lambda}^{I,h_0}(z_0)}$, 
$(g_{1}^*-h_0,z_0)\in I$ and $(g_{2}^*-h_0,z_0)\in I$. It is clear that 
$$(\frac{g^*_{1}+g^*_{2}}{2}-h_0,z_0)\in I.$$
Note that
\begin{equation}\nonumber
\begin{split}
\|\frac{g_1+g_2}{2}\|_{S,\lambda}^2+\|\frac{g_1-g_2}{2}\|_{S,\lambda}^2=
\frac{\|g_1\|_{S,\lambda}^2+\|g_2\|_{S,\lambda}^2}{2}=\frac{1}{K_{S,\lambda}^{I,h_0}(z_0)},
\end{split}
\end{equation}
then we obtain that
$$\|\frac{g_1+g_2}{2}\|_{S,\lambda}^2<\frac{1}{K_{S,\lambda}^{I,h_0}(z_0)},$$
 which contradicts the definition of $K_{S,\lambda}^{I,h_0}(z_0)$.

Finally, we prove equality \eqref{eq:220807d}.
It is clear that
for any complex number $\alpha$,
$f+\alpha (\hat f-f)\in H^2_{\lambda}(M,S)$ and $(f^*+\alpha (\hat f^*-f^*),z_0)\in I$,
and 
$$\|f\|^2_{S,\lambda}\le \|f+\alpha (\hat f-f)\|_{S,\lambda}<+\infty.$$
Thus we have 
$$\ll f,\hat f-f\gg_{S,\lambda}=0,$$
which implies that 
$$\|\hat f\|_{S,\lambda}^2=\|f\|_{S,\lambda}^2+\|\hat f-f\|_{S,\lambda}^2.$$

Thus, Lemma \ref{l:a K<+infty} has been proved.
\end{proof}

Let $M_a=\prod_{1\le j\le n_a}D_j$ be a bounded domain in $\mathbb{C}^{n_a}$, where $D_j$ is planar regular region with finite boundary components which are analytic Jordan curves for any $1\le j\le n_a$. Denote that $S_a:=\prod_{1\le j\le n_a}\partial D_j$. Let $M_b=\prod_{1\le j\le n_b}\tilde D_j$ be a bounded domain in $\mathbb{C}^{n_b}$, where $\tilde D_j$ is planar regular region with finite boundary components which are analytic Jordan curves for any $1\le j\le n_b$. Denote that $S_b:=\prod_{1\le j\le n_b}\partial \tilde D_j$. Denote that $M:=M_a\times M_b\subset \mathbb{C}^{n_a+n_b}$ and $S:=S_a\times S_b$.

Let $\lambda_a$ be a Lebesgue measurable function on $S_a$ such that $\inf_{S_a}\lambda_a>0$, and let $\lambda_a$ be a Lebesgue measurable function on $S_b$ such that $\inf_{S_b}\lambda_b>0$. Denote that $\lambda:=\lambda_a\lambda_b$ on $S$.   

\begin{Proposition}
	\label{p:7}Assume that $H^2_{\lambda}(M,S)\not=\{0\}$. Then we have $H^2_{\lambda_a}(M_a,S_a)\not=\{0\}$ and $H^2_{\lambda_b}(M_b,S_b)\not=\{0\}$. Furthermore,
	$\{e_m(z)\tilde e_l(w)\}_{m,l\in\mathbb{Z}_{>0}}$ is a complete orthonormal basis for $H^2_{\lambda}(M,S)$, where $\{e_m\}_{m\in\mathbb{Z}_{>0}}$ is a complete orthonormal basis for $H^2_{\lambda_a}(M_a,S_a)$, and $\{\tilde e_m\}_{m\in\mathbb{Z}_{>0}}$ is a complete orthonormal basis for $H^2_{\lambda_b}(M_b,S_b)$.
\end{Proposition}
\begin{proof}By definition of $H^2_{\lambda}(M,S)$, there is $f\in\mathcal{O}(M)\cap C(\overline{M})$ such that $f|_{S}\in H^2_{\lambda}(M,S)$. Thus, for any $z_{b}\in S_b$, we have $f(\cdot,z_b)\in\mathcal{O}(M_a)\cap C(\overline{M_a})$. Note that 
$$\int_{S_b}\int_{S_a}|f|^2\lambda_a\lambda_bd\sigma<+\infty,$$
which shows that $\|f(\cdot,z_b)\|_{S_a,\lambda_a}<+\infty$ for a.e. $z_b\in S_b$. Thus, we have $H^2_{\lambda_a}(M_a,S_a)\not=\{0\}$. Similarly, we have $H^2_{\lambda_b}(M_b,S_b)\not=\{0\}$.

Denote that $e_{m,l}:=e_m\tilde e_l$, and it is clear that $e_{m,l}\in H^2_{\lambda}(M,S)$.
	It follows from Fubini's theorem that 
	$$\ll e_{m,l},e_{m',l'} \gg_{S,\lambda}=0$$
	for any $(m,l)\not=(m',l')$, and $\|e_{m,l}\|_{S,\lambda}=1$ for any $m,l\in \mathbb{Z}_{>0}$. Thus, it suffices to prove that, if $f\in H^2_{\lambda}(M,S)$ and $\ll f,e_{m,l}\gg_{S,\lambda}=0$ for any $m,l\in \mathbb{Z}_{>0}$ then $f=0$.
	
	Let $f\in H^2_{\lambda}(M,S)$ satisfy $\ll f,e_{m,l}\gg_{S,\lambda}=0$ for any $m,l\in \mathbb{Z}_{>0}$. As $f\in H^2_{\lambda}(M,S)\subset L^2(S,\lambda)$, we have $f(z,\cdot)\in L^2(S_b,\lambda_b)$ for a.e. $z\in S_a$. 
	Denote that
	$$\tilde f_l(z):=\ll f(z,\cdot),\tilde e_l\gg_{S_b,\lambda_b},$$
	and following from Fubini's theorem, we have $\tilde f_l\in L^2(S_a,\lambda_a)$.
	
	By the definition of $H^2_{\lambda}(M,S)$, there is $\{f_j\}_{j\in\mathbb{Z}_{>0}}\subset\mathcal{O}(M)\cap C(\overline{M})$ such that $\lim_{j\rightarrow+\infty}\|f_j-f\|_{S,\lambda}=0$. Thus, we get that $\{\|f_j(z,\cdot)-f(z,\cdot)\|_{S_b,\lambda_b}\}_j\subset L^2(S_a,\lambda_a)$  converges to $0$ in $L^2(S_a,\lambda_a)$, which implies that there exists a subsequence of $\{\|f_j(z,\cdot)-f(z,\cdot)\|_{S_b,\lambda_b}\}_j$ converges to $0$ a.e. on $S_a$. By definition, we have 
	$$f(z,\cdot)\in H^2_{\lambda_b}(M_b,S_b)$$ for a.e. $z\in S_a$. Denote that 
	$$\tilde f_{j,l}(z):=\ll f_j(z,\cdot),\tilde e_l\gg_{S_b,\lambda_b}.$$
	By Fubini's theorem, we know $\tilde f_{j,l}(z)\in L^2(S_a,\lambda_a)$.
	Now, we prove $\tilde f_{j,l}\in \mathcal{O}(M_a)\cap C(\overline{M_a})$ for any $j,l\in\mathbb{Z}_{>0}$. $ \tilde f_{j,l}\in  C(\overline{M_a})$ follows from $f_l\in C(\overline{M})$, then it suffices to prove $\tilde f_{j,l}\in \mathcal{O}(M_a).$ Following from Lemma \ref{l:hartogs}, we can assume that $n_a=1$ in the proof of $\tilde f_{j,l}\in \mathcal{O}(M_a)$ without loss of generality. Note that $f_j(\cdot,w)\in \mathcal{O}(M_a)\cap C(\overline{M_a})$ for any $w\in \overline{M_b}$, then it follows from Lemma \ref{l:a1} that  
	\begin{equation}
		\label{eq:0730a}\sup_{|z-z_0|<\frac{r}{2}}\left|\left(\frac{\partial}{\partial z}\right)^2f_j(z,w)\right|\le C_1\sup_{|z-z_0|<r}|f_j(z,w)|\le C_2\| f_j(\cdot,w) \|_{S_a,\lambda_a},
	\end{equation}
	where $z_0\in M_a$, $r>0$ such that $\{z:|z-z_0|<r\}\Subset M_a$, and $C_i$ is a constant independent of $w\in \overline{M_b}$ for $i=1,2$. Denote that 
	$$\tilde g_{j,l}(z):=\ll \frac{\partial}{\partial z} f_j(z,\cdot),\tilde e_l\gg_{S_b,\lambda_b}.$$ 
	It follows from inequality \eqref{eq:0730a} that 
\begin{equation}
	\label{eq:0730b}\begin{split}
		&|\tilde f_{j,l}(z)-\tilde f_{j,l}(\tilde z)-(z-\tilde z)\tilde g_{j,l}(z)|\\
		=&\left|\ll f_j(z,\cdot)-f_j(\tilde z,\cdot)-(z-\tilde z) \frac{\partial}{\partial z} f_j(z,\cdot),\tilde e_l\gg_{S_b,\lambda_b}\right|\\
		\le&\|f_j(z,\cdot)-f_j(\tilde z,\cdot)-(z-\tilde z) \frac{\partial}{\partial z} f_j(z,\cdot)\|_{S_b,\lambda_b}\\
		\le&|z-\tilde z|^2\cdot\left\|\sup_{|z-z_0|<\frac{r}{2}}\left|\left(\frac{\partial}{\partial z}\right)^2f_j(z,\cdot)\right|\right\|_{S_b,\lambda_b}\\
		\le&C_2|z-\tilde z|^2\cdot\|f_j\|_{S,\lambda},
	\end{split}
\end{equation}	
	where $|z-z_0|<\frac{r}{2}$ and $|\tilde z-z_0|<\frac{r}{2}$. Inequality \eqref{eq:0730b} shows that $\tilde f_{j,l}$ is holomorphic on $M_a$. Thus, we have proved that $\tilde f_{j,l}\in \mathcal{O}(M_a)\cap C(\overline{M_a})$, which implies that $\tilde f_{j,l}\in H^2_{\lambda_a}(M,S)$. It follows from $\lim_{j\rightarrow+\infty}\|f_j-f\|_{S,\lambda}=0$ and $\|\tilde e_l\|_{S_b,\lambda_b}=1$ that 
	 \begin{equation}
	\label{eq:0730c}
	 	\begin{split}
	 		&\lim_{j\rightarrow+\infty}\|\tilde f_{j,l}-\tilde f_l\|_{S_a,\lambda_a}\\
	 		&=\lim_{j\rightarrow+\infty}\|\ll f_j(z,\cdot)-f(z,\cdot),\tilde e_l\gg_{S_b,\lambda_b}\|_{S_a,\lambda_a}\\
	 		&\leq \lim_{j\rightarrow+\infty}\|f_j-f\|_{S,\lambda}\\
	 		&=0.
	 	\end{split}
	 \end{equation}
As $\tilde f_{j,l}\in H^2_{\lambda_a}(M,S)$ and $\tilde f_l\in L^2(S_a,\lambda_a)$,	inequality \eqref{eq:0730c} shows that $\tilde f_l\in H^2_{\lambda_a}(M,S)$.

	Note that $\ll f,e_{m,l}\gg_{S,\lambda}=0$ for any $m,l\in \mathbb{Z}_{>0}$ and $\tilde f_l(z)=\ll f(z,\cdot),\tilde e_l\gg_{S_b,\lambda_b}$, then  it follows from Fubini's theorem that $\ll\tilde f_l,e_m\gg_{S_a,\lambda_a}=0$ for any $m,l\in \mathbb{Z}_{>0}$. As $\{e_m\}_{m\in\mathbb{Z}_{>0}}$ is a complete orthonormal basis for $H^2_{\lambda_a}(M_a,S_a)$, then we have $\tilde f_l=0$ a.e. on $S_a$ for any $l$. As $\{\tilde e_m\}_{m\in\mathbb{Z}_{>0}}$ is a complete orthonormal basis for $H^2_{\lambda_b}(M_b,S_b)$ and $f(z,\cdot)\in H^2_{\lambda_b}(M_b,S_b)$ for a.e. $z\in S_a$, then we have $f=0$ a.e. on $S$.
	
	Thus, Proposition \ref{p:7} has been proved. 
\end{proof}

In the following lemma, we consider the case $n=1$.
	Denote that 
	$$L_m:=\{f\in H^2_{\lambda}(M,S):ord_{z_0}(f^*)=m\}$$ for $m\in\mathbb{Z}_{\ge0}$. Assume that $H^2_{\lambda}(M,S)\not=\{0\}$, then $L_m\not=\emptyset$ for any $m\ge0$ (as $M\Subset \mathbb{C}$).

\begin{Lemma}\label{l:a6}
There is a complete orthonormal basis $\{e_m\}_{m\in \mathbb{Z}_{\ge0}}$ for $H^2_{\lambda}(M,S)$ satisfying that $ord_{z_0}(e_m^*)=m$ for any $m\in \mathbb{Z}_{\ge0}$. 
\end{Lemma}
\begin{proof}
	As $M$ is a bounded domain in $\mathbb{C}$ and $H^2_{\lambda}(M,S)\not=\{0\}$,  then $L_m\not=\emptyset$ for any $m\in\mathbb{Z}_{\ge0}$.
	For any $m\in\mathbb{Z}_{\ge0}$, denote that 
	$$l_m:=\inf\{\|f\|_{S,\lambda}:f\in L_m\,\&\,(f^*)^{(m)}(z_0)=1\}<+\infty,$$
	then there exists $\{f_l\}_{l\in\mathbb{Z}_{>0}}\subset L_m$ such that $(f_l^*)^{(m)}(z_0)=1$ for any $l$ and 
	\begin{equation}
		\label{eq:0729a}\lim_{l\rightarrow+\infty}\|f_l\|_{S,\lambda}=l_m.
	\end{equation}
	Note that $\inf_{S}\lambda>0$, it follows from Lemma \ref{l:a3} that there is a subsequence of $\{f_l\}_{l\in\mathbb{Z}_{>0}}$ denoted also by $\{f_l\}_{l\in\mathbb{Z}_{>0}}$ such that $f_l^*$ uniformly converges to a holomorphic function $h$  on $M$ on any compact subset of $M$. As $H_{\lambda}^2(M,S)$ is a Hilbert space, following from inequality \eqref{eq:0729a}, we get that there is a subsequence of $\{f_l\}_{l\in\mathbb{Z}_{>0}}$ denoted also by $\{f_l\}_{l\in\mathbb{Z}_{>0}}$, which weakly converges to an element $\hat{e}_m\in H_{\lambda}^2(M,S)$, i.e.,
	\begin{equation}
		\label{eq:0729b}\lim_{l\rightarrow+\infty}\ll f_l,g\gg_{S,\lambda}=\ll \hat{e}_m,g\gg_{S,\lambda}
	\end{equation}
	holds for any $g\in H_{\lambda}^2(M,S)$. Then, by equality \eqref{eq:0729a}, we have 
	\begin{equation}
		\label{eq:0729c}\|\hat{e}_m\|_{S,\lambda}\le \lim_{l\rightarrow+\infty}\|f_l\|_{S,\lambda}=l_m.
	\end{equation}
	 It follows from Lemma \ref{l:a5} and equality \eqref{eq:0729b} that $f_l^*$ pointwise converges to $\hat{e}_m^*$, which implies that $\hat{e}_m^*=h$. As $f_l^*$ uniformly converges to  $h$   on any compact subset of $M$ and $\{f_l\}_{l\in\mathbb{Z}_{>0}}\subset L_m$ such that $(f_l^*)^{(m)}(z_0)=1$ for any $l$, we obtain that $\hat{e}_m\in L_m$ and $(\hat{e}_m^*)^{(m)}(z_0)=1$. Following from the definition of $l_m$ and inequality \eqref{eq:0729c}, we have $\|\hat{e}_m\|_{S,\lambda}=l_m$. Thus, for any $m\ge0$, there exists $\hat{e}_m\in L_m$ such that $(\hat{e}_m^*)^{(m)}(z_0)=1$ and 
	 \begin{equation}
	 	\label{eq:0729d}\|\hat{e}_m\|_{S,\lambda}=l_m=\inf\{\|f\|_{S,\lambda}:f\in L_m\,\&\,(f^*)^{(m)}(z_0)=1\}.
	 \end{equation}
Denote that 
$$e_m:=\frac{\hat e_m}{\|\hat{e}_m\|_{S,\lambda}},$$
then $\|{e}_m\|_{S,\lambda}=1$ and $ord_{z_0}(e_m^*)=m$ for any $m\in\mathbb{Z}_{\ge0}$.

In the following, we prove that $\{e_m\}_{m\in\mathbb{Z}_{\ge0}}$ is a complete orthonormal basis for $H^2_{\lambda}(M,S)$.

For any $m_2>m_1\ge0$, note that $ord_{z_0}(e_{m_2}^*)=m_2>m_1=ord_{z_0}(e_{m_1}^*)$, then it follows  from equality \eqref{eq:0729d} that 
\begin{displaymath}
	\begin{split}
		\|e_{m_1}+ce_{m_2}\|_{S,\lambda}&=\frac{1}{\|\hat e_{m_1}\|_{S,\lambda}}\|\hat e_{m_1}+c\|\hat e_{m_1}\|_{S,\lambda}e_{m_2}\|_{S,\lambda} \\
		&\ge \frac{1}{\|\hat e_{m_1}\|_{S,\lambda}}\|\hat e_{m_1}\|_{S,\lambda}\\
		&=\|e_{m_1}\|_{S,\lambda}
	\end{split}
\end{displaymath}
holds for any complex number $c$,
which implies that 
$$\ll e_{m_1},e_{m_2}\gg_{S,\lambda}=0.$$
Let $f\in H_{\lambda}^2(M,S)$ satisfy $\ll f,e_m\gg_{S,\lambda}=0$ for any $m\in\mathbb{Z}_{\ge0}$. We prove that $f=0$ by contradiction: if not, then it follows from Lemma \ref{l:a2-injective} that $f^*\not=0$. Denote that $k:=ord_{z_0}(f^*)<+\infty$, then there is a constant $c_1\not=0$ such that $ord_{z_0}(f^*-c_1e^*_k)>k$. It follows from equality \eqref{eq:0729d} that 
$$\|e_k+c(f-c_1e_k)\|_{S,\lambda}\ge \|e_k\|_{S,\lambda}$$
holds for any complex number $c$,
which implies that 
$$\ll e_k,f-c_1e_k\gg_{S,\lambda}=0.$$
As $\ll f,e_k\gg_{S,\lambda}=0$, we have $c_1\|e_k\|_{S,\lambda}=0$, which contradicts to $c\not=0$ and $\|e_k\|_{S,\lambda}=1$. Then we have $f=0$.

Thus, $\{e_m\}_{m\in\mathbb{Z}_{\ge0}}$ is a complete orthonormal basis for $H^2_{\lambda}(M,S)$, and Lemma \ref{l:a6} has been proved.
\end{proof}

Let $\lambda_j$ be a positive continuous function on $\partial D_j$ for any $1\le j\le n$, and let $\lambda=\prod_{1\le j\le n}\lambda_j$ on $S=\prod_{1\le j\le n}\partial D_j$. Assume that $H^2_{\lambda}(M,S)\not=\{0\}$.

We prove the following product property of $K_{S,\lambda}(z_0)$ by using Lemma \ref{l:a6} and Proposition \ref{p:7}. 

\begin{Proposition}
	\label{p:8}$K_{S,\lambda}(z_0)=\prod_{1\le j\le n} K_{\partial D_j,\lambda_j}(z_j)$ holds for any  $z_0=(z_1,\ldots,z_n)\in M$.
\end{Proposition}
\begin{proof} Using Lemma \ref{l:a6}, there is 
	a complete orthonormal basis $\{e_{j,m}\}_{m\in\mathbb{Z}_{\ge0}}$   for $H^2_{\lambda_j}(D_j,\partial D_j)$ satisfying that $ord_{z_j}(e_{j,m}^*)=m$ for any $m\in\mathbb{Z}_{\ge0}$ and any $1\le j\le n$. It follows from Proposition \ref{p:7} that $\{\prod_{1\le j\le n}e_{j,\sigma_j}\}_{\sigma\in \mathbb{Z}_{\ge0}^n}$ is a complete orthonormal basis for $H^2_{\lambda}(M,S)$.
	
For any $j$, as $\ll e_{j,0},f\gg_{\partial D_j,\lambda_j}=0$ for any $f\in H_{\lambda_j}^2(D_j,\partial D_j)$ satisfying $f^*(z_j)=0$, then we have  that 
\begin{equation}
	\label{eq:0730d}\|e_{j,0}\|_{\partial D_j,\lambda_j}=\inf\{\|f\|_{\partial D_j,\lambda_j}:f\in H^2_{\lambda_j}(D_j,\partial D_j)\,\&\,f^*(z_j)=e_{j,0}^*(z_j)\}=1.
\end{equation}
Proposition \ref{p:7} shows that $\{\prod_{1\le j\le n}e_{j,\alpha_j}\}_{\alpha\in\mathbb{Z}_{\ge0}^n}$ is a complete orthonormal basis for $H^2_{\lambda}(M,S)$, then
we have $\ll \prod_{1\le j\le n}e_{j,0},f\gg_{S,\lambda}=0$ for any $f\in H_{\lambda}^2(M,S)$ satisfying $f^*(z_0)=0$, which implies  that 
\begin{equation}
	\label{eq:0730e}\|\prod_{1\le j\le n}e_{j,0}\|_{S,\lambda}=\inf\{\|f\|_{M,\lambda}:f\in H^2_{\lambda}(M,S)\,\&\,f^*(z_0)=\prod_{1\le j\le n}e_{j,0}^*(z_j)\}=1.
\end{equation}
Following from Lemma \ref{l:sup}, equality \eqref{eq:0730d} and equality \eqref{eq:0730e} that 	
	\begin{displaymath}
		\begin{split}
			 K_{\partial D_j,\lambda_j}(z_j)&=\sup_{f\in H^2_{\lambda_j}(D_j,\partial D_j)}\frac{|f^*(z_0)|^2}{\|f\|_{\partial D_j,\lambda_j}^2}\\
			 &=\frac{1}{\inf\{\|f\|^2_{\partial D_j,\lambda_j}:f\in H^2_{\lambda_j}(D_j,\partial D_j)\,\&\,f^*(z_j)=1\}}\\
			 &=|e^*_{j,0}(z_j)|^2
		\end{split}
	\end{displaymath}
	and
		\begin{displaymath}
		\begin{split}
			 K_{S,\lambda}(z_0)&=\sup_{f\in H^2_{\lambda}(M,S)}\frac{|f^*(z_0)|^2}{\|f\|_{S,\lambda}^2}\\
			 &=\inf\{\|f\|_{M,\lambda}:f\in H^2_{\lambda}(M,S)\,\&\,f^*(z_0)=1\}\\
			 &=|\prod_{1\le j\le n}e_{j,0}^*(z_j)|^2.
		\end{split}
	\end{displaymath}
	Thus, we have $K_{S,\lambda}(z_0)=\prod_{1\le j\le n} K_{\partial D_j,\lambda_j}(z_j)$.
\end{proof}

Let $z_0=(z_1,\ldots,z_n)\in M$. Let $h_j$ be a holomorphic function on a neighborhood of  $z_j$, and let $$h_0=\prod_{1\le j\le n}h_j.$$
Denote that $\beta_j:=ord_{z_j}(h_j)$ and $\beta:=(\beta_1,\cdots,\beta_n)$. Let $\tilde\beta=(\tilde\beta_1,\ldots,\tilde\beta_n)\in\mathbb{Z}_{\ge0}^n$ satisfy $\tilde\beta_j\ge\beta_j$ for any $1\le j\le n$.
Let 
$$I=\left\{(g,z_0)\in\mathcal{O}_{z_0}:g=\sum_{\alpha\in\mathbb{Z}_{\ge0}^n}b_{\alpha}(w-z_0)^{\alpha}\text{ near  }z_0\text{ s.t. $b_{\alpha}=0$ for $\alpha\in L_{\tilde\beta}$}\right\},$$
where $L_{\tilde\beta}=\{\alpha\in\mathbb{Z}_{\ge0}^n:\alpha_j\le\tilde\beta_j$ for any $1\le j\le n\}$.
It is clear that $(h_0,z_0)\not\in I$.

\begin{Proposition}
	\label{p:9}$K_{S,\lambda}^{I,h_0}(z_0)=\prod_{1\le j\le n} K_{\partial D_j,\lambda_j}^{I_{\tilde\beta_j,z_j},h_j}(z_j)$ holds, where $I_{\tilde\beta_j,z_j}=((w_j-z_j)^{\tilde\beta_j+1})$ is an ideal of $\mathcal{O}_{D_j,z_j}$.
\end{Proposition}
\begin{proof} Using Lemma \ref{l:a6}, there is 
	a complete orthonormal basis $\{e_{j,m}\}_{m\in\mathbb{Z}_{\ge0}}$   for $H^2_{\lambda_j}(D_j,\partial D_j)$ satisfying that $ord_{z_j}(e_{j,m}^*)=m$ for any $m\in\mathbb{Z}_{\ge0}$ and any $1\le j\le n$. It follows from Proposition \ref{p:7} that $\{\prod_{1\le j\le n}e_{j,\sigma_j}\}_{\sigma\in \mathbb{Z}_{\ge0}^n}$ is a complete orthonormal basis for $H^2_{\lambda}(M,S)$.
	
For any $1\le j\le n$,	there is a constant $c_{j,m}$  for $\beta_j\le m\le \tilde\beta_j$, which satisfies that $c_{j,\beta_j}\not=0$ and  $(h_j-\sum_{\beta_j\le m\le\tilde\beta_j}c_{j,m}e_{j,m},z_j)\in I_{\tilde\beta_j,z_j}$. It is clear that 
 $\ll e_{j,0},f\gg_{\partial D_j,\lambda_j}=0$ for any $f\in H_{\lambda_j}^2(D_j,\partial D_j)$ satisfying $(f^*,z_j)\in I_{\tilde\beta_j,z_j}$, then we have  that 
	\begin{equation}
		\label{eq:0811a}
		\begin{split}
	\sum_{\beta_j\le m\le\tilde\beta_j}|c_{j,m}|^2&=\|\sum_{\beta_j\le m\le\tilde\beta_j}c_{j,m}e_{j,m}\|^2_{\partial D_j,\lambda_j}\\
	&=\inf\{\|f\|^2_{\partial D_j,\lambda_j}:f\in H^2_{\lambda_j}(D_j,\partial D_j)\,\&\,(f^*-h_j,z_j)\in I_{\tilde\beta_j,z_j}\}\\
	&=\frac{1}{K_{\partial D_j,\lambda_j}^{I_{\tilde\beta_j,z_j},h_j}}.
		\end{split}
	\end{equation}
Note that $(h_0-\prod_{1\le j\le n}\sum_{\beta_j\le m\le\tilde\beta_j}c_{j,m}e_{j,m},z_0)\in I$.
Proposition \ref{p:7} shows that $\{\prod_{1\le j\le n}e_{j,\alpha_j}\}_{\alpha\in\mathbb{Z}_{\ge0}^n}$ is a complete orthonormal basis for $H^2_{\lambda}(M,S)$, then
we have $\ll \prod_{1\le j\le n}\sum_{\beta_j\le m\le\tilde\beta_j}c_{j,m}e_{j,m},f\gg_{S,\lambda}=0$ for any $f\in H_{\lambda}^2(M,S)$ satisfying $(f^*,z_0)\in I$, which implies  that 
	\begin{equation}
		\label{eq:0811b}
		\begin{split}
		\prod_{1\le j\le n}\sum_{\beta_j\le m\le\tilde\beta_j}|c_{j,m}|^2&=\|\prod_{1\le j\le n}\sum_{\beta_j\le m\le\tilde\beta_j}c_{j,m}e_{j,m}\|^2_{\partial D_j,\lambda_j}\\
		&=\inf\{\|f\|_{S,\lambda}^2:f\in H^2_{\lambda}(M,S)\,\&\,(f^*-h_0,z_0)\in I\}\\
		&=\frac{1}{K_{S,\lambda}^{I,h_0}(z_0)}.
	\end{split}
	\end{equation}
	Following from  equality \eqref{eq:0811a} and equality \eqref{eq:0811b}, we have  $K_{S,\lambda}^{I,h_0}(z_0)=\prod_{1\le j\le n} K_{\partial D_j,\lambda_j}^{I_{\tilde\beta_j,z_j},h_j}(z_j)$.
\end{proof}

\section{Useful propositions}\label{sec:p}

In this section, we give some useful propositions, which will be used in the proofs of the main results.

We recall the following coarea formula.
\begin{Lemma}[see \cite{federer}]
	\label{l:coarea}Suppose that $\Omega$ is an open set in $\mathbb{R}^n$ and $u\in C^1(\Omega)$. Then for any $g\in L^1(\Omega)$, 
	$$\int_{\Omega}g(x)|\bigtriangledown u(x)|dx=\int_{\mathbb{R}}\left(\int_{u^{-1}(t)}g(x)dH_{n-1}(x)\right)dt,$$
	where $H_{n-1}$ is the $(n-1)$-dimensional Hausdorff measure.
\end{Lemma}

Let $z_0=(z_1,\ldots,z_n)\in M,$ and let  
$$\psi(w_1,\ldots,w_n)=\max_{1\le j\le n}\{2p_jG_{D_j}(w_j,z_j)\}$$
 be a plurisubharmonic function on $M$, where $G_{D_j}(\cdot,z_j)$ is the Green function on $D_j$ and $p_j>0$ for any $1\le j\le n$.

Let $\hat\rho$ be a Lebesgue measurable function on $\overline M$, which satisfies that $\inf_{\overline{M}}\hat\rho>0$ and $\hat\rho(w_j,\hat w_j)\leq \liminf_{w\rightarrow w_j}\hat\rho(w,\hat w_j)$ for any $(w_j,\hat w_j)\in \partial D_j\times M_j\subset \partial M$ and any $1\le j\le n$, where $M_j=\prod_{l\not=j}D_l$. Let $\rho$ be a Lebesgue measurable function on $\partial M$ such that
$$\rho(w_1,\ldots,w_n):=\frac{1}{p_j}\left(\frac{\partial G_{D_j}(w_j,z_j)}{\partial v_{w_j}}\right)^{-1}\hat\rho$$
on $\partial D_j\times {M_j}$ for any $1\le j\le n$, thus we have $\inf_{\partial M}\rho>0$.

\begin{Proposition}
	\label{p:b6}Let $g$ be a holomorphic function on $M$. Assume that 
	$$\liminf_{r\rightarrow1-0}\frac{\int_{\{z\in M:\psi(z)\ge\log r\}}|g(z)|^2\hat\rho}{1-r}<+\infty,$$
	then there is $f\in H^2_{\rho}(M,\partial M)$ such that $f^*=g$ and 
	$$\|f\|_{\partial M,\rho}^2\le \frac{1}{\pi} \liminf_{r\rightarrow1-0}\frac{\int_{\{z\in M:\psi(z)\ge\log r\}}|g(z)|^2\hat\rho}{1-r}.$$
\end{Proposition}

\begin{proof}
	Denote that 
	$$D_{1,r}:=\{z\in D_1:G_{D_1}(z,z_1)<\log r\},$$
	where $G_{D_1}(\cdot,z_1)$ is the Green function on $D_1$ and $r\in(0,1)$. It is well-known that $G_{D_1}(\cdot,z_1)-\log r$ is the Green function on $D_{1,r}$.
	By the analyticity of the boundary of $D_1$, we have $G_{D_1}(z,w)$ has an analytic extension on $U\times V\backslash\{z=w\}$ and $\frac{\partial G_{D_1}(z,z_1)}{\partial v_z}$ is positive and smooth on $\partial D_1$, where  $\partial/\partial v_z$ denotes the derivative along the outer normal unit vector $v_z$, $U$ is a neighborhood of $\overline{D_1}$ and $V\Subset D_1$. Then there exist $r_0\in(0,1)$ and $C_1>0$ such that  $\frac{1}{C_1}\le|\bigtriangledown G_{D_1}(\cdot,z_1)|\le C_1$ on $\{z\in D_1:G_D(z,z_1)>\log r_0\}$, which implies 
	\begin{equation}
		\label{eq:0731a}\frac{1}{C_1}\le\frac{\partial G_{D_1}(z,z_1)}{\partial v_z}\le C_1
	\end{equation}
	holds on $\{z\in D:G_{D_1}(z,z_1)>\log r_0\}$ (by using Lemma \ref{l:v}).
Denote that 
$$U_{r,\hat w_1}(w_1):=\frac{1}{2\pi}\int_{\partial D_{1,r}}|g(z,\hat w_1)|^2\frac{\partial G_{D_{1,r}}(z,w_1)}{\partial v_z} |dz|$$
for any $\hat w_1\in M_1,$	
	where $r\in (r_0,1)$ and $G_{D_{1,r}}(\cdot,\cdot)$ is the Green function on $D_{1,r}$. Then we know that $U_{r,\hat w_1}(w_1)\in C(\overline{D_{1,r}})$, $U_{r,\hat w_1}(w_1)|_{\partial D_{1,r}}=|g(z,\hat w_1)|^2$ and $U_{r,\hat w_1}(w_1)$ is harmonic on $D_{1,r}$ for any $\hat w_1\in M_1$.
	As $|g(z,\hat w_1)|^2$ is subharmonic on $D_1$ for any $\hat w_1\in M_1$, then $U_{r,\hat w_1}$ is increasing with respect to $r$. 
	
	As $g(w_1,\cdot)$ is a holomorphic function on $M_1$ for any $w_1\in D_1$, for any compact subset $K$ of $M_1$, there is a positive constant $C_K$ such that 
	\begin{equation}
		\label{eq:0731b}\sup_{\hat w_1\in K}\int_{\partial D_{1,r}}|g(w_1,\hat w_1)|^2|dw_1|\le C_K\int_{M_1}\left(\int_{\partial D_{1,r}}|g(w_1,\hat w_1)|^2|dw_1|\right)d\mu_1(\hat w_1).
	\end{equation}
Note that $G_{D_{1,r}}(\cdot,z_1)=G_{D_1}(\cdot,z_1)-\log r$ on $D_{1,r}$, then	it follows from  inequality \eqref{eq:0731a} and inequality \eqref{eq:0731b} that
	\begin{equation}\label{eq:0731e}
	\begin{split}
				U_{r,\hat w_1}(z_1)&=\frac{1}{2\pi}\int_{\partial D_{1,r}}|g(z,\hat w_1)|^2\frac{\partial G_{D_{1}}(z,z_1)}{\partial v_z} |dz|\\
				&\le \frac{C_1C_K}{2\pi}\int_{M_1}\left(\int_{\partial D_{1,r}}|g(z,\hat w_1)|^2|dz|\right)d\mu_1(\hat w_1)\\
				&\le \frac{C_1^2C_K}{2\pi}\int_{M_1}\left(\int_{\partial D_{1,r}}|g(z,\hat w_1)|^2\frac{\partial G_{D_{1}}(z,z_1)}{\partial v_z}|dz|\right)d\mu_1(\hat w_1)\\
				&=\frac{C_1^2C_K}{2\pi}\int_{M_1}U_{r,\hat w_1}(z_1)d\mu_1(\hat w_1).
	\end{split}
	\end{equation}
	for any $r\in (r_0,1)$ and any $\hat w_1\in K$.
	It follows from Lemma \ref{l:coarea} that  
	\begin{equation}
		\label{eq:0731c}
		\begin{split}
					&\int_{\{z=(w_1,\hat w_1)\in M:G_{D_1}(w_1,z_1)\ge\frac{\log r}{2p_j}\}}|g(z)|^2\hat\rho(z)\\
					=&\int_{r^{\frac{1}{2p_j}}}^1\left(\int_{M_1}\left(\int_{\partial D_{1,s} }\frac{|g(w_1,\hat w_1)|^2\hat\rho(w_1,\hat w_1)}{|\bigtriangledown e^{G_{D_1}(w_1,z_1)}|}|dw_1|\right)d\mu_1(\hat w_1)\right) ds
		\end{split}
	\end{equation}
	for any $r\in (r_0^{2p_j},1)$. Denote that $$C_2=\inf_{(w_1,\hat w_1)\in\{G_{D_1}(w_1,z_1)\ge\log r_0\}\times M_1 }\frac{\hat\rho(w_1,\hat w_1)}{|\bigtriangledown e^{G_{D_1}(w_1,z_1)}|}>0.$$
Note that $U_{r,\hat w_1}(z_1)$ is increasing with respect $r$ for any $\hat w_1\in M_1$, then it follows from Lemma \ref{l:v} and inequality \eqref{eq:0731e} that
\begin{equation}
	\label{eq:0731d}
	\begin{split}
			U_{r,\hat w_1}(z_1)\le &\frac{C_1^2C_K}{2\pi}\int_{M_1}U_{r,\hat w_1}(z_1)d\mu_1(\hat w_1)\\
			\le &\frac{C_1^2C_K}{2\pi}\int_{M_1}U_{\tilde r,\hat w_1}(z_1)d\mu_1(\hat w_1)\\
			\le&\frac{C_1^3C_K}{2\pi}\int_{M_1}\left(\int_{\partial D_{1,\tilde r}}|g(z,\hat w_1)|^2|dz|\right)d\mu_1(\hat w_1)\\
			\le &\frac{C_1^3C_K}{2\pi C_2}\int_{M_1}\left(\int_{\partial D_{1,\tilde r}}|g(z,\hat w_1)|^2\frac{\hat\rho(z,\hat w_1)}{|\bigtriangledown e^{G_{D_1}(z,z_1)}|}|dz|\right)d\mu_1(\hat w_1)
	\end{split}
\end{equation}
holds for any $r_0<r<\tilde r<1$. 
	As $$\cup_{j=1}^n\left\{(w_j,\hat w_j)\in D_j\times M_j=M:G_{D_j}(w_j,z_j)\ge\frac{\log r}{2p_j} \right\}=\{z\in M:\psi(z)\ge\log r\},$$ following from $\liminf_{r\rightarrow1-0}\frac{\int_{\{z\in M:\psi(z)\ge\log r\}}|g(z)|^2\hat\rho}{1-r}<+\infty,$ equality \eqref{eq:0731c} and inequality \eqref{eq:0731d}, we obtain that
	\begin{equation}\nonumber
		\begin{split}
			&U_{r,\hat w_1}(z_1)\\
			\le &\frac{C_1^3C_K}{2\pi C_2}\liminf_{\tilde r\rightarrow1-0}\frac{\int_{\tilde r^{\frac{1}{2p_j}}}^1\left(\int_{M_1}\left(\int_{\partial D_{1,s} }\frac{|g(w_1,\hat w_1)|^2\hat\rho(w_1,\hat w_1)}{|\bigtriangledown e^{G_{D_1}(w_1,z_1)}|}|dw_1|\right)d\mu_1(\hat w_1)\right)}{1-\tilde r^{\frac{1}{2p_j}}}\\
			=&\frac{C_1^3C_K}{2\pi C_2}\liminf_{\tilde r\rightarrow1-0}\frac{\int_{\left\{z=(w_1,\hat w_1)\in M:G_{D_1}(w_1,z_1)\ge\frac{\log \tilde r}{2p_j}\right\}}|g(z)|^2\hat\rho(z)}{1-\tilde r^{\frac{1}{2p_j}}}\\
			\le&\frac{C_1^3C_K}{2\pi C_2}\liminf_{\tilde r\rightarrow1-0}\frac{\int_{\{z\in M:\psi(z)\ge\log \tilde r\}}|g(z)|^2\hat\rho}{1-\tilde r^{\frac{1}{2p_j}}}\\
			\le&C_3,
		\end{split}
	\end{equation} 
	where $C_3$ is a constant, which is independent of $r$, which shows that 
	$$\lim_{r\rightarrow1-0}U_{r,\hat w_1}(z_1)<+\infty.$$ Note that $U_{r,\hat w_1}$ is increasing with respect to $r$. By Harnack's principle (see \cite{ahlfors}), the sequence $U_{r,\hat w_1}$ converges to a harmonic function $U_{\hat w_1}$ on $D_1$ when $r\rightarrow1-0$ for any $\hat w_1\in M_1$, which satisfies that $|g(\cdot,\hat w_1)|^2\le U_{\hat w_1}$. Thus, we have 
	$$g(\cdot,\hat w_1)\in H^2(D_1)$$
	 for any $\hat w_1\in M_1$. By similar discussion, we know that $g(\cdot,\hat w_j)\in H^2(D_j)$ holds for any $\hat w_j\in M_j$ and any $1\le j\le n$.

In the following part, we will prove that 
$$\sum_{1\le j\le n}\int_{M_j}\int_{\partial D_j}|\gamma_j(g(\cdot,\hat w_j))|^2\rho|dw_j|d\mu_j(\hat w_j)\le 2\liminf_{r\rightarrow1-0}\frac{\int_{\{z\in M:\psi(z)\ge\log r\}}|g(z)|^2\hat\rho}{1-r},$$
which implies that
there is $f\in H^2_{\rho}(M,\partial M)$ such that $f^*=g$ and 
	$$\|f\|_{\partial M,\rho}^2\le \frac{1}{\pi}
	\liminf_{r\rightarrow1-0}\frac{\int_{\{z\in M:\psi(z)\ge\log r\}}|g(z)|^2\hat\rho}{1-r},$$
	where $\rho(w_1,\ldots,w_n)=\frac{1}{p_j}\left(\frac{\partial G_{D_j}(w_j,z_j)}{\partial v_{w_j}}\right)^{-1}\hat\rho$
on $\partial D_j\times \overline{M_j}$.	
	
Choose any compact subset $K_j$ of $M_j$ for any $1\le j\le n$, and denote that
$$\Omega_{j,r}:=\{z\in D_j:2p_jG_{D_j}(z,z_j)\geq \log r\}\times K_j\subset \{z\in M:\psi(z)\ge \log r\}$$
for any $1\le j\le n$. There exists $r_1\in(\max_{1\le j\le n}\{r_0^{2p_j}\},1)$, such that $\Omega_{j,r_1}\cap\Omega_{j',r_1}=\emptyset$ for any $j\not=j'$.
Note that $\gamma_1(g(\cdot,\hat w_1))$ denotes the nontangential boundary value of $g(\cdot,\hat w_1)$ a.e. on $\partial D_1$ for any $\hat w_1\in M_1$. As
 $$\hat\rho(w_1,\hat w_1)\leq \liminf_{w\rightarrow w_1}\hat\rho(w,\hat w_1)$$
  for any $(w_1,\hat w_1)\in \partial D_1\times M_1$,  
it follows from Fatou's lemma, Lemma \ref{l:coarea} and Lemma \ref{l:v} that
\begin{equation}
\nonumber \begin{split}
	&\int_{K_1}\int_{\partial D_1}|\gamma_1(g(\cdot,\hat w_1))|^2\rho|dw_1|d\mu_1(\hat w_1)\\
	=&\int_{K_1}\left(\frac{1}{ p_1}\int_{\partial D_1}\frac{|\gamma_1(g(\cdot,\hat w_1))|^2}{\frac{\partial G_{D_1}(w_1,z_1)}{\partial v_{w_1}}}\hat\rho|dw_1|\right)d\mu_1(\hat w_1)\\
	\le& \liminf_{r\rightarrow1-0}\frac{\int_{\frac{\log r}{2p_1}}^0\left(\int_{K_1}\left(\frac{1}{p_1}\int_{\{G_{D_1}(\cdot,z_1)=s\}}\frac{|g|^2\hat\rho}{|\bigtriangledown G_{D_1}(\cdot,z_1)|}|dw_1|\right)d\mu_1(\hat w_1)\right)ds}{-\frac{\log r}{2p_1}}\\
	=&2\liminf_{r\rightarrow1-0}\frac{\int_{\Omega_{1,r}}|g|^2\hat\rho}{\log r}.
\end{split}
\end{equation}	
By similar discussion, we have 
\begin{equation}
	\label{eq:220731a}
	\int_{K_j}\int_{\partial D_j}|\gamma_j(g(\cdot,\hat w_j))|^2\rho|dw_j|d\mu_j(\hat w_j)
	\le2\liminf_{r\rightarrow1-0}\frac{\int_{\Omega_{j,r}}|g|^2\hat\rho}{\log r}
\end{equation} 
for any $1\le j\le n$. As $\Omega_{j,r}\cap\Omega_{j',r}=\emptyset$ for any $j\not=j'$ and $r\in(r_1,1)$, following from the arbitrariness of $K_j$ and inequality \eqref{eq:220731a} that 
\begin{equation}\nonumber
	\begin{split}
		&\sum_{1\le j\le n}	\int_{M_j}\int_{\partial D_j}|\gamma_j(g(\cdot,\hat w_j))|^2\rho|dw_j|d\mu_j(\hat w_j)\\
		\le&2\liminf_{r\rightarrow1-0}\frac{\int_{\{z\in M:\psi(z)\ge\log r\}}|g|^2\hat\rho}{\log r}\\
		=&2\liminf_{r\rightarrow1-0}\frac{\int_{\{z\in M:\psi(z)\ge\log r\}}|g|^2\hat\rho}{1-r}\\
		<&+\infty.
	\end{split}
\end{equation}

Thus, Proposition \ref{p:b6} holds.
\end{proof}

Let $h_0$ be a holomorphic function on a neighborbood $V_0$ of $z_0$, and let $I$ be an ideal of $\mathcal{O}_{z_0}$ such that $\mathcal{I}(\psi)_{z_0}\subset I\not= \mathcal{O}_{z_0}$. Let $c$ be a positive function on $[0,+\infty)$,  which satisfies that $c(t)e^{-t}$ is decreasing on $[0,+\infty)$, $\lim_{t\rightarrow0+0}c(t)=c(0)=1$ and $\int_{0}^{+\infty}c(t)e^{-t}dt<+\infty$. Denote that $$\tilde\rho=c(-\psi)\hat\rho.$$

\begin{Proposition}\label{p:inequality}
	Assume that $-\log\hat\rho$ is plurisubharmonic on $M$. We have  
	$$K_{\partial M,\rho}^{I,h_0}(z_0)\geq\left(\int_{0}^{+\infty}c(t)e^{-t}dt\right)\pi B_{\tilde \rho}^{I,h_0}(z_0).$$
\end{Proposition}
\begin{proof}
	It suffices to prove the case $ B_{\tilde \rho}^{I,h_0}(z_0)>0$.  Denote \begin{equation*}
\begin{split}
\inf\left\{\int_{\{\psi<-t\}}|\tilde{f}|^{2}\tilde\rho:(\tilde{f}-h_0,z_0)\in I\,\&{\,}\tilde{f}\in \mathcal{O} (\{\psi<-t\})\right\},
\end{split}
\end{equation*}
by $G(t)$ for any $t\ge0$, then we have 
$$G(0)=\frac{1}{B^{I,h_0}_{\tilde\rho}(z_0)}.$$ 
Since $\mathcal{I}(\psi)_{z_0}\subset I$, it follows from Theorem \ref{thm:general_concave} that $G(h^{-1}(r))$ is concave with respect to $r\in(0,\int_0^{+\infty}c(t)e^{-t}dt)$, where $h(t)=\int_t^{+\infty}c(l)e^{-l}dl$.

 Lemma \ref{l:unique} tells us that there exists a holomorphic function $F_0$ on $M$ such that $(F_0-h_0,z_0)\in I$ and 
 $$G(0)=\int_{M}|F_0|^2\tilde\rho<+\infty.$$  
 As 
 $G(t)\le \int_{\{\psi<-t\}}|F_0|^2\tilde\rho$
 for any $t\ge 0$, $\hat \rho=\frac{\tilde \rho}{c(-\psi)}$, $c(t)e^{-t}$ is decreasing and $\lim_{t\rightarrow0+0}c(t)=c(0)=1$, then we have 
\begin{equation}
	\label{eq:0806b}\begin{split}
		\limsup_{r\rightarrow1-0}\frac{\int_{\{z\in M:\psi(z)\ge\log r\}}|F_0(z)|^2\hat\rho}{1-r}&=\limsup_{r\rightarrow1-0}\frac{\int_{\{z\in M:\psi(z)\ge\log r\}}|F_0(z)|^2\tilde\rho}{\int_0^{-\log r}c(l)e^{-l}dl}\\
		&\le \limsup_{r\rightarrow1-0}\frac{G(0)-G(-\log r)}{\int_0^{+\infty}c(l)e^{-l}dl-\int_{-\log r}^{+\infty}c(l)e^{-l}dl}\\
		&=\limsup_{t \rightarrow 0+0}\frac{G(0)-G(t)}{\int_0^{+\infty}c(l)e^{-l}dl-\int_t^{+\infty}c(l)e^{-l}dl}.
	\end{split}
\end{equation}
Proposition \ref{p:b6} shows that there is $f\in H^2_{\rho}(M,\partial M)$ such that $f^*=F_0$ and 
 \begin{equation}
 	\label{eq:0801c2}\|f\|^2_{\partial M,\rho}\le \frac{1}{\pi}\limsup_{r\rightarrow1-0}\frac{\int_{\{z\in M:\psi(z)\ge\log r\}}|F_0(z)|^2\hat\rho}{1-r}.
 \end{equation}
Note that $(f^*-h,z_0)\in I$, then combining the definition of $K^{I,h_0}_{\partial M,\rho}(z_0)$, concavity of $G(h^{-1}(r))$, inequality \eqref{eq:0806b} and inequality \eqref{eq:0801c2}, we have
\begin{equation}
	\label{eq:0801d2}
	\begin{split}
		\left(\int_0^{+\infty}c(t)e^{-t}dt\right)B^{I,h_0}_{\tilde\rho}(z_0)&=\frac{\int_0^{+\infty}c(t)e^{-t}dt}{G(0)}\\
		&\le\liminf_{t\rightarrow0+0}\frac{\int_0^{+\infty}c(l)e^{-l}dl-\int_t^{+\infty}c(l)e^{-l}dl}{G(0)-G(t)} \\
		&\leq\liminf_{r\rightarrow 1-0}\frac{1-r}{\int_{\{z\in M:\psi(z)\ge\log r\}}|F_0|^2\hat\rho}\\
		&\le \frac{1}{\pi\|f\|^2_{\partial M,\rho}} \\
		&\le \frac{1}{\pi} K^{I,h_0}_{\partial M,\rho}(z_0).
	\end{split}
\end{equation}
Thus, we have proved the Proposition \ref{p:inequality}.
\end{proof}

Let $\hat\rho=\prod_{1\le j\le n}\rho_j$, where $\rho_j$ is a Lebesgue measurable function on $\overline{D_j}$ such that $\inf_{\overline{D_j}}\rho_j>0$ and $\rho_j(p)\le\liminf_{z\rightarrow p}\rho_j(z)$ for any $p\in \partial D_j$.
 Let 
$$\rho(w_1,\ldots,w_n)=\frac{1}{p_j}\left(\frac{\partial G_{D_j}(w_j,z_j)}{\partial v_{w_j}}\right)^{-1}\hat\rho$$
on $\partial D_j\times {M_j}$, and let
$$\lambda (w_1,\ldots,w_n)=\hat\rho\prod_{1\le j\le n}\left(\frac{\partial G_{D_j}(w_j,z_j)}{\partial v_{w_j}}\right)^{-1}$$
on $S=\prod_{1\le j\le n}\partial D_j$. 
\begin{Proposition}\label{p:inequality2}Assume that $-\log\rho_j$ is subharmonic on $D_j$ for any $1\le j\le n$, then we have
		$$K_{S,\lambda}(z_0)\geq\left(\sum_{1\le j\le n}\frac{1}{p_j}\right)\pi^{n-1} K_{\partial M,\rho}(z_0).$$
\end{Proposition}
	\begin{proof}
	It follows from Proposition \ref{p:8} that 
	\begin{equation}
		\label{eq:0805a}K_{S,\lambda}(z_0)=\prod_{1\le j\le n}K_{\partial D_j,\lambda_j}(z_j)=K_{\partial D_1,\lambda_1}(z_1)\times K_{S_1,\hat \lambda_1}(\hat z_1),
	\end{equation}
	where $\lambda_j(w_j)=\left(\frac{\partial G_{D_j}(w_j,z_j)}{\partial v_{w_j}}\right)^{-1}\rho_j$ on $\partial D_j$, $S_1=\prod_{2\le j\le n}\partial D_j$, $\hat\lambda_1=\prod_{2\le j\le n}\lambda_j$ and $\hat z_1=(z_2,\ldots,z_n)$. It follows from Lemma \ref{l:Bergman-prod} that 
	\begin{equation}
		\label{eq:0805b}B_{M_1,\hat\rho_1}(\hat z_1)=\prod_{2\le j\le n}B_{D_j, \rho_j}(z_j),
	\end{equation}
	where $M_1=\prod_{2\le j\le n}D_j$ and $\hat\rho_1=\prod_{2\le j\le n}\rho_j$.
	Proposition \ref{p:inequality} shows that $\pi B_{\rho_j}(z_j)\le K_{\partial D_j,\lambda_j}(z_j)$ for any $1\le j\le n$. 
	By equality \eqref{eq:0805a} and \eqref{eq:0805b}, we know that 
	\begin{equation}
		\label{eq:1017a}K_{S,\lambda}(z_0)\geq\pi^{n-1} K_{\partial D_1,\lambda_1}(z_1)\times B_{M_1,\hat\rho_1}(\hat z_1).
	\end{equation}	
Assume that $K_{\partial M,\rho}(z_0)>0$ without loss of generality, then Lemma \ref{l:b7} shows that there is $f\in H^2_{\rho}(M,\partial M)$ such that $f^*(z_0)=1$ and 
	$$\|f\|^2_{\partial M,\rho}=\frac{1}{K_{\partial M,\rho}(z_0)}<+\infty.$$ 
	Since $f\in H^2_{p_1\rho}(\partial D_1\times M_1)$ and $f^*(z_0)=1$, it follows from Lemma \ref{l:prod-d1xm1} that $\{g\in H^2_{\lambda_1}(D_1,\partial D_1):g^*(z_1)=1\}\not=\emptyset$ and $\{g\in A^2(M_1,\hat\rho_1):g(\hat z_1)=1\}\not=\emptyset$. Thus, there are $g_1\in H^2_{\lambda_1}(D_1,\partial D_1)$ and $g_2\in \{g\in A^2(M_1,\hat\rho_1):g(\hat z_1)=1\}$ such that $g_1^*(z_1)=1$, $g_2(\hat z_1)=1$, 
	$$\|g_1\|_{\partial D_1,\lambda_1}^2=\frac{1}{K_{\partial D_1,\lambda_1}(z_1)}$$
	and 
	$$\|g_2\|^2_{M_1,\hat\rho_1}=\frac{1}{ B_{M_1,\hat\rho_1}(\hat z_1)}.$$
	Following from Lemma \ref{l:b9}, we obtain that 
	$$\ll g_1g_2,g_1g_2-f \gg_{\partial D_1,\times M_1,p_1\rho}=0,$$
	which implies that 
\begin{equation}
	\label{eq:0810a}\frac{1}{K_{\partial D_1,\lambda_1}(z_1) B_{M_1,\hat\rho_1}(\hat z_1)}=\|g_1g_2\|^2_{\partial D_1\times M_1,p_1\rho}\le \|f\|^2_{\partial D_1\times M_1,p_1\rho}.
\end{equation}
By inequality \eqref{eq:1017a} and inequality \eqref{eq:0810a}, we have
\begin{displaymath}
	\frac{1}{K_{S,\lambda}(z_0)}\le\frac{1}{\pi^{n-1}K_{\partial D_1,\lambda_1}(z_1) B_{M_1,\hat\rho_1}(\hat z_1)}\le \frac{1}{\pi^{n-1}}\|f\|^2_{\partial D_1\times M_1,p_1\rho}.
\end{displaymath}
By similar discussion, we obtain that 
\begin{equation}
	\label{eq:0810b}	\frac{1}{K_{S,\lambda}(z_0)}\le \frac{1}{\pi^{n-1}K_{\partial D_j,\lambda_j}(z_j) B_{M_j,\hat\rho_j}(\hat z_j)}\le  \frac{1}{\pi^{n-1}}\|f\|^2_{\partial D_j\times M_j,p_j\rho}
\end{equation}
holds for any $1\le j\le n$.
By the definition of $\|\cdot \|_{\partial M,\rho}$ and inequality \eqref{eq:0810b}, we have 
\begin{equation}
	\label{eq:0810c}\left(\sum_{1\le j\le n}\frac{1}{p_j}\right)\frac{1}{K_{S,\lambda}(z_0)}\le \frac{1}{\pi^{n-1}}\|f\|^2_{\partial M,\rho}=\frac{1}{\pi^{n-1}K_{\partial M,\rho}(z_0)}.
\end{equation}

Thus, Proposition \ref{p:inequality2} holds.
	\end{proof}

Let $h_j$ be a holomorphic function on a neighborhood of  $z_j$, and let $$h_0=\prod_{1\le j\le n}h_j.$$
Denote that $\beta_j:=ord_{z_j}(h_j)$ and $\beta:=(\beta_1,\cdots,\beta_n)$. Let $\tilde\beta=(\tilde\beta_1,\ldots,\tilde\beta_n)\in\mathbb{Z}_{\ge0}^n$ satisfy $\tilde\beta_j\ge\beta_j$ for any $1\le j\le n$.
Let 
$$I=\left\{(g,z_0)\in\mathcal{O}_{z_0}:g=\sum_{\alpha\in\mathbb{Z}_{\ge0}^n}b_{\alpha}(w-z_0)^{\alpha}\text{ near  }z_0\text{ s.t. $b_{\alpha}=0$ for $\alpha\in L_{\tilde\beta}$}\right\},$$
where $L_{\tilde\beta}=\{\alpha\in\mathbb{Z}_{\ge0}^n:\alpha_j\le\tilde\beta_j$ for any $1\le j\le n\}$.
It is clear that $(h_0,z_0)\not\in I$.
Assume that $\sum_{1\le j\le n}\frac{\tilde\beta_j+1}{p_j}\le1$.

\begin{Proposition}
	\label{p:inequality3}Assume that $-\log\rho_j$ is subharmonic on $D_j$ for any $1\le j\le n$, then we have
		$$K_{S,\lambda}^{I,h_0}(z_0)\prod_{1\le j\le n}(\tilde\beta_j+1)\geq\left(\sum_{1\le j\le n}\frac{\tilde\beta_j+1}{p_j}\right)\pi^{n-1} K^{I,h_0}_{\partial M,\rho}(z_0).$$
\end{Proposition}

	\begin{proof}
Denote that 
 \begin{displaymath}
	\begin{split}
		I'=\Bigg\{(g,\hat z_1)\in\mathcal{O}_{M_1, \hat z_1}:g=\sum_{\alpha=(\alpha_2,\ldots,\alpha_n)\in\mathbb{Z}_{\ge0}^{n-1}}b_{\alpha}&\prod_{2\le j\le n}(w_j-z_j)^{\alpha_j}\text{ near  }\hat z_1\\
		&\text{s.t. $b_{\alpha}=0$ for any $\alpha\in L_1$}\Bigg\},
	\end{split}
\end{displaymath}
where $L_1=\{\alpha=(\alpha_2,\ldots,\alpha_n)\in\mathbb{Z}_{\ge0}^{n-1}:\alpha_j\le\tilde\beta_j$ for any $2\le j\le n\}$,
then it follows from Proposition \ref{p:9} that 
	\begin{equation}
		\label{eq:0812a}K_{S,\lambda}^{I,h_0}(z_0)=\prod_{1\le j\le n}K^{I_{\tilde\beta_j,z_j},h_j}_{\partial D_j,\lambda_j}(z_j)=K^{I_{\tilde\beta_1,z_1},h_1}_{\partial D_1,\lambda_1}(z_1)\times K^{I',\prod_{2\le j\le n}h_j}_{S_1,\hat \lambda_1}(\hat z_1),
	\end{equation}
	where $\lambda_j(w_j)=\left(\frac{\partial G_{D_j}(w_j,z_j)}{\partial v_{w_j}}\right)^{-1}\rho_j$ on $\partial D_j$, $S_1=\prod_{2\le j\le n}\partial D_j$, $\hat\lambda_1=\prod_{2\le j\le n}\lambda_j$ and $\hat z_1=(z_2,\ldots,z_n)$. 
	 It follows from Lemma \ref{l:Bergman-prod2} that 
	\begin{equation}
		\label{eq:0812b}B^{I',\prod_{2\le j\le n}h_j}_{M_1,\hat\rho_1}(\hat z_1)=\prod_{2\le j\le n}B^{I_{\tilde\beta_j,z_j},h_j}_{D_j, \rho_j}(z_j),
	\end{equation}
	where $M_1=\prod_{2\le j\le n}D_j$ and $\hat\rho_1=\prod_{2\le j\le n}\rho_j$.
	Proposition \ref{p:inequality} shows that 
	$$\pi B_{\rho_j}^{I_{\tilde\beta_j,z_j},h_j}(z_j)\le(\tilde \beta_j+1) K^{I_{\tilde\beta_j,z_j},h_j}_{\partial D_j,\lambda_j}(z_j)$$
	 for any $1\le j\le n$. 
	By equality \eqref{eq:0812a} and \eqref{eq:0812b}, we know that 
	\begin{equation}
		\label{eq:0812c}K_{S,\lambda}^{I,h_0}(z_0)\geq\frac{\pi^{n-1}}{\prod_{2\le j\le n}(\tilde\beta_j+1)} K^{I_{\tilde\beta_1,z_1},h_1}_{\partial D_1,\lambda_1}(z_1)\times B^{I',\prod_{2\le j\le n}h_j}_{M_1,\hat\rho_1}(\hat z_1).
	\end{equation}

	Assume that $K^{I,h_0}_{\partial M,\rho}(z_0)>0$ without loss of generality, then Lemma \ref{l:b7} shows that there is $f\in H^2_{\rho}(M,\partial M)$ such that $(f^*-h_0,z_0)\in I$ and 
	$$\|f\|^2_{\partial M,\rho}=\frac{1}{K^{I,h_0}_{\partial M,\rho}(z_0)}<+\infty.$$ 
	Let $\{e_m\}_{m\in\mathbb{Z}_{\ge0}}$ be a complete orthonormal basis of $H^2(\partial D_1,\lambda_1)$ such that $ord_{z_1}(e_m)=m$ for any $m\ge0$. 
	Since $f\in H^2_{p_1\rho}(\partial D_1\times M_1)$, it follows from Lemma \ref{l:prod-d1xm1} that  $$f=\sum_{m\ge0}e_{m}f_{m}\text{ (convergence under the norm $\|\cdot\|_{\partial M,\rho}$)}$$
	and 
	$$f^*=\sum_{m\ge0}e_{m}^*f_{m} \text{ (uniform convergence on any  compact subset of $M$)},$$ 
	where $f_{m}\in A^2(M_1,\hat\rho_1)$. Note that $(f^*-\prod_{1\le j\le n}h_j,z_0)\in I$, then there is a constant $c_m$ such that 
	$$(f_{m}-c_m\prod_{2\le j\le n}h_j,\hat z_1)\in I'$$
	for $0\le m\le \tilde\beta_1$.
	Note that $\sum_{0\le m\le \tilde\beta_1}|c_m|>0$, then there is $\tilde f_1\in A^2(M_1,\hat\rho_1)$ such that $(\tilde f_{1}-\prod_{2\le j\le n}h_j,\hat z_1)\in I'$ and 
	\begin{displaymath}
\begin{split}
\|\tilde f_{1}\|^2_{M_1,\hat\rho_1}&=\inf\left\{\|f\|^2_{M_1,\hat\rho_1}:f\in A^2(M_1,\hat\rho_1)\,\&\,(f-\prod_{2\le j\le n}h_j,\hat z_1)\in I'\right\}\\
&=\frac{1}{B_{M_1,\hat\rho_1}^{I',\prod_{2\le j\le n}h_j}(\hat z_1)}.
\end{split}
	\end{displaymath}
	Thus, we have $(f-\sum_{0\le m\le \tilde\beta_1}c_me_{m}^*\tilde f_1,z_0)\in I$, $(\sum_{0\le m\le \tilde\beta_1}c_me_{m}^*-h_1,z_1)\in I_{\tilde\beta_1,z_1}$ and 
	\begin{equation}
		\label{eq:0812d}
		\begin{split}
			\|f\|^2_{\partial D_1\times M_1,p_1\rho}&\ge\|\sum_{0\le m\le \tilde\beta_1}e_{m}f_{m}\|^2_{\partial D_1,\times M_1,p_1\rho}\\
			&=\sum_{0\le m\le \tilde\beta_1}\|f_{m}\|_{M_1,\hat\rho_1}^2\\
			&\ge\sum_{0\le m\le \tilde\beta_1}|c_m|^2\|\tilde f_{1}\|_{M_1,\hat\rho_1}^2\\
			&\ge\frac{1}{K_{\partial D_1,\lambda_1}^{I_{\tilde\beta_1,z_1},h_1}B_{M_1,\hat\rho_1}^{I',\prod_{2\le j\le n}h_j}(\hat z_1)}.
		\end{split}
	\end{equation}

	By inequality \eqref{eq:0812c} and inequality \eqref{eq:0812d}, we have
	\begin{displaymath}
		\begin{split}
		\frac{1}{K_{S,\lambda}^{I,h_0}(z_0)}&\le\frac{\prod_{2\le j\le n}(\tilde\beta_j+1)}{\pi^{n-1} K^{I_{\tilde\beta_1,z_1},h_1}_{\partial D_1,\lambda_1}(z_1)\times B^{I',\prod_{2\le j\le n}h_j}_{M_1,\hat\rho_1}(\hat z_1)}\\
		&\le \frac{\prod_{2\le j\le n}(\tilde\beta_j+1)}{\pi^{n-1}}\|f\|^2_{\partial D_1\times M_1,p_1\rho}
		\end{split}
	\end{displaymath}
	By similar discussion, we obtain that 
	\begin{equation}
		\label{eq:0812e}	
			\begin{split}
			\frac{1}{K_{S,\lambda}^{I,h_0}(z_0)}\le \frac{\prod_{1\le l\le n}(\tilde\beta_l+1)}{\pi^{n-1}(\tilde\beta_j+1)}\|f\|^2_{\partial D_j\times M_j,p_j\rho}
		\end{split}
	\end{equation}
	holds for any $1\le j\le n$.
	By the definition of $\|\cdot \|_{\partial M,\rho}$ and inequality \eqref{eq:0812e}, we have 
	\begin{equation}
		\label{eq:0812f}
		\begin{split}
		\left(\sum_{1\le j\le n}\frac{\tilde\beta_j+1}{p_j}\right)\frac{1}{K^{I,h_0}_{S,\lambda}(z_0)}&\le \frac{\prod_{1\le j\le n}(\tilde\beta_j+1)}{\pi^{n-1}}\|f\|^2_{\partial M,\rho}\\
		&=\frac{\prod_{1\le j\le n}(\tilde\beta_j+1)}{\pi^{n-1}}\frac{1}{K_{\partial M,\rho}^{I,h_0}(z_0)}.
		\end{split}
	\end{equation}

	Thus, Proposition \ref{p:inequality3} holds.
\end{proof}

\section{Proofs of Theorem \ref{thm:main1-1} and Remark \ref{r:product 1-1}}\label{sec:proof1}

In this section, we prove Theorem \ref{thm:main1-1} and Remark \ref{r:product 1-1}.

\subsection{Proof of Theorem \ref{thm:main1-1}}
\

Let $\hat\rho=\prod_{1\le j\le n}e^{-\varphi_j}$, $h_0=1$ and $I$ be the maximal ideal of $\mathcal{O}_{M,z_0}$, then the inequality part of Theorem \ref{thm:main1-1} follows from Proposition \ref{p:inequality}.

We prove the characterization in Theorem \ref{thm:main1-1} in two steps: Firstly,  we prove the necessity of the characterization; Secondly, we prove the sufficiency of the characterization.

\

\emph{Step 1:} 
 In this step, we will prove the necessity of the characterization.  
Assume that the equality
\begin{equation}
	\label{eq:20802a}K_{\partial M,\rho}(z_0)=\left(\int_0^{+\infty}c(t)e^{-t}dt\right)\pi B_{\tilde\rho}(z_0)
\end{equation}
holds. 
 
Denote \begin{equation*}
\begin{split}
\inf\left\{\int_{\{\psi<-t\}}|\tilde{f}|^{2}\tilde\rho:\tilde{f}(z_0)=1\,\&{\,}\tilde{f}\in \mathcal{O} (\{\psi<-t\})\right\},
\end{split}
\end{equation*}
by $G(t)$ for any $t\ge0$, then we have 
$$G(0)=\frac{1}{B_{\tilde\rho}(z_0)}.$$ 
Note that $M$ is a bounded domain in $\mathbb{C}^n$, then it follows from Lemma \ref{l:G equal to 0} that $G(t)\in(0,\infty)$ for any $t\ge0$.
By Lemma \ref{l:psi} and $\sum_{1\le j\le n}\frac{1}{p_j}\le 1$, we know that $\mathcal{I}(\psi)_{z_0}\subset I_{z_0}$, where $I_{z_0}$ is the maximal ideal of $\mathcal{O}_{z_0}$. Then it follows from Theorem \ref{thm:general_concave} that $G(h^{-1}(r))$ is concave with respect to $r\in(0,\int_0^{+\infty}c(t)e^{-t}dt)$, where $h(t)=\int_t^{+\infty}c(l)e^{-l}dl$.  Let $\hat\rho=\prod_{1\le j\le n}e^{-\varphi_j}$, $h_0=1$ and $I=I_{z_0}$, then  inequality \eqref{eq:0801d2} shows that 
\begin{equation}
	\label{eq:0801d}
	\begin{split}
		\left(\int_0^{+\infty}c(t)e^{-t}dt\right)B_{\tilde\rho}(z_0)&=\frac{\int_0^{+\infty}c(t)e^{-t}dt}{G(0)}\\
		&\le\liminf_{t\rightarrow0+0}\frac{\int_0^{+\infty}c(l)e^{-l}dl-\int_t^{+\infty}c(l)e^{-l}dl}{G(0)-G(t)}\\
		&\le \frac{|f^*(z_0)|^2}{\pi\|f\|^2_{\partial M,\rho}} \\
		&\le \frac{1}{\pi} K_{\partial M,\rho}(z_0),
	\end{split}
\end{equation}
where $f\in H^2_{\rho}(M,\partial M)$ such that $f^*=F_0$, and $F_0$ is a holomorphic function on $M$ such that $F(z_0)=1$ and $\int_M|F_0|^2\tilde\rho=G(0)$.
It follows from equality \eqref{eq:20802a} that inequality \eqref{eq:0801d} becomes an equality, which implies  that 
 $$\frac{\int_0^{+\infty}c(t)e^{-t}dt}{G(0)}=\liminf_{t\rightarrow0+0}\frac{\int_0^{+\infty}c(l)e^{-l}dl-\int_t^{+\infty}c(l)e^{-l}dl}{G(0)-G(t)}.$$
Following from the concavity of $G(h^{-1}(r))$, we obtain that $G(h^{-1}(r))$ is linear with respect to $r\in(0,\int_0^{+\infty}c(t)e^{-t}dt)$. It follows from Corollary \ref{c:linear} that 
\begin{equation}
	\label{eq:0801f}G(t)=\int_{\{\psi<-t\}}|F_0|^2\tilde\rho
\end{equation}
holds for any $t\ge0$.

 Denote 
\begin{equation*}
\begin{split}
\inf\left\{\int_{\{\psi<-t\}}|\tilde{f}|^{2}\tilde\rho:(\tilde{f}-F_0,z_0)\in\mathcal{I}(\psi)_{z_0}\,\&{\,}\tilde{f}\in \mathcal{O} (\{\psi<-t\})\right\},
\end{split}
\end{equation*}
by $G_1(t)$ for any $t\ge0$. By equality \eqref{eq:0801f}, we know $G_1(t)\le G(t)$ for any $t\ge0$. As $F_0(z_0)=1$ and $\mathcal{I}(\psi)_{z_0}\subset I_{z_0}$, where $I_{z_0}$ is the maximal ideal of $\mathcal{O}_{z_0}$, then we get that $(\tilde{f}-F_0,z_0)\in\mathcal{I}(\psi)_{z_0}$ deduces $\tilde f(z_0)=1$, which shows $G_1(t)\ge G(t)$ for any $t\geq0$. Thus, we have 
$$G(t)=G_1(t)$$
 for any $t\ge0$ and $G_1(h^{-1}(r))$ is linear with respect to $r\in (0,\int_0^{+\infty}c(t)e^{-t}dt)$.
Theorem \ref{thm:linear-2d} and Remark \ref{r:var=-infty} show that the following the following statements hold:
	
	$(1)$ $F_0=\sum_{\alpha\in E}d_{\alpha}(w-z_0)^{\alpha}+g_0$ on a neighborhood of $z_0$, where $g_0$ is a holomorphic function on a neighborhood of $z_0$ such that $(g_0,z_0)\in\mathcal{I}(\psi)_{z_0}$, $E=\left\{(\alpha_1,\ldots,\alpha_n)\in\mathbb{Z}_{\ge0}:\sum_{1\le j\le n}\frac{\alpha_j+1}{p_j}=1\right\}\not=\emptyset$ and $d_{\alpha}\in\mathbb{C}$ such that $\sum_{\alpha\in E}|d_{\alpha}|\not=0$.
	
	$(2)$ $\varphi_j=2\log|\tilde g_j|+2\tilde u_j$ on $D_j$ for any $1\le j\le n$, where $\tilde u_j$ is a harmonic function on $D_j$ and $\tilde g_j$ is a holomorphic function on $D_j$ such that $g_j(z_j)\not=0$;

    $(3)$ $\chi_{j,z_j}^{\alpha_j+1}=\chi_{j,-\tilde u_j}$ for any $1\le j\le n$ and $\alpha\in E$ satisfying $d_{\alpha}\not=0$, $\chi_{j,z_j}$ and $\chi_{j,-\tilde u_j}$ are the characters associated to functions $G_{D_j}(\cdot,z_j)$ and $-\tilde u_j$ respectively.

As $F_0(z_0)=1$, statement $(1)$ implies that 
$$\sum_{1\le j\le n}\frac{1}{p_j}=1$$ and statement $(3)$ implies that 
 $$\chi_{j,z_j}=\chi_{j,-\tilde u_j}$$
for any $1\le j\le n$.

 As $\varphi_j$ is continuous at $p$ for any $p\in \partial D_j$ and $D_j\Subset \mathbb{C}$ , then $\tilde g_j$ has at most a finite number of zeros on $D_j$.
Then there is $ g_j\in\mathcal{O}(\mathbb{C})$, which satisfies that $\frac{\tilde g_j}{ g_j}\in \mathcal{O}(D_j)$, $\frac{\tilde g_j}{ g_j}\in \mathcal{O}(D_j)$ has no zero point on $D_j$ and $ g_j$ has no zero point on $\partial D_j$. It is clear that $\chi_{j,\log|g_j|-\log|\tilde g_j|}=1$. Thus, we have 
$$\varphi_j=2\log|g_j|+2u_j\text{ and }\chi_{j,z_j}=\chi_{j,-u_j}$$
for any $1\le j\le n$, where $u_j=\tilde u_j+\log|\tilde g_j|-\log|g_j|$ is harmonic on $D_j$.

\

\emph{Step 2:} In this step, we prove the sufficiency of the characterization. Assume that the three statements $(1)-(3)$ hold.

Without loss of generality, assume that $g_j(z_j)=1$. Note that, for any $h\in\mathcal{O}(M)$ satisfying $\int_{M}\frac{|h|^2}{\prod_{1\le j\le n}|g_j|^2}c(-\psi)\prod_{1\le j\le n}e^{-2u_j}<+\infty$, we have
$\frac{h}{\prod_{1\le j\le n}g_j}\in\mathcal{O}(M)$.
 It is clear that 
\begin{displaymath}
	\begin{split}
		B_{\tilde\rho}(z_0)&=\sup_{f\in\mathcal{O}(M)}\frac{|f(z_0)|^2}{\int_{M}|f|^2\tilde\rho}\\
		&=\sup_{f\in\mathcal{O}(M)}\frac{|f(z_0)|^2}{\int_{M}\frac{|f|^2}{\prod_{1\le j\le n}|g_j|^2}c(-\psi)\prod_{1\le j\le n}e^{-2u_j}}\\
		&=\sup_{f\in\mathcal{O}(M)}\frac{|\tilde f(z_0)|^2}{\int_{M}|\tilde f|^2c(-\psi)\prod_{1\le j\le n}e^{-2u_j}}\\
		&=B_{\tilde\rho\prod_{1\le j\le n}|g_j|^2}(z_0).
	\end{split}
\end{displaymath}
It follows from Lemma \ref{l:divide g}, Lemma \ref{l:sup-b} and $g_j(z_j)=1$ that 
\begin{displaymath}
	\begin{split}
			K_{\partial M,\rho}(z_0)&=\sup_{f\in H^2_{\rho}(M,\partial M)}\frac{|f^*(z_0)|^2}{\|f\|^2_{\partial M,\rho}}\\
			&=\sup_{\tilde f\in H^2_{\rho\prod_{1\le j\le n}|g_j|^2}(M,\partial M)}\frac{|\tilde f^*(z_0)|^2}{\|\tilde f\|^2_{\partial M,\rho\prod_{1\le j\le n}|g_j|^2}}\\
			&=	K_{\partial M,\rho\prod_{1\le j\le n}|g_j|^2}(z_0).
				\end{split}
\end{displaymath}
Thus, it suffices to prove the case that $g_j\equiv1$ for any $1\le j\le n$. 

In the following part, we assume that $g_j\equiv1$ for any $1\le j\le n$.

 Denote \begin{equation*}
\begin{split}
\inf\left\{\int_{\{\psi<-t\}}|\tilde{f}|^{2}\tilde\rho:\tilde{f}(z_0)=1\,\&{\,}\tilde{f}\in \mathcal{O} (\{\psi<-t\})\right\}
\end{split}
\end{equation*}
by $G(t)$ for any $t\ge0$. Theorem \ref{thm:linear-2d} tells us  that $G(h^{-1}(r))$ is linear with respect to $r\in(0,\int_0^{+\infty}c(t)e^{-t}dt)$. Following from Corollary \ref{c:linear} and Remark \ref{r:1.1}, we get that 
\begin{equation}
	\label{eq:220801a}
	G(t)=\int_{\{\psi<-t\}}|F_0|^2\tilde\rho
	\end{equation}
holds for any $t\geq0$ and 
$$F_0=c_0\prod_{1\le j\le n} ((P_j)_*(f_{z_j}))'(P_j)_*(f_{u_j}),$$
 where $c_0$ is a constant such that $F_0(z_0)=1$, $P_j$ is the universal covering from unit disc $\Delta$ to $ D_j$, $f_{u_j}$ is a holomorphic function on $\Delta$ such that $|f_{u_j}|=P_j^*(e^{u_j})$, and $f_{z_j}$ is a holomorphic function on $\Delta$ such that $|f_{z_j}|=P_j^*(e^{G_{D_j}(\cdot,z_j)})$ for any $1\le j\le n$. It follows from $\lim_{t\rightarrow0+0}c(t)=1$ and equality \eqref{eq:220801a} that 
 \begin{equation}
 	\label{eq:220801b}
 	\begin{split}
 			\frac{\int_0^{+\infty}c(t)e^{-t}dt}{G(0)}&=\limsup_{r\rightarrow 1-0}\frac{\int_0^{-\log r}c(t)e^{-t}dt}{\int_{\{z\in M:\psi(z)\ge\log r\}}|F_0|^2\tilde\rho}\\
 			&=\limsup_{r\rightarrow 1-0}\frac{1-r}{\int_{\{z\in M:\psi(z)\ge\log r\}}|F_0|^2\hat\rho},
 	\end{split}
 \end{equation}
 where $\hat\rho=\prod_{1\le j\le n}e^{-\varphi_j}$.

Note that $u_j\in C(\overline{D_j})$ and $G_{D_j}(\cdot,z_j)$ can be extended to a harmonic function on a $U\backslash\{z_j\}$  for any $1\le j\le n$, where $U$ is a neighborhood of $\overline {D_j}$, then  we have 
$$|F_0|\in C(\overline M).$$
Denote that 
$$\Omega_{j,r}:=\{z\in D_j:2p_jG_{D_j}(z,z_j)\geq \log r\}\times M_j\subset \{z\in M:\psi(z)\ge \log r\}$$
for any $1\le j\le n$. It follows from Proposition \ref{p:b6} that
 there is $f_0\in H^2_{\rho}(M,\partial M)$ such that $F_0=f_0^*$
 and 
 \begin{equation}
 	\label{eq:0806c}\pi\|f_0\|^2_{\partial M,\rho}\le\limsup_{r\rightarrow 1-0}\frac{\int_{\{z\in M:\psi(z)\ge\log r\}}|F_0|^2\hat\rho}{1-r}.
 \end{equation}
Note that $F_0(\cdot,\hat w_j)$ has nontangential boundary value $f_0(\cdot,\hat w_j)$ almost everywhere on $\partial D_j$ for any $\hat w_j\in M_j$ and any $1\le j\le n$. Following from   $|F_0|\in C(\overline M)$, the dominated convergence theorem, Lemma \ref{l:coarea} and Lemma \ref{l:v}, we obtain that
\begin{equation}\nonumber
	\begin{split}
		&\limsup_{r\rightarrow 1-0}\frac{\int_{\Omega_{1,r}}|F_0|^2\hat\rho}{1-r}\\
		=&\limsup_{r\rightarrow1-0}\frac{\int_{\frac{\log r}{2p_1}}^0\left(\int_{M_1}\left(\int_{\{G_{D_1}(\cdot,z_1)=s\}}\frac{|F_0|^2e^{-\varphi_1}}{|\bigtriangledown G_{D_1}(\cdot,z_1)|}|dw_1|\right)\prod_{2\le l\le n}e^{-\varphi_l}d\mu_1(\hat w_1)\right)ds}{-\log r}\\
		=&\frac{1}{2p_1}\int_{M_1}\left(\int_{\partial D_1}\frac{|f_0(\cdot,\hat w_1) |^2e^{-\varphi_1}}{\frac{\partial G_{D_1}(w_1,z_1)}{\partial v_{w_1}}}|dw_1|\right)\prod_{2\le l\le n}e^{-\varphi_l}d\mu_1(\hat w_1).
	\end{split}
\end{equation}
By the similar discussions, we have
\begin{displaymath}
	\begin{split}
		&\limsup_{r\rightarrow 1-0}\frac{\int_{\Omega_{j,r}}|F_0|^2\hat\rho}{1-r}\\
		=&\frac{1}{2p_j}\int_{M_j}\left(\int_{\partial D_j}|f_0(\cdot,\hat w_j) |^2\rho|dw_j|\right)d\mu_j(\hat w_j),
	\end{split}
\end{displaymath}
which implies that 
\begin{displaymath}
	\begin{split}
		&\limsup_{r\rightarrow 1-0}\frac{\int_{\{z\in M:\psi(z)\ge\log r\}}|F_0|^2\hat\rho}{1-r}\\
		\le& \sum_{1\le j\le n}\limsup_{r\rightarrow 1-0}\frac{\int_{\Omega_{j,r}}|F_0|^2\hat\rho}{1-r}\\
		=&\sum_{1\le j\le n}\frac{1}{2p_j}\int_{M_j}\left(\int_{\partial D_j}|f_0(\cdot,\hat w_j) |^2\rho|dw_j|\right)d\mu_j(\hat w_j)\\
		=&\pi\|f_0\|^2_{\partial M,\rho}.
	\end{split}
\end{displaymath}
Combining inequality \eqref{eq:0806c}, we have
\begin{equation}
		\label{eq:220801c}
\limsup_{r\rightarrow 1-0}\frac{\int_{\{z\in M:\psi(z)\ge\log r\}}|F_0|^2\hat\rho}{1-r}=\pi\|f_0\|^2_{\partial M,\rho}.
\end{equation}

Choosing any $\tilde f\in H^2_{\rho}(M,\partial M)$ satisfying $\tilde f^*(z_0)=1$, in the following, we will prove that $\| f_0\|_{\partial M,\rho}\le \| \tilde f\|_{\partial M,\rho}$.

As $\tilde f^*(\cdot,\hat w_1)\in H^2(D_1)$, $\tilde f(\cdot,\hat w_1)$ denotes the nontangential boundary value of $\tilde f^*(\cdot,\hat w_1)$ a.e. on $\partial D_1$ and
$$\hat K_{D_1,e^{-\varphi_1}}(w_1)=\sup_{h\in H^2(D_1)}\frac{|h(w_1)|^2}{\frac{1}{2\pi}\int_{\partial D_1}|h(z)|^2e^{-\varphi_1}\left(\frac{\partial G_{D_1}(z,z_1)}{\partial v_z} \right)^{-1}|dz|}$$
for any $w_1\in D_1$, then we have 
\begin{equation}
	\label{eq:0802a}\begin{split}
		&\int_{M_1}\left(\int_{\partial D_1}|\tilde f(\cdot,\hat w_1) |^2\rho |dw_1|\right)d\mu_1(\hat w_1)\\
		\ge&\frac{1}{\hat K_{D_1,e^{-\varphi_1}}(z_1)} \int_{M_1}|\tilde f^*(z_1,\hat w_1)|\prod_{2\le l\le n}e^{-\varphi_l}d\mu_1(\hat w_1),
	\end{split}
\end{equation}
where $\rho=\left(\frac{\partial G_{D_1}(z,z_1)}{\partial v_z} \right)^{-1}\prod_{1\le l\le n}e^{-\varphi_l}$ on $\partial D_1\times M_1$.
Lemma \ref{l:Bergman-prod} shows that 
$$B_{M_1,\lambda_1}(z_2,\ldots,z_n)=\prod_{2\le l\le n}B_{D_l,e^{-\varphi_l}}(z_l),$$
where  $B_{M_1,\lambda_1}$ is the Bergman kernel on $M_1$ with the weight $\lambda_1:=\prod_{2\le l\le n}e^{-\varphi_l}$ and $B_{D_l,e^{-\varphi_l}}$ is the Bergman kernel on $D_l$ with the weight $e^{-\varphi_l}$. Since $\tilde f^*(z_1,\ldots,z_n)=1$, inequality \eqref{eq:0802a} shows that 
\begin{equation}
	\label{eq:0802b}\begin{split}
		&\int_{M_1}\left(\int_{\partial D_1}|\tilde f(\cdot,\hat w_1) |^2\rho |dw_1|\right)d\mu_1(\hat w_1)\\
		\ge&\frac{1}{\hat K_{D_1,e^{-\varphi_1}}(z_1)\cdot B_{M_1,\lambda_1}(z_2,\ldots,z_n)}\\
		=&\frac{1}{\hat K_{D_1,e^{-\varphi_1}}(z_1)\prod_{2\le l\le n}B_{D_l,e^{-\varphi_l}}(z_l)}.
		\end{split}
\end{equation}
Since $\varphi_j=2 u_j$ is harmonic on $D_j$ and $\chi_{j,z_j}=\chi_{j,-u_j}$ for any $1\le j\le n$,  
it follows from Theorem \ref{thm:saitoh-1d} and Remark \ref{rem:function} that 
\begin{equation}
	\label{eq:0802c}
	\begin{split}
			\frac{1}{\hat K_{D_j,e^{-\varphi_j}}(z_j)}&=\frac{1}{2\pi}\int_{\partial D_j}\left|\frac{K_{D_j,e^{-\varphi_j}}(z,z_j)}{K_{D_j,e^{-\varphi_j}}(z_j)}\right|^2e^{-\varphi_j}\left(\frac{\partial G_{D_j}(z,z_j)}{\partial v_z} \right)^{-1}|dz|
			\\
			&=\frac{1}{2\pi}\int_{\partial D_j}|c_j((P_j)_*(f_{z_j}))'(P_j)_*(f_{u_j})|^2e^{-\varphi_j}\left(\frac{\partial G_{D_j}(z,z_j)}{\partial v_z} \right)^{-1}|dz|
	\end{split}
\end{equation}
and 
\begin{equation}
	\label{eq:0802d}
	\begin{split}
			\frac{1}{B_{D_j,e^{-\varphi_j}}(z_j)}&=\int_{D_j}\left|\frac{B_{D_j,e^{-\varphi_j}}(\cdot,z_j)}{B_{D_j,e^{-\varphi_j}}(z_j)}\right|^2e^{-\varphi_j}\\&=\int_{D_j}|c_j((P_j)_*(f_{z_j}))'(P_j)_*(f_{u_j})|^2e^{-\varphi_j},
	\end{split}
\end{equation}
where $c_j$ is constant such that $\left(c_j((P_j)_*(f_{z_j}))'(P_j)_*(f_{u_j})\right)(z_j)=1$.
Note that  $F_0=c_0\prod_{1\le j\le n} ((P_j)_*(f_{z_j}))'(P_j)_*(f_{u_j})$ and $F_0(z_0)=0$. Following from $F_0=f^*$, inequality \eqref{eq:0802b}, equality \eqref{eq:0802c} and equality \eqref{eq:0802d}, we have
\begin{displaymath}
\begin{split}
		\int_{M_1}\left(\int_{\partial D_1}|\tilde f(\cdot,\hat w_1) |^2\rho |dw_1|\right)d\mu_1(\hat w_1)&\ge \int_{M_1}\left(\int_{\partial D_1}|F_0|^2\rho |dw_1|\right)d\mu_1(\hat w_1)\\
		&=\int_{M_1}\left(\int_{\partial D_1}|f_0|^2\rho |dw_1|\right)d\mu_1(\hat w_1).
\end{split}
\end{displaymath}
By similar discussion, we have 
\begin{displaymath}
	\int_{M_j}\left(\int_{\partial D_j}|\tilde f(\cdot,\hat w_j) |^2\rho |dw_j|\right)d\mu_j(\hat w_j)\ge\int_{M_j}\left(\int_{\partial D_j}|f_0|^2\rho |dw_j|\right)d\mu_j(\hat w_j)
\end{displaymath}
for any $1\le j\le n$,
which shows that
\begin{equation}
	\|\tilde f\|_{\partial M,\rho}\ge \| f_0\|_{\partial M,\rho}.
\end{equation}
By Lemma \ref{l:sup-b}, we know that
\begin{equation}
	\label{eq:0802e}K_{\partial M,\rho}=\frac{|f_0^*(z_0)|^2}{ \| f_0\|_{\partial M,\rho}^2}.
\end{equation}
 Combining inequality \eqref{eq:0801d}, equality \eqref{eq:220801b}, equality \eqref{eq:220801c} and equality \eqref{eq:0802e}, we have 
$$\left(\int_0^{+\infty}c(t)e^{-t}dt\right)B_{\tilde\rho}(z_0)= \frac{1}{\pi} K_{\partial M,\rho}(z_0).$$

Then the sufficiency of the characterization in Theorem \ref{thm:main1-1} has been proved.

\subsection{Proof of Remark \ref{r:product 1-1}}
\

Assume that the three statements in Theorem \ref{thm:main1-1} hold.

Note that, for any $h\in\mathcal{O}(M)$ satisfying 
$\int_M|h|^2\tilde\rho=\int_{M}\frac{|h|^2}{\prod_{1\le j\le n}|g_j|^2}c(-\psi)\prod_{1\le j\le n}e^{-2u_j}<+\infty,$ we have
$$\frac{h}{\prod_{1\le j\le n}g_j}\in\mathcal{O}(M).$$ Then we have $$B_{\tilde\rho}(\cdot,\overline{z_0})=\left(\prod_{1\le j\le n}\overline{g_j(z_j)}\right)B_{\tilde\rho\prod_{1\le j\le n}g_j}(\cdot,\overline{z_0})\prod_{1\le j\le n}g_j.$$
Following from Lemma \ref{l:divide g} and the definition of $K_{\partial M,\rho}$, we have 
$$K_{\partial M,\rho}(\cdot,\overline{z_0})=\left(\prod_{1\le j\le n}\overline{g_j(z_j)}\right)K_{\partial M,\rho\prod_{1\le j\le n}g_j}(\cdot,\overline{z_0})\prod_{1\le j\le n}g_j.$$ 
Thus, it suffices to consider the case $g_j\equiv1$ for any $j$.

Following the discussions in Step 2 in the proof of Theorem \ref{thm:main1-1}, we obtain that
\begin{equation}
	\label{eq:0823a}G(0)=\int_{M}|F_0|^2\tilde\rho \end{equation}
	\begin{equation}
		\label{eq:0823b}\|\tilde f\|_{\partial M,\rho}\ge\|f_0\|_{\partial M,\rho}	\end{equation}
		for any $\tilde f\in H^2_{\rho}(M,\partial M)$ satisfying $\tilde f^*(z_0)=1$,
		where 
		$$F_0=c_0\prod_{1\le j\le n} ((P_j)_*(f_{z_j}))'(P_j)_*(f_{u_j}),$$
		$c_0$ is a constant such that $F_0(z_0)=1$ and $f_0\in H^2_{\rho}(M,\partial M)$ such that $f_0^*=F_0$.
As 	
$$\int_M\left|\frac{B_{\tilde\rho}(\cdot,\overline{z_0})}{B_{\tilde\rho}(z_0,\overline{z_0})}\right|^2\tilde\rho=\frac{1}{B_{\tilde\rho}(z_0,\overline{z_0})}=G(0),$$ it follows from Lemma \ref{l:unique} and equality \eqref{eq:0823a} that 
\begin{equation}
	\label{eq:0823d}\frac{B_{\tilde\rho}(\cdot,\overline{z_0})}{B_{\tilde\rho}(z_0,\overline{z_0})}=F_0.\end{equation}
By definition, we know that there exists $f(=\sum_{m=1}^{+\infty}e_m\overline{e^*_m(z_0)} )\in H^2_{\rho}(M, \partial M)$ such that $f^*=K_{\partial M,\rho}(\cdot,\overline{z_0})$, where $\{e_m\}_{m\in\mathbb{Z}_{\ge1}}$ is a complete orthonormal basis of $H^2_{\rho}(M,\partial M)$. It follows from Lemma \ref{l:b5} that
\begin{equation}
	\label{eq:0823c}\|\frac{f}{K_{\partial M,\rho}(z_0,\overline{z_0})}\|_{\partial M,\rho}^2=\frac{\ll f,f\gg_{\partial M,\rho}}{|K_{\partial M,\rho}(z_0,\overline{z_0})|^2}=\frac{1}{K_{\partial M,\rho}(z_0,\overline{z_0})}.\end{equation}
Following from Lemma \ref{l:b7}, inequality \eqref{eq:0823b} and equality \eqref{eq:0823c}, we have 
$$\frac{f}{K_{\partial M,\rho}(z_0,\overline{z_0})}=f_0,$$
	which shows that 
	\begin{equation}
		\label{eq:0823e}\frac{K_{\partial M,\rho}(\cdot,\overline{z_0})}{K_{\partial M,\rho}(z_0,\overline{z_0})}=\frac{f^*}{K_{\partial M,\rho}(z_0,\overline{z_0})}=F_0.	\end{equation}
Theorem \ref{thm:main1-1} shows that 
	$K_{\partial M,\rho}(z_0,\overline{z_0})=\left(\int_0^{+\infty}c(t)e^{-t}dt\right)\pi B_{\tilde\rho}(z_0,\overline{z_0}),$
	thus combining equality \eqref{eq:0823d} and equality \eqref{eq:0823e}, we have
	$$K_{\partial M,\rho}(\cdot,\overline{z_0})=\left(\int_0^{+\infty}c(t)e^{-t}dt\right)\pi B_{\tilde\rho}(\cdot,\overline{z_0})=c_1\prod_{1\le j\le n} ((P_j)_*(f_{z_j}))'(P_j)_*(f_{u_j}),$$
	where $c_1$ is a constant.

\section{Proofs of Theorem \ref{thm:main1-2} and Remark \ref{product 1-2}}\label{sec:proof2}

In this section, we prove Theorem \ref{thm:main1-2} and Remark \ref{product 1-2}.

\subsection{Proof of Theorem \ref{thm:main1-2}}
\

The inequality part of Theorem \ref{thm:main1-2} follows from Proposition \ref{p:inequality}.

We prove the characterization in Theorem \ref{thm:main1-2} in two steps: Firstly,  we prove the necessity of the characterization; Secondly, we prove the sufficiency of the characterization.

\

\emph{Step 1:} In this step, we will prove the necessity of the characterization.

Assume that the equality
\begin{equation}
	\label{eq:220802a}K_{\partial M,\rho}^{I,h_0}(z_0)=\left(\int_0^{+\infty}c(t)e^{-t}dt\right)\pi B_{\tilde\rho}^{I,h_0}(z_0)
\end{equation}
holds.

Denote \begin{equation*}
\begin{split}
\inf\left\{\int_{\{\psi<-t\}}|\tilde{f}|^{2}\tilde\rho:(\tilde{f}-h_0,z_0)\in I\,\&{\,}\tilde{f}\in \mathcal{O} (\{\psi<-t\})\right\}
\end{split}
\end{equation*}
by $G(t)$ for any $t\ge0$, then we have 
$$G(0)=\frac{1}{B^{I,h_0}_{\tilde\rho}(z_0)}.$$ 
Since $\mathcal{I}(\psi)_{z_0}\subset I$, it follows from Theorem \ref{thm:general_concave} that $G(h^{-1}(r))$ is concave with respect to $r\in(0,\int_0^{+\infty}c(t)e^{-t}dt)$, where $h(t)=\int_t^{+\infty}c(l)e^{-l}dl$. 
Inequality \eqref{eq:0801d2} shows that 
\begin{equation}
	\label{eq:0806e}
	\begin{split}
		\left(\int_0^{+\infty}c(t)e^{-t}dt\right)B^{I,h_0}_{\tilde\rho}(z_0)&=\frac{\int_0^{+\infty}c(t)e^{-t}dt}{G(0)}\\
		&\le\liminf_{t\rightarrow 0+0}\frac{\int_0^{+\infty}c(l)e^{-l}dl-\int_t^{+\infty}c(l)e^{-l}dl}{G(0)-G(t)} \\
		&\leq\liminf_{r\rightarrow 1-0}\frac{1-r}{\int_{\{z\in M:\psi(z)\ge\log r\}}|F_0|^2\hat\rho}\\
		&\le \frac{1}{\pi\|f\|^2_{\partial M,\rho}} \\
		&\le \frac{1}{\pi} K^{I,h_0}_{\partial M,\rho}(z_0),
	\end{split}
\end{equation}
where $\hat\rho=\prod_{1\le j\le n}e^{-\varphi_j}$, $f\in H^2_{\rho}(M,\partial M)$ such that $f^*=F_0$, and $F_0$ is a holomorphic function on $M$ such that $(F_0-h_0,z_0)\in I$ and $\int_M|F_0|^2\tilde\rho=G(0)$.
It follows from equality \eqref{eq:220802a} that inequality \eqref{eq:0806e} becomes an equality, which implies  that 
 $$\frac{\int_0^{+\infty}c(t)e^{-t}dt}{G(0)}=\liminf_{t\rightarrow0+0}\frac{\int_0^{+\infty}c(l)e^{-l}dl-\int_t^{+\infty}c(l)e^{-l}dl}{G(0)-G(t)}.$$
Following from the concavity of $G(h^{-1}(r))$, we obtain that $G(h^{-1}(r))$ is linear with respect to $r\in(0,\int_0^{+\infty}c(t)e^{-t}dt)$.  It follows from Corollary \ref{c:linear} that 
\begin{equation}
	\label{eq:0801f2}G(t)=\int_{\{\psi<-t\}}|F_0|^2\tilde\rho
\end{equation}
holds for any $t\ge0$.

 Denote 
\begin{equation*}
\begin{split}
\inf\left\{\int_{\{\psi<-t\}}|\tilde{f}|^{2}\tilde\rho:(\tilde{f}-F_0,z_0)\in\mathcal{I}(\psi)_{z_0}\,\&{\,}\tilde{f}\in \mathcal{O} (\{\psi<-t\})\right\},
\end{split}
\end{equation*}
by $G_1(t)$ for any $t\ge0$. By equality \eqref{eq:0801f2}, we know $G_1(t)\le G(t)$ for any $t\ge0$. As $\mathcal{I}(\psi)_{z_0}\subset I$, then we get that $G_1(t)\ge G(t)$ for any $t\geq0$. Thus, we have $G(t)=G_1(t)$ for any $t\ge0$ and $G_1(h^{-1}(r))$ is linear with respect to $r\in (0,\int_0^{+\infty}c(t)e^{-t}dt)$.
Theorem \ref{thm:linear-2d} and Remark \ref{r:var=-infty} show that the following the following statements hold:
	
	$(1)$ $F_0=\sum_{\alpha\in E}d_{\alpha}(w-z_0)^{\alpha}+g_0$ on a neighborhood of $z_0$, where $g_0$ is a holomorphic function on $V_0$ such that $(g_0,z_0)\in\mathcal{I}(\psi)_{z_0}$ and $d_{\alpha}\in\mathbb{C}$ such that $\sum_{\alpha\in E}|d_{\alpha}|\not=0$;
	
		$(2)$ $\varphi_j=2\log|\tilde g_j|+2\tilde u_j$ on $D_j$ for any $1\le j\le n$, where $\tilde u_j$ is a harmonic function on $D_j$ and $\tilde g_j$ is a holomorphic function on $D_j$ such that $g_j(z_j)\not=0$;

    $(3)$ $\chi_{j,z_j}^{\alpha_j+1}=\chi_{j,-\tilde u_j}$ for any $1\le j\le n$ and $\alpha\in E$ satisfying $d_{\alpha}\not=0$, $\chi_{j,z_j}$ and $\chi_{j,-\tilde u_j}$ are the characters associated to functions $G_{D_j}(\cdot,z_j)$ and $-\tilde u_j$ respectively.
    
     As $\varphi_j$ is continuous at $p$ for any $p\in \partial D_j$ and $D_j\Subset \mathbb{C}$ , then $\tilde g_j$ has at most a finite number of zeros on $D_j$.
Then there is $ g_j\in\mathcal{O}(\mathbb{C})$, which satisfies that $\frac{\tilde g_j}{ g_j}\in\mathcal{O}(D_j)$, $\frac{\tilde g_j}{ g_j}$ has no zero point on $D_j$ and $g_j$ has no zero point on $\partial D_j$. It is clear that $\chi_{j,\log|g_j|-\log|\tilde g_j|}=1$. Thus, we have 
$$\varphi_j=2\log|g_j|+2u_j\text{ and }\chi_{j,z_j}^{\alpha_j+1}=\chi_{j,-u_j}$$
for any $1\le j\le n$ and $\alpha\in E$, where $u_j=\tilde u_j+\log|\tilde g_j|-\log|g_j|$ is harmonic on $D_j$.    
    By equality \eqref{eq:220802a} and inequality \eqref{eq:0806e}, we know that
\begin{displaymath}
	\begin{split}
		\|f\|^2_{\partial M,\rho}=\frac{1}{K_{\partial M,\rho}^{I,h_0}(z_0)}.
	\end{split}
\end{displaymath}
Note that $(f^*-h_0,z_0)=(F_0-h_0,z_0)\in I$, it follows from Lemma \ref{l:b7} that $f=f_0$. Then necessity of the characterization in case $n>1$ has been proved.

Now, we prove that $\varphi_1$ is harmonic on $D_1$ when $n=1$.  By the analyticity of $\partial D_1$, there exists $\tilde\varphi_1 \in C(\overline D_1)$  such that $\tilde\varphi_1 |_{\partial D_1}=\varphi_1 |_{\partial D_1}$ and $\tilde\varphi_1$ is harmonic on $D_1$. Denote that 
 $$\tilde\rho_{1}:=c(-\psi)e^{-\tilde\varphi_1}.$$
  As $\varphi_1$ is subharmonic on $D_1$, we have 
 $$\tilde\rho_1\le \tilde\rho$$
  on $M$. It follows from Proposition \ref{p:inequality} that 
 \begin{displaymath}
 	K_{\partial M,\rho}^{I,h_0}(z_0)\ge\left(\int_0^{+\infty}c(t)e^{-t}dt\right)\pi B^{I,h_0}_{\tilde\rho_1}(z_0)\ge \left(\int_0^{+\infty}c(t)e^{-t}dt\right)B_{\tilde\rho}^{I,h_0}(z_0).
 \end{displaymath}
As $\left(\int_0^{+\infty}c(t)e^{-t}dt\right)B_{\tilde\rho}^{I,h_0}(z_0)=K_{\partial M,\rho}^{I,h_0}(z_0)>0$, we have
$$B^{I,h_0}_{\tilde\rho_1}(z_0)=B_{\tilde\rho}^{I,h_0}(z_0)>0.$$ Thus, we have $\tilde\rho=\tilde\rho_1$, a.e., $\tilde\varphi_1=\varphi_1$, which implies that $\varphi_1$ is harmonic on $D_j$. 

Thus the necessity of the characterization has been proved.

\

\emph{Step 2:} In this step, we prove the sufficiency of the characterization.  
  Assume that the three statements $(1)-(3)$ hold.

    Denote 
  \begin{equation*}
\begin{split}
\inf\left\{\int_{\{\psi<-t\}}|\tilde{f}|^{2}\tilde\rho:(\tilde{f}-f^*_0,z_0)\in\mathcal{I}(\psi)_{z_0}\,\&{\,}\tilde{f}\in \mathcal{O} (\{\psi<-t\})\right\},
\end{split}
\end{equation*}
by $G_2(t)$.  Theorem \ref{thm:linear-2d} tells us  that $G_2(h^{-1}(r))$ is linear with respect to $r\in(0,\int_0^{+\infty}c(t)e^{-t}dt)$. Denote that $B_{\tilde\rho}^{\mathcal{I}(\psi)_{z_0},f_0^*}=\frac{1}{G_2(0)}$. Since $\mathcal{I}(\psi)_{z_0}\subset I$ and $(f_0^*-h_0,z_0)\in I$, we have $G_2(t)\ge G(t)$ for any $t\ge0$ and 
$$B_{\tilde\rho}^{\mathcal{I}(\psi)_{z_0},f_0^*}(z_0)\le B_{\tilde\rho}^{I,h_0}(z_0).$$
Since $\mathcal{I}(\psi)_{z_0}\subset I$ and $(f_0^*-h_0,z_0)\in I$, we know that $K_{\partial M,\rho}^{\mathcal{I}(\psi)_{z_0},f_0^*}(z_0)\le K_{\partial M,\rho}^{I,h_0}(z_0)$. 
As $K_{\partial M,\rho}^{I,h_0}(z_0)=\frac{1}{\|f_0\|^2_{\partial M,\rho}}$, we have $K_{\partial M,\rho}^{\mathcal{I}(\psi)_{z_0},f_0^*}(z_0)\geq K_{\partial M,\rho}^{I,h_0}(z_0)$, which shows that 
$$K_{\partial M,\rho}^{\mathcal{I}(\psi)_{z_0},f_0^*}(z_0)= K_{\partial M,\rho}^{I,h_0}(z_0).$$
Following Proposition \ref{p:inequality}, we obtain that
\begin{equation}
	\label{eq:0824a}\begin{split}
		K_{\partial M,\rho}^{\mathcal{I}(\psi)_{z_0},f_0^*}(z_0)&= K_{\partial M,\rho}^{I,h_0}(z_0)\\
		&\ge\left(\int_0^{+\infty}c(t)e^{-t}dt\right)\pi B_{\tilde\rho}^{I,h_0}(z_0)\\
		&\ge \left(\int_0^{+\infty}c(t)e^{-t}dt\right)\pi B_{\tilde\rho}^{\mathcal{I}(\psi)_{z_0},f_0^*}(z_0).
	\end{split}
	\end{equation}
Thus, it suffices to prove that 
\begin{equation}
	\label{eq:0824b}K_{\partial M,\rho}^{\mathcal{I}(\psi)_{z_0},f_0^*}(z_0)=\left(\int_0^{+\infty}c(t)e^{-t}dt\right)\pi B_{\tilde\rho}^{\mathcal{I}(\psi)_{z_0},f_0^*}(z_0).
	\end{equation}
 As $M\Subset \mathbb{C}^n$, it is clear that $B_{\tilde\rho}^{\mathcal{I}(\psi)_{z_0},f_0^*}(z_0)>0$.

For the case $n>1$,  without loss of generality, assume that $g_j(z_j)=1$. Note that, for any $h\in\mathcal{O}(M)$ satisfying $\int_M|h|^2\tilde\rho= \int_{M}\frac{|h|^2}{\prod_{1\le j\le n}|g_j|^2}c(-\psi)\prod_{1\le j\le n}e^{-2u_j}<+\infty$, we have
$\frac{h}{\prod_{1\le j\le n}g_j}\in\mathcal{O}(M)$. As $f_0^*=\sum_{\alpha\in E}d_{\alpha}w^{\alpha}+g_0$ near $z_0$ for any $t\ge0$, where $E=\{\alpha\in \mathbb{Z}_{\ge0}^n:\sum_{1\le j\le n}\frac{\alpha_j+1}{p_j}=1\}$ and  $(g_0,z_0)\in\mathcal{I}(\psi)_{z_0}$, it follows from Lemma \ref{l:psi} and $g_j(z_j)=1$ that for any holomorphic function $(h,z_0)\in\mathcal{O}_{M,z_0}$, $(h-f_0^*,z_0)\in\mathcal{I}(\psi)_{z_0}$ if and only if $(h\prod_{1\le j\le n}g_j-f_0^*,z_0)\in\mathcal{I}(\psi)_{z_0}$.
 It is clear that 
 \begin{equation}
 	\label{eq:0824e}\begin{split}
		&B_{\tilde \rho}^{\mathcal{I}(\psi)_{z_0},f_0^*}(z_0)\\
		=&\frac{1}{\inf\left\{\int_{M}|f|^2\tilde \rho:f\in\mathcal{O}(M) \,\&\,(f-f^*_0,z_0)\in {\mathcal{I}(\psi)_{z_0}}\right\}}\\
		=&\frac{1}{\inf\left\{\int_{M}|\tilde f|^2\prod_{1\le j\le n}|g_j|^2\tilde \rho:f\in\mathcal{O}(M) \,\&\,(\tilde f\prod_{1\le j\le n}g_j-f^*_0,z_0)\in {\mathcal{I}(\psi)_{z_0}}\right\}}\\
		=&\frac{1}{\inf\left\{\int_{M}|\tilde f|^2\prod_{1\le j\le n}|g_j|^2\tilde \rho:f\in\mathcal{O}(M) \,\&\,(\tilde f-f^*_0,z_0)\in {\mathcal{I}(\psi)_{z_0}}\right\}}\\
		=&B_{\tilde\rho\prod_{1\le j\le n}|g_j|^2}^{\mathcal{I}(\psi)_{z_0},f^*_0}(z_0).
	\end{split}
 \end{equation}
It follows from Lemma \ref{l:divide g} that 
\begin{equation}
	\label{eq:0824f}\begin{split}
			&K_{\partial M,\rho}^{\mathcal{I}(\psi)_{z_0},f_0^*}(z_0)\\
			=&\frac{1}{\inf\left\{\|f\|_{\partial M,\rho}^2:f\in H^2_{\rho}(M,\partial M)\,\&\,(f^*-f^*_0,z_0)\in\mathcal{I}(\psi)_{z_0} \right\}}\\
			=&\frac{1}{\inf\left\{\|\tilde f\|_{\partial M,\rho\prod_{1\le j\le n}g_j}^2:f\in H^2_{\rho}(M,\partial M)\,\&\,(\tilde f^*\prod_{1\le j\le n}g_j-f^*_0,z_0)\in\mathcal{I}(\psi)_{z_0} \right\}}\\
			=&\frac{1}{\inf\left\{\|\tilde f\|_{\partial M,\rho\prod_{1\le j\le n}g_j}^2:f\in H^2_{\rho}(M,\partial M)\,\&\,(\tilde f^*-f^*_0,z_0)\in\mathcal{I}(\psi)_{z_0} \right\}}\\
			=&K_{\partial M,\rho\prod_{1\le j\le n}|g_j|^2}^{\mathcal{I}(\psi)_{z_0},f_0^*}(z_0).
				\end{split}
\end{equation}
Thus, when $n>1$, it suffices to prove the case that $g_j\equiv1$ for any $1\le j\le n$.

In the following part, $n$ is any positive integer, and we assume that $\varphi_j$ is harmonic on $D_j$ for any $1\le j\le n$.

 Following from Corollary \ref{c:linear} and Remark \ref{r:1.1}, we get that there is a holomorphic function $F$ on $M$, such that $(F-f^*_0,z_0)\in\mathcal{I}(\psi)_{z_0},$
\begin{equation}
	\label{eq:220802b}
G_2(t)=\int_{\{\psi<-t\}}|F|^2\tilde\rho
	\end{equation}
holds for any $t\geq0$ and 
$$F=\sum_{\alpha\in E}\tilde d_{\alpha}\prod_{1\le j\le n} ((P_j)_*(f_{z_j}))'(P_j)_*(f_{\frac{\varphi_j}{2}}f_{z_j}^{\alpha_j}),$$
 where $\tilde d_{\alpha}=\lim_{z\rightarrow z_0}\frac{\sum_{\alpha\in E}d_{\alpha}(w-z_0)^{\alpha}}{\prod_{1\le j\le n} ((P_j)_*(f_{z_j}))'(P_j)_*(f_{\frac{\varphi_j}{2}}f_{z_j}^{\alpha_j})}$ for any $\alpha\in E$, $P_j$ is the universal covering from unit disc $\Delta$ to $ D_j$, $f_{\frac{\varphi_j}{2}}$ is a holomorphic function on $\Delta$ such that $|f_{\frac{\varphi_j}{2}}|=P_j^*(e^{\frac{\varphi_j}{2}})$, and $f_{z_j}$ is a holomorphic function on $\Delta$ such that $|f_{z_j}|=P_j^*(e^{G_{D_j}(\cdot,z_j)})$ for any $1\le j\le n$. Equality \eqref{eq:220802b} implies 
 \begin{equation}
 	\label{eq:220802c}
 	\begin{split}
 			\left(\int_0^{+\infty}c(t)e^{-t}dt\right)B_{\tilde\rho}^{\mathcal{I}(\psi)_{z_0},f_0^*}(z_0)&=\frac{\int_0^{+\infty}c(t)e^{-t}dt}{G_2(0)}\\&=\limsup_{r\rightarrow 1-0}\frac{\int_0^{-\log r}c(t)e^{-t}dt}{\int_{\{z\in M:\psi(z)\ge\log r\}}|F|^2\tilde\rho}\\
 			&=\limsup_{r\rightarrow 1-0}\frac{1-r}{\int_{\{z\in M:\psi(z)\ge\log r\}}|F|^2\hat\rho},
 	\end{split}
 \end{equation}
 where $\hat\rho=\prod_{1\le j\le n}e^{-\varphi_j}$.

Note that $\varphi_j\in C(\overline{D_j})$ and $G_{D_j}(\cdot,z_j)$ can be extended to a harmonic function on a $U\backslash\{z_j\}$  for any $1\le j\le n$, where $U$ is a neighborhood of $\overline {D_j}$, then  we have 
$$|F|\in C(\overline M).$$
Denote that 
$$\Omega_{j,r}:=\{z\in D_j:2p_jG_{D_j}(z,z_j)\geq \log r\}\times M_j\subset \{z\in M:\psi(z)\ge \log r\}$$
for any $1\le j\le n$. 
Proposition \ref{p:b6} shows that there is $f\in H^2_{\rho}(M,\partial M)$, such that $F=f^*$
and
\begin{equation}
	\label{eq:220803a}\pi\|f\|^2_{\partial M,\rho}\le \limsup_{r\rightarrow 1-0}\frac{\int_{\{z\in M:\psi(z)\ge\log r\}}|F|^2\hat\rho}{1-r}.
\end{equation}
Note that $F(\cdot,\hat w_j)$ has nontangential boundary value $f(\cdot,\hat w_j)$ almost everywhere on $\partial D_j$ for any $\hat w_j\in M_j$ and any $1\le j\le n$. Following from   $|F|\in C(\overline M)$, the dominated convergence theorem, Lemma \ref{l:coarea} and Lemma \ref{l:v}, we obtain that
\begin{equation}\nonumber
	\begin{split}
		&\limsup_{r\rightarrow 1-0}\frac{\int_{\Omega_{1,r}}|F|^2\hat\rho}{1-r}\\
		=&\limsup_{r\rightarrow1-0}\frac{\int_{\frac{\log r}{2p_1}}^0\left(\int_{M_1}\left(\int_{\{G_{D_1}(\cdot,z_1)=s\}}\frac{|F|^2e^{-\varphi_1}}{|\bigtriangledown G_{D_1}(\cdot,z_1)|}|dw_1|\right)\prod_{2\le l\le n}e^{-\varphi_l}d\mu_1(\hat w_1)\right)ds}{-\log r}\\
		=&\frac{1}{2p_1}\int_{M_1}\left(\int_{\partial D_1}\frac{|f(\cdot,\hat w_1) |^2e^{-\varphi_1}}{\frac{\partial G_{D_1}(w_1,z_1)}{\partial v_{w_1}}}|dw_1|\right)\prod_{2\le l\le n}e^{-\varphi_l}d\mu_1(\hat w_1).
	\end{split}
\end{equation}
By the similar discussions, we have
\begin{displaymath}
	\begin{split}
		&\limsup_{r\rightarrow 1-0}\frac{\int_{\Omega_{j,r}}|F|^2\hat\rho}{1-r}\\
		=&\frac{1}{2p_j}\int_{M_j}\left(\int_{\partial D_j}|f(\cdot,\hat w_j) |^2\rho|dw_j|\right)d\mu_j(\hat w_j),
	\end{split}
\end{displaymath}
which implies that 
\begin{displaymath}
	\begin{split}
		&\limsup_{r\rightarrow 1-0}\frac{\int_{\{z\in M:\psi(z)\ge\log r\}}|F|^2\hat\rho}{1-r}\\
		\le& \sum_{1\le j\le n}\limsup_{r\rightarrow 1-0}\frac{\int_{\Omega_{j,r}}|F|^2\hat\rho}{1-r}\\
		=&\sum_{1\le j\le n}\frac{1}{2p_j}\int_{M_j}\left(\int_{\partial D_j}|f(\cdot,\hat w_j) |^2\rho|dw_j|\right)d\mu_j(\hat w_j)\\
		=&\pi\|f\|^2_{\partial M,\rho}.
	\end{split}
\end{displaymath}
Combining inequality \eqref{eq:220803a}, we have
\begin{equation}
		\label{eq:220802d}
\limsup_{r\rightarrow 1-0}\frac{\int_{\{z\in M:\psi(z)\ge\log r\}}|F|^2\hat\rho}{1-r}=\pi\|f\|^2_{\partial M,\rho}.
\end{equation}

Let  $\alpha\in E$ satisfing $d_{\alpha}\not=0$. Denote that 
$$\lambda_j:=\frac{1}{\alpha_j+1}e^{-\varphi_j}\left(\frac{\partial G_{D_j}(\cdot,z_j)}{\partial v_z}\right)^{-1},\ \rho_j:=e^{-\varphi_j}\text{ and }\hat\rho_j:=\prod_{1\le l\le n,\,l\not=j}e^{-\varphi_l}.$$
 It follows from Corollary \ref{c:saitoh-higher jet} and Remark \ref{rem:saitoh2} that 
 \begin{equation}
 	\label{eq:0804d}\frac{1}{K_{\partial D_j,\lambda_j}^{I_{\alpha_j},(w_j-z_j)^{\alpha_j}}(z_j)}=\frac{1}{2\pi}\int_{\partial D_j}|c_{j,\alpha_j}((P_j)_*(f_{z_j}))'(P_j)_*(f_{\frac{\varphi_1}{2}}f_{z_j}^{\alpha_j})|^2\lambda_j |dw_j|
 \end{equation}
and
\begin{equation}
	\label{eq:0804e}\frac{1}{B_{ D_j,\rho_j}^{I_{\alpha_j},(w_j-z_j)^{\alpha_j}}(z_j)}=\int_{D_j}|c_{j,\alpha_j}((P_j)_*(f_{z_j}))'(P_j)_*(f_{\frac{\varphi_1}{2}}f_{z_j}^{\alpha_j})|^2\rho_j,
\end{equation}
where $c_{j,\alpha_j}=\lim_{w_j\rightarrow z_j}\frac{(w_j-z_j)^{\alpha_j}}{((P_j)_*(f_{z_j}))'(P_j)_*(f_{\frac{\varphi_1}{2}}f_{z_j}^{\alpha_j})}\not=0$ is a constant  and $I_{\alpha_j}=((w_j-z_j)^{\alpha_j+1})$ is an ideal of $\mathcal{O}_{D_j,z_j}$. Denote that
$$\tilde\psi_{\alpha_1}:=\max_{2\le j\le n}\left\{2p_j(1-\frac{\alpha_1+1}{p_1})G_{D_j}(\cdot,z_j)\right\}$$
on $M_1$. 
Denote that 
$$f_1=c_{1,\alpha_1}((P_1)_*(f_{z_1}))'(P_1)_*(f_{\frac{\varphi_1}{2}}f_{z_1}^{\alpha_1})$$ and 
$$\hat f_1=\prod_{2\le j\le n}c_{j,\alpha_j}((P_j)_*(f_{z_j}))'(P_j)_*(f_{\frac{\varphi_1}{2}}f_{z_j}^{\alpha_j}).$$
As $\alpha \in E$, it is clear that $\sum_{2\le j\le n}\frac{\alpha_j+1}{2p_j(1-\frac{\alpha_1+1}{p_1})}=1$. Using Lemma \ref{l:psi}, we know that $(\hat f_1-\prod_{2\le j\le n}(w_j-w_j)^{\alpha_j},\hat z_1)\in\mathcal{I}(\tilde \psi_{\alpha_1})_{\hat z_1}$,  where $\hat z_1=(z_2,\ldots,z_n)\in M_1$.  By Lemma \ref{l:Bergman-prod} and equality \eqref{eq:0804e}, we obtain that 
 \begin{equation}
\nonumber
 	\begin{split}
 		 	\frac{1}{B_{M_1,\hat \rho_1}^{\mathcal{I}(\tilde\psi_{\alpha_1})_{\hat z_1},\prod_{2\le j\le n}(w_j-w_j)^{\alpha_j}}}&=\prod_{2\le j\le n}\frac{1}{B_{ D_j,\rho_j}^{I_{\alpha_j},(w_j-z_j)^{\alpha_j}}(z_j)}\\
 		 	&=\int_{M_1}|\hat f_1|^2\hat\rho_1.
 	\end{split}
 \end{equation}
Thus, we get that
\begin{equation}
	\label{eq:0804f}\int_{M_1}\hat f_1\overline{g}\hat\rho_1=0
\end{equation}
holds for any $g\in A^2_{\hat\rho_1}(M_1)$ satisfying $(g,\hat z_1)\in\mathcal{I}(\tilde \psi_{\alpha_1})_{\hat z_1}$. Following from equality \eqref{eq:0804d}, we know that 
\begin{equation}
	\label{eq:0804g}\frac{1}{2\pi}f_1\overline{g}\lambda_1|dw_1|=0
\end{equation}
holds for any $g\in H^2_{\lambda_1}(D_1,\partial D_1)$ satisfying $(g^*,z_0)\in I_{\alpha_j}$. Note that $\rho|_{\partial D_1\times M_1}=\frac{\alpha_1+1}{p_1} \lambda_1\times\hat\rho_1$. Using equality \eqref{eq:0804f}, equality \eqref{eq:0804g} and Lemma \ref{l:b9}, we obtain that
\begin{equation}
	\nonumber\frac{1}{2\pi}\int_{\partial D_1}\int_{M_1}f_1\hat f_1\overline{g}\rho d\mu_1|dw_1|=0
\end{equation}
holds for any $g\in H^2_{\rho}(M,\partial M)$ satisfying $(g^*,z_0)\in\mathcal{I}(\psi)_{z_0}$. By similar discussion, we obtain that
\begin{equation}
	\nonumber\frac{1}{2\pi}\int_{\partial D_l}\int_{M_l} \left(\prod_{1\le j\le n}c_{j,\alpha_j}((P_j)_*(f_{z_j}))'(P_j)_*(f_{\frac{\varphi_1}{2}}f_{z_j}^{\alpha_j})\right) \overline{g}\rho d\mu_l|dw_l|=0
\end{equation}
holds for any $1\le l\le n$ and any $g\in H^2_{\rho}(M,\partial M)$ satisfying $(g^*,z_0)\in\mathcal{I}(\psi)_{z_0}$, which implies that 
\begin{equation}
	\nonumber \sum_{1\le j\le n}\frac{1}{2\pi}\int_{\partial D_l}\int_{M_l} f\overline{g}\rho d\mu_l|dw_l|=0
\end{equation}
holds  any $g\in H^2_{\rho}(M,\partial M)$ satisfying $(g^*,z_0)\in\mathcal{I}(\psi)_{z_0}$, where $f\in H^2_{\rho}(M,\partial M)$ and $f^*=F=\sum_{\alpha\in E}d_{\alpha}\prod_{1\le j\le n}((P_j)_*(f_{z_j}))'(P_j)_*(f_{\frac{\varphi_1}{2}}f_{z_j}^{\alpha_j})$. It follows that 
\begin{equation}
	\label{eq:0804h}\frac{1}{K_{\partial M,\rho}^{\mathcal{I}(\psi)_{z_0},f_0^*}(z_0)}=\|f\|^2_{\partial M,\rho}.
\end{equation}
Combining equality \eqref{eq:220802c}, equality \eqref{eq:220802d} and equality \eqref{eq:0804h}, we have 
\begin{displaymath}
	\begin{split}
		\left(\int_0^{+\infty}c(t)e^{-t}dt\right)\pi B_{\tilde\rho}^{\mathcal{I}(\psi)_{z_0},f_0^*}(z_0)&=\pi\limsup_{r\rightarrow 1-0}\frac{1-r}{\int_{\{z\in M:\psi(z)\ge\log r\}}|F|^2\hat\rho}\\
		&=\frac{1}{\|f\|^2_{\partial M,\rho}}\\
		&=K_{\partial M,\rho}^{\mathcal{I}(\psi)_{z_0},f_0^*}(z_0).
	\end{split}
\end{displaymath}
 
Thus, the sufficiency of the characterization has been proved.

\subsection{Proof of Remark \ref{product 1-2}}
\

Assume that the three statements in Theorem \ref{thm:main1-2} hold.

We follow the discussions and notations in Step 2 in the proof of Theorem \ref{thm:main1-2}. Inequality \eqref{eq:0824a} and equality \eqref{eq:0824b} imply that 
\begin{equation}
	\label{eq:0824c}K_{\partial M,\rho}^{I,h_0}(z_0)=K_{\partial M,\rho}^{\mathcal{I}(\psi)_{z_0},f_0^*}(z_0)
\end{equation}
and 
\begin{equation}
	\label{eq:0824d}B_{\tilde\rho}^{I,h_0}(z_0)= B_{\tilde\rho}^{\mathcal{I}(\psi)_{z_0},f_0^*}(z_0).
\end{equation}
For the case $n>1$,  without loss of generality, assume that $g_j(z_j)=1$. It follows from equality \eqref{eq:0824e} and equality \eqref{eq:0824f} that 
 it suffices to prove the case that $g_j\equiv1$ for any $1\le j\le n$.

In the following, $n$ is any positive integer, and we assume that $\varphi_j=2u_j$ is harmonic on $D_j$ for any $1\le j\le n$.
By equality \eqref{eq:220802b} and equality \eqref{eq:0824b}, we have 
$$B_{\tilde\rho}^{I,h_0}(z_0)= B_{\tilde\rho}^{\mathcal{I}(\psi)_{z_0},f_0^*}(z_0)=\frac{1}{G_2(0)}=\frac{1}{\|F\|^2_{M,\tilde\rho}},$$
where $F=\sum_{\alpha\in E}\tilde d_{\alpha}\prod_{1\le j\le n}((P_j)_*(f_{z_j}))'(P_j)_*(f_{u_j}f_{z_j}^{\alpha_j}).$ By equality \eqref{eq:0804h} and equality \eqref{eq:0824a}, we have 
$$K_{\partial M,\rho}^{I,h_0}(z_0)=K_{\partial M,\rho}^{\mathcal{I}(\psi)_{z_0},f_0^*}(z_0)=\frac{1}{\|f\|^2_{\partial M,\rho}},$$
where $f\in H^2_{\rho}(M,\partial M)$ such that $f^*=F$.

Thus, Remark \ref{product 1-2} holds.

\section{Proofs of Theorem \ref{thm:main2-1} and Theorem \ref{thm:main2-3}}\label{sec:proof3}

In this section, we prove Theorem \ref{thm:main2-1} and Theorem \ref{thm:main2-3}.

\subsection{Proof of Theorem \ref{thm:main2-1}}

The inequality part of Theorem \ref{thm:main2-1} follows from Proposition \ref{p:inequality2}. In the following, we prove the characterization for holding of the equality.

It follows from inequality \eqref{eq:0810b} and inequality \eqref{eq:0810c} in the proof of Proposition \ref{p:inequality2} that 
\begin{equation}\label{eq:0810d}
	\begin{split}
			\left(\sum_{1\le j\le n}\frac{1}{p_j}\right)\frac{1}{K_{S,\lambda}(z_0)}
			&\le\sum_{1\le j\le n}\frac{1}{\pi^{n-1}p_jK_{\partial D_j,\lambda_j}(z_j) B_{M_j,\hat\rho_j}(\hat z_j)}\\
			&\le\sum_{1\le j\le n}\frac{1}{\pi^{n-1}p_j}\|f\|^2_{\partial D_j\times M_j,p_j\rho}\\
			&=\frac{1}{\pi^{n-1}}\|f\|^2_{\partial M,\rho}\\
			&=\frac{1}{\pi^{n-1}K_{\partial M,\rho}(z_0)},
	\end{split}
\end{equation}
where $f\in H^2_{\rho}(\partial M,\rho)$ satisfying $f^*(z_0)=1$, $\hat\rho_j=\prod_{1\le l\le n,\,j\not=j}e^{-\varphi_l}$, $\hat z_j=(z_1,\ldots,z_{j-1},z_{j+1},\cdots,z_n)$ and  $\lambda_j=\left(\frac{\partial G_{D_j}(w_j,z_j)}{\partial v_{w_j}}\right)^{-1}e^{-\varphi_j}$ on $\partial D_j$ for any $1\le j\le n$.
It follows from Proposition \ref{p:8}, Lemma \ref{l:Bergman-prod} and Proposition \ref{p:inequality} that 	\begin{equation}
	\label{eq:0810e}K_{S,\lambda}(z_0)=\prod_{1\le j\le n}K_{\partial D_j,\lambda_j}(z_j),
\end{equation}
\begin{equation}
	\label{eq:0810g}K_{\partial D_j,\lambda_j}(z_j)\ge\pi B_{D_j,\rho_j}(z_j)
\end{equation}
for any $1\le j\le n$, and
\begin{equation}
	\label{eq:0810f}B_{M_1,\hat\rho_j}(\hat z_j)=\prod_{1\le l\le n,\,l\not=j}B_{D_l, \rho_l}(z_l),
\end{equation}
where $\rho_l=e^{-\varphi_l}$ for any $1\le l\le n$.

Firstly, we prove the necessity of the characterization.
Assume that $K_{S,\lambda}(z_0)=\left(\sum_{1\le j\le n}\frac{1}{p_j}\right)\pi^{n-1}K_{\partial M,\rho}(z_0)$. Hence inequality \eqref{eq:0810d} becomes an equality, and it follows from equality \eqref{eq:0810e}, inequality \eqref{eq:0810g} and equality \eqref{eq:0810f} that 
\begin{equation}\nonumber
	K_{\partial D_j,\lambda_j}(z_j)=\pi B_{D_j,\rho_j}(z_j)
\end{equation}
holds for any $1\le j\le n$. Using Theorem \ref{thm:saitoh-1d}, we get that  $\varphi_j$ is harmonic  on $D_j$ for any $1\le j\le n$ and $\chi_{j,z_j}=\chi_{j,-\frac{\varphi_j}{2}}$.

Now, we prove the sufficiency of the characterization. Assume that $\varphi_j$ is harmonic  on $D_j$ for any $1\le j\le n$ and $\chi_{j,z_j}=\chi_{j,-\frac{\varphi_j}{2}}$. By Theorem \ref{thm:saitoh-1d}, we know that  
\begin{equation}\nonumber
	K_{\partial D_j,\lambda_j}(z_j)=\pi B_{D_j,\rho_j}(z_j)
\end{equation}
holds for any $1\le j\le n$,
which shows that 
\begin{equation}
	\label{eq:0809b}\left(\sum_{1\le j\le n}\frac{1}{p_j}\right)\frac{1}{K_{S,\lambda}(z_0)}
=\sum_{1\le j\le n}\frac{1}{\pi^{n-1}p_jK_{\partial D_j,\lambda_j}(z_j) B_{M_j,\hat\rho_j}(\hat z_j)}.
\end{equation}
 Remark \ref{r:product 1-1} shows that 
and 
$$f^*=c_0\prod_{1\le j\le n} ((P_j)_*(f_{z_j}))'(P_j)_*(f_{u_j}),$$
where $c_0$ is a constant, $P_j$ is the universal covering from unit disc $\Delta$ to $ D_j$, $f_{\frac{\varphi_j}{2}}$ is a holomorphic function on $\Delta$ such that $|f_{\frac{\varphi_j}{2}}|=P_j^*(e^{\frac{\varphi_j}{2}})$, and $f_{z_j}$ is a holomorphic function on $\Delta$ such that $|f_{z_j}|=P_j^*(e^{G_{D_j}(\cdot,z_j)})$ for any $1\le j\le n$. It follows from Theorem \ref{thm:saitoh-1d} and Remark \ref{rem:function} that 
\begin{equation}
\nonumber
	\begin{split}
	\frac{1}{K_{\partial D_j,\lambda_j}(z_j)}=&\frac{1}{2\pi}\int_{\partial D_j}\left|\frac{K_{\partial D_j,\lambda_j}(z,z_j)}{K_{\partial D_j,\lambda_j}(z_j)}\right|^2e^{-\varphi_j}\left(\frac{\partial G_{D_j}(z,z_j)}{\partial v_z} \right)^{-1}|dz|
		\\
		=&\frac{1}{2\pi}\int_{\partial D_j}|c_j((P_j)_*(f_{z_j}))'(P_j)_*(f_{\frac{\varphi_j}{2}})|^2e^{-\varphi_j}\left(\frac{\partial G_{D_j}(z,z_j)}{\partial v_z} \right)^{-1}|dz|
	\end{split}
\end{equation}
and 
\begin{equation}
\nonumber
	\begin{split}
		\frac{1}{B_{D_j,\rho_j}(z_j)}&=\int_{D_j}\left|\frac{B_{D_j,\rho_j}(\cdot,z_j)}{B_{D_j,\rho_j}(z_j)}\right|^2e^{-\varphi_j}\\&=\int_{D_j}|c_j((P_j)_*(f_{z_j}))'(P_j)_*(f_{\frac{\varphi_j}{2}})|^2e^{-\varphi_j}
	\end{split}
\end{equation}
for any $1\le j\le n$,
where $c_j$ is constant such that $\left(c_j((P_j)_*(f_{z_j}))'(P_j)_*(f_{\frac{\varphi_j}{2}})\right)(z_j)=1$.
Thus, we have 
\begin{equation}
	\label{eq:0809a}\frac{1}{K_{\partial D_j,\lambda_j}(z_j)B_{M_j,\hat\rho_j}}=\frac{1}{K_{\partial D_j,\lambda_j}(z_j)\prod_{1\le l\le n,\,l\not=j}B_{D_l, \rho_l}(z_l)}=\|f\|^2_{\partial D_j\times M_j,p_j\rho}.
\end{equation}
Following from inequality \eqref{eq:0810d}, equality \eqref{eq:0809b} and equality \eqref{eq:0809a}, we obtain that 
\begin{equation}\nonumber
	\begin{split}
		\left(\sum_{1\le j\le n}\frac{1}{p_j}\right)\frac{1}{K_{S,\lambda}(z_0)}
		&=\sum_{1\le j\le n}\frac{1}{\pi^{n-1}p_jK_{\partial D_j,\lambda_j}(z_j) B_{M_j,\hat\rho_j}(\hat z_j)}\\
		&=\sum_{1\le j\le n}\frac{1}{\pi^{n-1}p_j}\|f\|^2_{\partial D_j\times M_j,p_j\rho}\\
		&=\frac{1}{\pi^{n-1}}\|f\|^2_{\partial M,\rho}\\
		&=\frac{1}{\pi^{n-1}K_{\partial M,\rho}(z_0)}.
	\end{split}
\end{equation}

Thus, Theorem \ref{thm:main2-1} holds.

\subsection{Proof of Theorem \ref{thm:main2-3}}

The inequality part of Theorem \ref{thm:main2-3} follows from Proposition \ref{p:inequality3}. In the following, we prove the characterization for holding of the equality.

Following the notations in the proof of Proposition \ref{p:inequality3}.
It follows from inequality \eqref{eq:0812c}, inequality \eqref{eq:0812d} and inequality \eqref{eq:0812f} in the proof of Proposition \ref{p:inequality2} that 
\begin{equation}\label{eq:0813a}
	\begin{split}
		\left(\sum_{1\le j\le n}\frac{\tilde\beta_j+1}{p_j}\right)\frac{1}{K^{I,h_0}_{S,\lambda}(z_0)}
		&\le
		\sum_{1\le j\le n}\frac{\prod_{1\le l\le n}(\tilde\beta_l+1)}{\pi^{n-1}p_jK_{\partial D_j,\lambda_j}^{I_{\tilde\beta_j,z_j},h_j}(z_j) B_{M_j,\hat\rho_j}^{I'_j,\hat h_j}(\hat z_j)}\\
		&\le\sum_{1\le j\le n}\frac{\prod_{1\le l\le n}(\tilde\beta_l+1)}{\pi^{n-1}p_j}\|f\|^2_{\partial D_j\times M_j,p_j\rho}\\
		&=\frac{\prod_{1\le j\le n}(\tilde\beta_j+1)}{\pi^{n-1}}\|f\|^2_{\partial M,\rho}\\
		&=\frac{\prod_{1\le j\le n}(\tilde\beta_j+1)}{\pi^{n-1}K_{\partial M,\rho}^{I,h_0}(z_0)},
	\end{split}
\end{equation}
where $f\in H^2_{\rho}(\partial M,\rho)$ satisfying $(f^*-h_0,z_0)\in I$, $\hat\rho_j:=\prod_{1\le l\le n,\,l\not=j}\rho_l=\prod_{1\le l\le n,\,l\not=j}e^{-\varphi_l}$, $\hat h_j=\prod_{1\le l\le n,\,l\not=j}h_l$, $M_j=\prod_{1\le l\le n,\,l\not=j}D_l$, $\hat z_j:=(z_1,\ldots,z_{j-1},z_{j+1},\ldots,z_n)$ and 
 \begin{displaymath}
	\begin{split}
		I'_j=\Bigg\{(g,\hat z_1)\in\mathcal{O}_{M_1, \hat z_1}:g=\sum_{\alpha=(\alpha_1,\ldots,\alpha_{j-1},\alpha_j\ldots\alpha_n)\in\mathbb{Z}_{\ge0}^{n-1}}b_{\alpha}&\prod_{1\le l\le n,\,l\not=j}(w_j-z_j)^{\alpha_j}\text{ near  }\hat z_1\\
		&\text{s.t. $b_{\alpha}=0$ for any $\alpha\in L_j$}\Bigg\},
	\end{split}
\end{displaymath}
where $L_j=\{\alpha=(\alpha_1,\ldots,\alpha_{j-1},\alpha_j\ldots\alpha_n)\in\mathbb{Z}_{\ge0}^{n-1}:\alpha_l\le\tilde\beta_l$ for any $l\not=j\}$.

It follows from Proposition \ref{p:9}, Lemma \ref{l:Bergman-prod2} and Proposition \ref{p:inequality} that 	\begin{equation}
	\label{eq:0813b}K^{I,h_0}_{S,\lambda}(z_0)=\prod_{1\le j\le n}K_{\partial D_j,\lambda_j}^{I_{\tilde\beta_j,z_j},h_j}(z_j),
\end{equation}
\begin{equation}
	\label{eq:0813c}(\tilde\beta_j+1)K^{I_{\tilde\beta_j,z_j},h_j}_{\partial D_j,\lambda_j}(z_j)\ge \pi B^{I_{\tilde\beta_j,z_j},h_j}_{D_j,\rho_j}(z_j)
\end{equation}
for any $1\le j\le n$, and
\begin{equation}
	\label{eq:0813d}B_{M_j,\hat\rho_j}^{I'_j,\hat h_j}(\hat z_j)=\prod_{1\le l\le n,\,l\not=j}B^{I_{\tilde\beta_l,z_l},h_l}_{D_l,\rho_l}(z_l),
\end{equation}
where $\rho_j=e^{-\varphi_j}$ and $\lambda_j=\left(\frac{\partial G_{D_j}(z,z_j)}{\partial v_z}\right)^{-1}\rho_j$ for any $1\le j\le n$.

Firstly, we prove the necessity of the characterization.
Assume that 
$$\left(\prod_{1\le j\le n}(\tilde\beta_j+1)\right)K_{S,\lambda}^{I,h_0}(z_0)=\left(\sum_{1\le j\le n}\frac{\tilde\beta_j+1}{p_j}\right)\pi^{n-1} K^{I,h_0}_{\partial M,\rho}(z_0)>0.$$ 
Hence inequality \eqref{eq:0813a} becomes an equality, and it follows from equality \eqref{eq:0813b}, inequality \eqref{eq:0813c} and equality \eqref{eq:0813d} that 
\begin{equation}\nonumber
	(\tilde\beta_j+1)K^{I_{\tilde\beta_j,z_j},h_j}_{\partial D_j,\lambda_j}(z_j)=\pi B^{I_{\tilde\beta_j,z_j},h_j}_{D_j,\rho_j}(z_j)
\end{equation}
holds for any $1\le j\le n$. Using the case $n=1$ of Theorem \ref{thm:main1-2} (replacing $I$ and $\psi$ by $I_{\tilde\beta_j,z_j}$ and $2(\tilde\beta_j+1)G_{D_1}(\cdot,z_j)$, respectively), we get that $\beta_j=\tilde\beta_j$, $\varphi_j$ is harmonic  on $D_j$ and $\chi_{j,z_j}^{\beta_j+1}=\chi_{j,-\frac{\varphi_j}{2}}$  for any $1\le j\le n$.

Now, we prove the sufficiency of the characterization. Assume that $\beta_j=\tilde\beta_j$, $\varphi_j$ is harmonic  on $D_j$ and $\chi_{j,z_j}^{\beta_j+1}=\chi_{j,-\frac{\varphi_j}{2}}$  for any $1\le j\le n$. By Theorem \ref{thm:main1-2}, we know that  
\begin{equation}\nonumber
	(\tilde\beta_j+1)K^{I_{\tilde\beta_j,z_j},h_j}_{\partial D_j,\lambda_j}(z_j)=\pi B^{I_{\tilde\beta_j,z_j},h_j}_{D_j,\rho_j}(z_j)
\end{equation}
holds for any $1\le j\le n$,
which shows that 
\begin{equation}
	\label{eq:0813e}	\left(\sum_{1\le j\le n}\frac{\tilde\beta_j+1}{p_j}\right)\frac{1}{K^{I,h_0}_{S,\lambda}(z_0)}
	=
	\sum_{1\le j\le n}\frac{\prod_{1\le l\le n}(\tilde\beta_l+1)}{\pi^{n-1}p_jK_{\partial D_j,\lambda_j}^{I_{\tilde\beta_j,z_j},h_j}(z_j) B_{M_j,\hat\rho_j}^{I'_j,\hat h_j}(\hat z_j)}
\end{equation}
by using equality \eqref{eq:0813b} and equality \eqref{eq:0813d}.
Denote that 
$$g_{j}:=c_j((P_j)_*(f_{z_j}))'(P_j)_*(f_{\frac{\varphi_j}{2}}f_{z_j}^{\beta_j})$$
on $D_j$,
where $c_j$ is constant such that $(g_j-h_j,z_j)\in I_{\beta_j,z_j}$, $P_j$ is the universal covering from unit disc $\Delta$ to $ D_j$, $f_{\frac{\varphi_j}{2}}$ is a holomorphic function on $\Delta$ such that $|f_{\frac{\varphi_j}{2}}|=P_j^*(e^{\frac{\varphi_j}{2}})$, and $f_{z_j}$ is a holomorphic function on $\Delta$ such that $|f_{z_j}|=P_j^*(e^{G_{D_j}(\cdot,z_j)})$ for any $1\le j\le n$. 
 It follows from the case $n=1$ of Theorem \ref{thm:main1-2} and Remark \ref{product 1-2} that there is $\tilde g_j\in H^2_{\lambda_j}(D_j,\partial D_j)$ such that $\tilde g_j^*=g_j$ and
\begin{equation}
	\nonumber
	\begin{split}
	\frac{1}{2\pi}\int_{\partial D_j}|\tilde g_j|^2e^{-\varphi_j}\left(\frac{\partial G_{D_j}(z,z_j)}{\partial v_z} \right)^{-1}|dz|&=\frac{1}{K_{\partial D_j,\lambda_j}^{I_{\beta_j,z_j},h_j}(z_j)}\\
	&=\frac{\tilde\beta_j+1}{B_{D_j,\rho_j}^{I_{\beta_j,z_j},h_j}(z_j)}\\
	&=(\tilde\beta_j+1)\int_{D_j}|g_j|^2e^{-\varphi_j}
	\end{split}
\end{equation}
for any $1\le j\le n$.
Thus, there is $\tilde f\in H^2_{\rho}(M,\partial M)$ such that $\tilde f^*=\prod_{1\le j\le n}g_j$ and 
\begin{equation}
	\nonumber
	\begin{split}
\frac{1}{K_{\partial D_j,\lambda_j}^{I_{\tilde\beta_j,z_j},h_j}(z_j) B_{M_j,\hat\rho_j}^{I'_j,\hat h_j}(\hat z_j)}
&=\frac{1}{K_{\partial D_j,\lambda_j}^{I_{\tilde\beta_j,z_j},h_j}(z_j) \prod_{1\le l\le n,\,l\not=j}B_{D_l,\rho_l}^{I_{\beta_l,z_l},h_l}(z_l)}
\\&=\|\tilde f\|^2_{\partial D_j\times M_j,p_j\rho},
	\end{split}
\end{equation}
which shows that 
\begin{equation}
	\label{eq:0813f}
	\begin{split}
	\sum_{1\le j\le n}\frac{1}{p_jK_{\partial D_j,\lambda_j}^{I_{\tilde\beta_j,z_j},h_j}(z_j) B_{M_j,\hat\rho_j}^{I'_j,\hat h_j}(\hat z_j)}
&=\sum_{1\le j\le n}\frac{1}{p_j}\|\tilde f\|^2_{\partial D_j\times M_j,p_j\rho}\\
&=\|\tilde f\|^2_{\partial M,\rho}\\
&\ge\frac{1}{K_{\partial M,\rho}^{I,h_0}(z_0)}.
	\end{split}
\end{equation}
Following from inequality \eqref{eq:0813a}, equality \eqref{eq:0813e} and equality \eqref{eq:0813f}, we obtain that 
\begin{equation}\nonumber
	\begin{split}
			\left(\sum_{1\le j\le n}\frac{\tilde\beta_j+1}{p_j}\right)\frac{1}{K^{I,h_0}_{S,\lambda}(z_0)}
		&=
		\sum_{1\le j\le n}\frac{\prod_{1\le l\le n}(\tilde\beta_l+1)}{\pi^{n-1}p_jK_{\partial D_j,\lambda_j}^{I_{\tilde\beta_j,z_j},h_j}(z_j) B_{M_j,\hat\rho_j}^{I'_j,\hat h_j}(\hat z_j)}\\
		&=\frac{\prod_{1\le j\le n}(\tilde\beta_j+1)}{\pi^{n-1}K_{\partial M,\rho}^{I,h_0}(z_0)}.
	\end{split}
\end{equation}

Thus, Theorem \ref{thm:main2-3} holds.

\section{appendix}\label{sec:app}

In this section, we do further discussion on the relations between $K_{S,\lambda}^{I,h_0}(z_0)$, $K_{\partial M,\rho}^{I,h_0}(z_0)$ and $B_{\tilde\rho}^{I,h_0}(z_0)$.

Let $z_0=(z_1,\ldots,z_n)\in M,$ and let  
$$\psi(w_1,\ldots,w_n)=\max_{1\le j\le n}\{2p_jG_{D_j}(w_j,z_j)\}$$ be a plurisubharmonic function on $M$, where $G_{D_j}(\cdot,z_j)$ is the Green function on $D_j$ and $p_j>0$ is a constant for any $1\le j\le n$.

Let $\varphi_j$ be a subharmonic function on $D_j$, which satisfies that  $\varphi_j$ is continuous at $z_j$ for any $z_j\in \partial D_j$. 
 Let $\rho$ be a Lebesgue measurable function on $\partial M$ such that
$$\rho(w_1,\ldots,w_n):=\frac{1}{p_j}\left(\frac{\partial G_{D_j}(w_j,z_j)}{\partial v_{w_j}}\right)^{-1}\times\prod_{1\le l\le n}e^{-\varphi_l(w_l)}$$
on $\partial D_j\times {M_j}$, thus we have $\inf_{\partial M}\rho>0$.
 Let 
 $$\lambda_j(w_j):=\left(\frac{\partial G_{D_j}(w_j,z_j)}{\partial v_{w_j}}\right)^{-1}e^{-\varphi_j(w_j)}$$
 on $\partial D_j$ for any $1\le j\le n$, and denote that
$$\lambda:=\prod_{1\le j\le n}\lambda_j$$
on $S=\prod_{1\le j\le n}\partial D_j$, thus $\lambda$ is continuous on $S$. 

Denote that $E_1:=\{\alpha\in\mathbb{Z}_{\ge0}^n:\sum_{1\le j\le n}\frac{\alpha_j+1}{p_j}\leq 1\}$. 
We call $\alpha>\beta$ for any $\alpha,\beta\in \mathbb{Z}_{\ge0}^n$ if $\alpha_j\ge \beta_j$ for any $1\le j\le n$ but $\alpha\not=\beta$.
Let $L\not=\emptyset$ be a subset of $E_1$ satisfying that if $\alpha\in L$ then $\beta\not\in E_1$ for any $\beta >\alpha$.
Let 
$$h_0(w)=\sum_{\alpha\in L}d_{\alpha}(w-z_0)^{\alpha}$$ be a holomorphic function on $M$, where $d_{\alpha}\not=0$ is a constant for any $\alpha\in L$.

\begin{Proposition}
	\label{p:app1}Assume that $H^2_{\lambda}(M,S)\not=\emptyset$, then we have $$\frac{1}{K^{\mathcal{I}(\psi)_{z_0},h_0}_{S,\lambda}(z_0)}=\sum_{\alpha\in L}\frac{|d_{\alpha}|^2}{K^{\mathcal{I}(\psi)_{z_0},(w-z_0)^{\alpha}}_{S,\lambda}(z_0)}.$$
	 Furthermore, assume that $K^{\mathcal{I}(\psi)_{z_0},h_0}_{\partial M,\rho}(z_0)>0$, then inequality 
	 \begin{equation}\label{eq:0814g}
	 \sum_{\alpha\in L}\frac{|d_{\alpha}|^2c_{\alpha}}{K^{\mathcal{I}(\psi)_{z_0},(w-z_0)^{\alpha}}_{S,\lambda}(z_0)}\le \frac{1}{\pi^{n-1}K^{\mathcal{I}(\psi)_{z_0},h_0}_{\partial M,\rho}(z_0)}
	 \end{equation}
	holds, where $c_{\alpha}=\frac{\sum_{1\le j\le n}\frac{\alpha_j+1}{p_j}}{\prod_{1\le j\le n}(\alpha_j+1)}$, and equality holds if and only if the following statements hold:
	
		$(1)$ $\varphi_j$ is harmonic  on $D_j$ for any $1\le j\le n$;
	
	$(2)$ $\chi_{j,z_j}^{\alpha_j+1}=\chi_{j,-\frac{\varphi_j}{2}}$ for any $\alpha=(\alpha_1,\ldots,\alpha_n)\in L$ and $1\le j\le n$, where $\chi_{j,-\frac{\varphi_j}{2}}$ and $\chi_{j,z_j}$ are the  characters associated to the functions $-\frac{\varphi_j}{2}$ and $G_{D_j}(\cdot,z_j)$ respectively.
\end{Proposition}
\begin{proof}
 For any $\alpha\in L$ and any $1\le j\le n$, it follows from Proposition \ref{p:7} and Lemma \ref{l:a6} that there is $f_{j,\alpha_j}\in H^2_{\lambda_j}(D_j,\partial D_j)$ such that $(f_{j,\alpha_j}-(w_j-z_j)^{\alpha_j},z_j)\in I_{\alpha_j,z_j}$ and 
 $$\ll f_{j,\alpha_j},g\gg_{\partial D_j,\lambda_j}=0$$
  for any $g\in H^2_{\lambda_j}(D_j,\partial D_j)$ satisfying $ord_{z_j}(g)>\alpha_j$, where $I_{\alpha_j,z_j}=((w_j-z_j)^{\alpha_j+1})$ is an ideal of $\mathcal{O}_{D_j,z_j}$. Thus, for any $\alpha\in L$,  we have 
 \begin{equation}
 	\label{eq:0814a}
 	\ll\prod_{1\le j\le n}f_{j,\alpha_j},g\gg_{S,\lambda}=0
 \end{equation}
 for any $g\in H^2_{\lambda}(M,S)$ satisfying $(g,z_0)\in \mathcal{I}(\psi)_{z_0}$ by using Proposition \ref{p:7}, which shows that 
 \begin{equation}
 	\label{eq:0814b}
 	\ll\sum_{\alpha\in L}d_{\alpha}\prod_{1\le j\le n}f_{j,\alpha_j},g\gg_{S,\lambda}=0
 \end{equation}
  for any $g\in H^2_{\lambda}(M,S)$ satisfying $(g,z_0)\in \mathcal{I}(\psi)_{z_0}$. By equality \eqref{eq:0814a}  equality \eqref{eq:0814b}, we have
  \begin{equation}
  	\label{eq:0814f}\frac{1}{K^{\mathcal{I}(\psi)_{z_0},(w-z_0)^{\alpha}}_{S,\lambda}(z_0)}=\|\prod_{1\le j\le n}f_{j,\alpha_j}\|^2_{M,S}=\prod_{1\le j\le n}\frac{1}{K_{\partial D_j,\lambda_j}^{I_{\alpha_j,z_j},(w_j-z_j)^{\alpha_j}}(z_j)}.
  \end{equation}
   and 
  \begin{equation}\nonumber
  	\begin{split}
  	\frac{1}{K^{\mathcal{I}(\psi)_{z_0},h_0}_{S,\lambda}(z_0)}&=\|\sum_{\alpha\in L}d_{\alpha}\prod_{1\le j\le n}f_{j,\alpha_j}\|^2_{S,\lambda}\\
  	&=\sum_{\alpha\in L}|d_{\alpha}|^2\|\prod_{1\le j\le n}f_{j,\alpha_j}\|^2_{S,\lambda}\\
  	&=\sum_{\alpha\in L}\frac{|d_{\alpha}|^2}{K^{\mathcal{I}(\psi)_{z_0},(w-z_0)^{\alpha}}_{S,\lambda}(z_0)}.
  	\end{split}
  \end{equation}
 
 It follows from Lemma \ref{l:b7} that there is $f\in H^2_{\rho}(M,\partial M)$ such that $(f-h_0,z_0)\in \mathcal{I}(\psi)_{z_0}$ and 
 \begin{equation}
 	\label{eq:0814c}
 	\frac{1}{K^{\mathcal{I}(\psi)_{z_0},h_0}_{\partial M,\rho}(z_0)}=\|f\|^2_{\partial M,\rho}.
 \end{equation}
 Denote that 
 $$\tilde\psi_{\gamma,j}:=\max_{1\le l\le n,\,l\not=j}\{2p_l(1-\frac{\gamma+1}{p_j})\log|w_l-z_l|\}$$
is a plurisubharmonic function on $M_j$, where $M_j=\prod_{1\le l\le n,\,l\not=j}D_l$.

Denote that $\hat\rho_j:=\prod_{1\le l\le n,\,l\not=j}\rho_l$ and $\hat h_{j,\alpha}:=\prod_{1\le l\le n,\,l\not=j}(w_l-z_l)^{\alpha_l}$ for any $1\le j\le n$, where $\rho_j=e^{-\varphi_j}$ on $D_j$. It follows from Lemma \ref{l:prod-d1xm1} and Lemma \ref{l:a6} that 
\begin{equation}
	\label{eq:0814d}f=\sum_{m=0}^{+\infty}f_{1,m}g_{m},
\end{equation}
where $g_m\in A^2(M_1,\hat\rho_1)$ and $f_{1,m}\in H^2_{\lambda_1}(D_1,\partial D_1)$ such that $ord_{z_1}(f_{1,m}-(w_1-z_1)^m)>m$ and $\ll f_{1,m},g\gg_{\partial D_1,\lambda_1}=0$ for any $g\in H^2_{\lambda_1}(D_1,\partial D_1) $ satisfying $ord_{z_1}(g)>m$. Note that $\{(w-z_0)^{\beta}:\exists \alpha\in L$ s.t. $\beta>\alpha\}\subset \mathcal{I}(\psi)_{z_0}$ and $(w-z_0)^{\alpha}\not\in \mathcal{I}(\psi)_{z_0}$ for any $\alpha\in L$. As $(f-h_0,z_0)\in \mathcal{I}(\psi)_{z_0}$, we know that 
$$(g_{\alpha_1}-d_{\alpha}\hat h_{1,\alpha},\hat z_1)\in\mathcal{I}(\tilde\psi_{\alpha_1,1})_{\hat z_1}$$ for any $\alpha\in L,$
where $\hat z_j=(z_1,\ldots,z_{l-1},z_{l+1},\ldots,z_n)$. By equality \eqref{eq:0814d}, we have 
\begin{displaymath}
	\begin{split}
		\|f\|^2_{\partial D_1\times M_1,p_1\rho}&\ge\sum_{\alpha\in L}\|f_{1,\alpha_1}\|^2_{\partial D_1,\lambda_1}\times\|g_{\alpha_1}\|^2_{M_1,\hat\rho_1}\\
		&\ge\sum_{\alpha\in L}\frac{|d_{\alpha}|^2}{K_{\partial D_1,\lambda_1}^{I_{\alpha_1,z_1},(w_1-z_1)^{\alpha_1}}(z_1)\cdot B_{M_1,\hat\rho_1}^{\mathcal{I}(\tilde\psi_{\alpha_1})_{\hat z_1},\hat h_{1,\alpha}}(\hat z_1)}.
	\end{split}
\end{displaymath}
Similarly, we have 
\begin{equation}\nonumber
	\|f\|^2_{\partial D_j\times M_j,p_j\rho}\ge\sum_{\alpha\in L}\frac{|d_{\alpha}|^2}{K_{\partial D_j,\lambda_j}^{I_{\alpha_j,z_j},(w_j-z_j)^{\alpha_j}}(z_j)\cdot B_{M_j,\hat\rho_j}^{\mathcal{I}(\tilde\psi_{\alpha_j})_{\hat z_j},\hat h_{j,\alpha}}(\hat z_j)}.
\end{equation}
By Lemma \ref{l:Bergman-prod}, Theorem \ref{thm:main1-2} and equality \eqref{eq:0814f}, we obtain that 
\begin{equation}
	\label{eq:0814e}
	\begin{split}
			\frac{1}{K^{\mathcal{I}(\psi)_{z_0},h_0}_{\partial M,\rho}(z_0)}=&\|f\|^2_{\partial M,\rho}
	\\	=&\sum_{1\le j\le n}\frac{1}{p_j}	\|f\|^2_{\partial D_j\times M_j,p_j\rho}\\
		\ge&\sum_{1\le j\le n}\frac{1}{p_j}	\sum_{\alpha\in L}\frac{|d_{\alpha}|^2}{K_{\partial D_j,\lambda_j}^{I_{\alpha_j,z_j},(w_j-z_j)^{\alpha_j}}(z_j)\cdot B_{M_j,\hat\rho_j}^{\mathcal{I}(\tilde\psi_{\alpha_j})_{\hat z_j},\hat h_{j,\alpha}}(\hat z_j)}\\
		=&\sum_{1\le j\le n}\frac{1}{p_j}	\sum_{\alpha\in L}\frac{|d_{\alpha}|^2}{K_{\partial D_j,\lambda_j}^{I_{\alpha_j,z_j},(w_j-z_j)^{\alpha_j}}(z_j)\cdot \prod_{1\le l\le n,\,l\not=j}B_{D_l,\rho_l}^{I_{\alpha_l,z_l},(w_l-z_l)^{\alpha_l}}(z_l)}\\
		\ge&\sum_{1\le j\le n}\sum_{\alpha\in L}\frac{|d_{\alpha}|^2\frac{\alpha_j+1}{p_j}\pi^{n-1}}{\prod_{1\le l\le n}(\alpha_l+1)K_{\partial D_l,\lambda_l}^{I_{\alpha_l,z_l},(w_l-z_l)^{\alpha_l}}(z_l)}\\
		=&\sum_{\alpha\in L}\frac{|d_{\alpha}|^2c_{\alpha}\pi^{n-1}}{K^{\mathcal{I}(\psi)_{z_0},(w-z_0)^{\alpha}}_{S,\lambda}(z_0)}.
	\end{split}
\end{equation}

In the following part, we prove the characterization of the holding of inequality becoming an equality.

Firstly, we prove the necessity. Assume that 
$$\frac{\pi^{n-1}}{K^{\mathcal{I}(\psi)_{z_0},h_0}_{\partial M,\rho}(z_0)}
=\sum_{\alpha\in L}\frac{|d_{\alpha}|^2c_{\alpha}}{K^{\mathcal{I}(\psi)_{z_0},(w-z_0)^{\alpha}}_{S,\lambda}(z_0)}$$ holds. Following from inequality \eqref{eq:0814e}, we get that 
$$K_{\partial D_l,\lambda_l}^{I_{\alpha_l,z_l},(w_l-z_l)^{\alpha_l}}(z_l)=\pi B_{D_l,\rho_l}^{I_{\alpha_l,z_l},(w_l-z_l)^{\alpha_l}}(z_l)$$
holds  for any $\alpha=(\alpha_1,\ldots,\alpha_n)\in L$ and $1\le j\le n$. Using Theorem \ref{thm:main1-2}, we obtain that $\varphi_j$ is harmonic  on $D_j$ for any $1\le j\le n$ and  $\chi_{l,z_l}^{\alpha_l+1}=\chi_{l,-\frac{\varphi_l}{2}}$ for any $\alpha=(\alpha_1,\ldots,\alpha_n)\in L$ and $1\le l\le n$.

Now, we prove the sufficiency. Assume that $\varphi_j$ is harmonic  on $D_j$ for any $1\le j\le n$ and $\chi_{j,z_j}^{\alpha_j+1}=\chi_{j,-\frac{\varphi_j}{2}}$ for any $\alpha=(\alpha_1,\ldots,\alpha_n)\in L$ and $1\le j\le n$. Using Theorem \ref{thm:main1-2}, we obtain that
$$K_{\partial D_l,\lambda_l}^{I_{\alpha_l,z_l},(w_l-z_l)^{\alpha_l}}(z_l)=\pi B_{D_l,\rho_l}^{I_{\alpha_l,z_l},(w_l-z_l)^{\alpha_l}}(z_l)$$
holds  for any $\alpha=(\alpha_1,\ldots,\alpha_n)\in L$ and $1\le j\le n$, which implies that
\begin{equation}
	\label{eq:220814a}
	\begin{split}
	&\sum_{1\le j\le n}\frac{1}{p_j}	\sum_{\alpha\in L}\frac{|d_{\alpha}|^2}{K_{\partial D_j,\lambda_j}^{I_{\alpha_j,z_j},(w_j-z_j)^{\alpha_j}}(z_j)\cdot \prod_{1\le l\le n,\,l\not=j}B_{D_l,\rho_l}^{I_{\alpha_l,z_l},(w_l-z_l)^{\alpha_l}}(z_l)}\\
	=&\sum_{1\le j\le n}\sum_{\alpha\in L}\frac{|d_{\alpha}|^2\frac{\alpha_j+1}{p_j}\pi^{n-1}}{\prod_{1\le l\le n}(\alpha_l+1)K_{\partial D_l,\lambda_l}^{I_{\alpha_l,z_l},(w_l-z_l)^{\alpha_l}}(z_l)}.
	\end{split}
\end{equation}
For any $\alpha\in L$ and any $1\le j\le n$, denote that 
$$g_{j,\alpha_j}:=c_{j,\alpha_j}((P_j)_*(f_{z_j}))'(P_j)_*(f_{\frac{\varphi_j}{2}}f_{z_j}^{\alpha_j})$$
on $D_j$,
where $c_{j,\alpha_j}$ is constant such that $(g_{j,\alpha_j}-(w_j-z_j)^{\alpha_j},z_j)\in I_{\alpha_j,z_j}$, $P_j$ is the universal covering from unit disc $\Delta$ to $ D_j$, $f_{\frac{\varphi_j}{2}}$ is a holomorphic function on $\Delta$ such that $|f_{\frac{\varphi_j}{2}}|=P_j^*(e^{\frac{\varphi_j}{2}})$, and $f_{z_j}$ is a holomorphic function on $\Delta$ such that $|f_{z_j}|=P_j^*(e^{G_{D_j}(\cdot,z_j)})$. 
It follows from the case $n=1$ of Theorem \ref{thm:main1-2} and Remark \ref{product 1-2} that there is $\tilde g_{j,\alpha_j}\in H^2_{\lambda_j}(D_j,\partial D_j)$ such that $\tilde g_{j,\alpha_j}^*=g_{j,\alpha_j}$ and
\begin{equation}
	\nonumber
	\begin{split}
		\frac{1}{2\pi}\int_{\partial D_j}|\tilde g_{j,\alpha_j}|^2e^{-\varphi_j}\left(\frac{\partial G_{D_j}(z,z_j)}{\partial v_z} \right)^{-1}|dz|&=\frac{1}{K_{\partial D_j,\lambda_j}^{I_{\alpha_j,z_j},(w_j-z_j)^{\alpha_j}}(z_j)}\\
		&=\frac{\alpha_j+1}{B_{D_j,\rho_j}^{I_{\alpha_j,z_j},(w_j-z_j)^{\alpha_j}}(z_j)}\\
		&=(\alpha_j+1)\int_{D_j}|g_{j,\alpha_j}|^2e^{-\varphi_j}
	\end{split}
\end{equation}
for any $\alpha\in L$ and any $1\le j\le n$.
Thus, for any $\alpha\in L$, there is $\tilde f_{\alpha}\in H^2_{\rho}(M,\partial M)$ such that $\tilde f_{\alpha}^*=\prod_{1\le j\le n}g_{j,\alpha_j}$, $\|\sum_{\alpha\in L}d_{\alpha}\tilde f_{\alpha}\|^2_{\partial D_j\times M_j,p_j\rho}=\sum_{\alpha\in L}|d_{\alpha}|^2\|\tilde f_{\alpha}\|^2_{\partial D_j\times M_j,p_j\rho}$ and 
\begin{equation}
	\nonumber
	\begin{split}
	&	\frac{1}{K_{\partial D_j,\lambda_j}^{I_{\alpha_j,z_j},(w_j-z_j)^{\alpha_j}}(z_j)\cdot B_{M_j,\hat\rho_j}^{\mathcal{I}(\tilde\psi_{\alpha_j})_{\hat z_j},\hat h_{j,\alpha}}(\hat z_j)}\\
		=&
\frac{1}{K_{\partial D_j,\lambda_j}^{I_{\alpha_j,z_j},(w_j-z_j)^{\alpha_j}}(z_j)\cdot \prod_{1\le l\le n,\,l\not=j}B_{D_l,\rho_l}^{I_{\alpha_l,z_l},(w_l-z_l)^{\alpha_l}}(z_l)}
		\\=&\|\tilde f_{\alpha}\|^2_{\partial D_j\times M_j,p_j\rho},
	\end{split}
\end{equation}
which shows that 
\begin{equation}
	\label{eq:0815a}
	\begin{split}
		&\sum_{1\le j\le n}\frac{1}{p_j}	\sum_{\alpha\in L}\frac{|d_{\alpha}|^2}{K_{\partial D_j,\lambda_j}^{I_{\alpha_j,z_j},(w_j-z_j)^{\alpha_j}}(z_j)\cdot B_{M_j,\hat\rho_j}^{\mathcal{I}(\tilde\psi_{\alpha_j})_{\hat z_j},\hat h_{j,\alpha}}(\hat z_j)}
		\\
		=&\sum_{1\le j\le n}\frac{1}{p_j}	\sum_{\alpha\in L}|d_{\alpha}|^2\|\tilde f_{\alpha}\|^2_{\partial D_j\times M_j,p_j\rho}\\
		=&\|\sum_{\alpha\in L}d_{\alpha}\tilde f_{\alpha}\|^2_{\partial M,\rho}\\
		\ge&\frac{1}{K_{\partial M,\rho}^{\mathcal{I}(\psi)_{z_0},h_0}(z_0)},
	\end{split}
\end{equation}
since $(\sum_{\alpha\in L}d_{\alpha}\tilde f^*_{\alpha}-h_0,z_0)\in \mathcal{I}(\psi)_{z_0}$. Following from inequality \eqref{eq:0814e}, equality \eqref{eq:220814a} and equality \eqref{eq:0815a}, we obtain that 
\begin{equation}\nonumber
\begin{split}
	\frac{1}{K^{\mathcal{I}(\psi)_{z_0},h_0}_{\partial M,\rho}(z_0)}=&\sum_{1\le j\le n}\frac{1}{p_j}	\sum_{\alpha\in L}\frac{|d_{\alpha}|^2}{K_{\partial D_j,\lambda_j}^{I_{\alpha_j,z_j},(w_j-z_j)^{\alpha_j}}(z_j)\cdot B_{M_j,\hat\rho_j}^{\mathcal{I}(\tilde\psi_{\alpha_j})_{\hat z_j},\hat h_{j,\alpha}}(\hat z_j)}\\
	=&\sum_{1\le j\le n}\frac{1}{p_j}	\sum_{\alpha\in L}\frac{|d_{\alpha}|^2}{K_{\partial D_j,\lambda_j}^{I_{\alpha_j,z_j},(w_j-z_j)^{\alpha_j}}(z_j)\cdot \prod_{1\le l\le n,\,l\not=j}B_{D_l,\rho_l}^{I_{\alpha_l,z_l},(w_l-z_l)^{\alpha_l}}(z_l)}\\
=&\sum_{1\le j\le n}\sum_{\alpha\in L}\frac{|d_{\alpha}|^2\frac{\alpha_j+1}{p_j}\pi^{n-1}}{\prod_{1\le l\le n}(\alpha_l+1)K_{\partial D_l,\lambda_l}^{I_{\alpha_l,z_l},(w_l-z_l)^{\alpha_l}}(z_l)}\\
	=&\sum_{\alpha\in L}\frac{|d_{\alpha}|^2c_{\alpha}\pi^{n-1}}{K^{\mathcal{I}(\psi)_{z_0},(w-z_0)^{\alpha}}_{S,\lambda}(z_0)}.
\end{split}
\end{equation}

Thus, Proposition \ref{p:app1} has been proved.
\end{proof}

Let $c$ be a positive function on $[0,+\infty)$, which satisfies that $c(t)e^{-t}$ is decreasing on $[0,+\infty)$, $\lim_{t\rightarrow0+0}c(t)=c(0)=1$ and $\int_{0}^{+\infty}c(t)e^{-t}dt<+\infty$.
Denote that 
$$ \tilde\rho:=c(-\psi)\prod_{1\le j\le n}e^{-\varphi_j}$$
on $M$.
Proposition \ref{p:app1} and Theorem \ref{thm:main1-2} imply the following corollary.

\begin{Corollary}\label{c:app2} Assume that $B_{\tilde\rho}^{\mathcal{I}(\psi)_{z_0},h_0}(z_0)>0$ and $\sum_{1\le j\le n}\frac{\alpha_j+1}{p_j}\le 1$ for any $\alpha\in L$. 
	Inequality 
	$$\left(\int_0^{+\infty}c(t)e^{-t}dt\right)\sum_{\alpha\in L}\frac{|d_{\alpha}|^2c_{\alpha}}{K^{\mathcal{I}(\psi)_{z_0},(w-z_0)^{\alpha}}_{S,\lambda}(z_0)}\le \frac{1}{\pi^nB_{\tilde\rho}^{\mathcal{I}(\psi)_{z_0},h_0}(z_0)}$$
	holds, where $c_{\alpha}=\frac{\sum_{1\le j\le n}\frac{\alpha_j+1}{p_j}}{\prod_{1\le j\le n}(\alpha_j+1)}$, and equality holds if and only if the following statements hold:
	
	$(1)$ $\sum_{1\le j\le n}\frac{\alpha_j+1}{p_j}= 1$ for any $\alpha\in L$;
	
	$(2)$ $\varphi_j$ is harmonic  on $D_j$ for any $1\le j\le n$;
	
	$(3)$ $\chi_{j,z_j}=\chi_{j,-\frac{\varphi_j}{2}}$, where $\chi_{j,-\frac{\varphi_j}{2}}$ and $\chi_{j,z_j}$ are the  characters associated to the functions $-\frac{\varphi_j}{2}$ and $G_{D_j}(\cdot,z_j)$ respectively.
\end{Corollary}

 Denote that 
$$I=\left\{(g,z_0)\in\mathcal{O}_{M,z_0}:g=\sum_{\alpha\in\mathbb{Z}_{\ge0}^n}b_{\alpha}(w-z_0)^{\alpha}\text{ near  }z_0\text{ s.t. $b_{\alpha}=0$ for $\alpha\in L_{\tilde\beta}$}\right\}$$
is a proper ideal of $\mathcal{O}_{M,z_0}$,
where $L_{\beta}=\{\alpha\in\mathbb{Z}_{\ge0}^n:\alpha_j\le \beta_j$ for any $1\le j\le n\}$. Denote that 
\begin{displaymath}
	\begin{split}
I'=\Bigg\{(g,\hat z_1)\in\mathcal{O}_{M_1,\hat z_1}:g=\sum_{\tilde\alpha=(\tilde \alpha_2,\ldots,\tilde\alpha_n)\in\mathbb{Z}_{\ge0}^{n-1}}&b_{\tilde \alpha}\prod_{2\le j\le n}(w_j-z_j)^{\tilde \alpha_j}\text{ near  }\hat z_1\\
&\text{ s.t. $b_{\tilde \alpha}=0$ for $\tilde\alpha\in L'_{\beta}$}\Bigg\}
	\end{split}
\end{displaymath}
 is a proper ideal of $\mathcal{O}_{M_1,\hat z_1}$,
 where $M_1=\prod_{2\le j\le n}D_j,$ $\hat z_1=(z_2,\ldots,z_n)$ and $L'_{\beta}=\{\tilde\alpha=(\tilde\alpha_2,\ldots,\tilde\alpha_n)\in\mathbb{Z}_{\ge0}^{n-1}:\tilde\alpha_j\le \beta_j$ for any $2\le j\le n\}$.
 
Let $h_0=\sum_{\alpha\in E_0}d_{\alpha}(w-z_0)^{\alpha}$ be a holomorphic function on a neighborhood of $z_0$ such that $(h_0,z_0)\not\in I$ , i.e. $h_0\not\equiv0$, where $E_0=\{\alpha\in\mathbb{Z}_{\ge0}^n:\alpha_j\le\beta_j$ for any $1\le j\le n\}$.
Let $p_j$ be a positive constant for any $1\le j\le n$, and
denote that 
$$\psi:=\max_{1\le j\le n}\{2p_j\log|w_j-z_j|\}.$$ 
Let $\varphi_j$, $\lambda_j$, $\lambda$ and $\rho$ be as in Proposition \ref{p:app1} and Corollary \ref{c:app2} .

\begin{Proposition}
	\label{p:app3} Assume that $n>1$ and $K_{\partial M,\rho}^{I,h_0}(z_0)>0$. Then 
	inequality
	\begin{equation}\nonumber
		\left(\prod_{1\le j\le n}(\beta_j+1)\right)K_{S,\lambda}^{I,h_0}(z_0)\ge \left(\sum_{1\le j\le n}\frac{\beta_j+1}{p_j}\right)\pi^{n-1}K_{\partial M,\rho}^{I,h_0}(z_0) \tag{A}
	\end{equation}
holds,  equality holds if and only if the following statements hold:

$(1)$ $h_0=d_{\beta}(w-z_0)^{\beta}$ near $z_0$;

$(2)$ $\varphi_j$ is harmonic on $D_j$ for any $1\le j\le n$;

$(3)$ $\chi_{j,z_j}^{\beta_j+1}=\chi_{j,-\frac{\varphi_j}{2}}$  for any $1\le j\le n$, where $\chi_{j,z_j}$ and $\chi_{j,-u_j}$ are the characters associated to functions $G_{\Omega_j}(\cdot,z_j)$ and $-u_j$ respectively. 
\end{Proposition}
\begin{proof} 
	We consider the following inequality 
		\begin{equation}\nonumber
		\left(\prod_{1\le j\le n}(\beta_j+1)\right)K_{S,\lambda}^{I,h_0}(z_0)\ge\pi^nB_{\tilde\rho}^{I,h_0}(z_0) \tag{B},
	\end{equation}
where $\tilde\rho=\prod_{1\le j\le n}e^{-\varphi_j}$.
Theorem \ref{thm:main1-2} (letting $p_j=n(\beta_j+1)$ and $c\equiv1$) tells us that inequality (A)  implies inequality (B) for any $n\ge1$. When $n=1$, inequality (A) holds by definitions. The sufficiency of the characterization in Proposition \ref{p:app3} follows from Theorem \ref{thm:main2-3}.
	
	We prove Proposition \ref{p:app3} in two steps: Firstly, we prove inequality (A) for the case $n$ ($n\ge2$) by using the holding of inequality (B) for the case $n-1$, which shows that inequality (A) holds for any $n$ by induction on $n$; Secondly, we prove the necessity of the characterization in Proposition \ref{p:app3}. 
	
	\
	
	\emph{Step 1:} 
		In this step, we prove inequality (A) for the case $n$ by using the holding of inequality (B) for the case $n-1$. We only prove the case that $K_{\partial M,\rho}^{I,h_0}(z_0)>0$. By Lemma \ref{l:b7}, there is $f\in H_{\rho}^2(M,\partial M)$ such that $(f^*-h_0,z_0)\in I$ and 
	\begin{equation}
		\label{eq:0817a}
		\frac{1}{K_{\partial M,\rho}^{I,h_0}(z_0)}=\|f\|_{\partial M,\rho}^2.
	\end{equation}
Following from Lemma \ref{l:prod-d1xm1} and Lemma \ref{l:a6}, we have 
\begin{equation}
	\label{eq:0817b}
	f=\sum_{l\in \mathbb{Z}_{\ge0}}f_lg_l\text{ (convergence under the norm $\|\cdot\|_{\partial D_1\times M_1,\rho}$)},
\end{equation} 
where $g_l\in A^2(M_1,\hat\rho_1)$ for any $l\in \mathbb{Z}_{\ge0}$, $M_1=\prod_{2\le j\le n}D_j$ and  $\hat\rho_1=\prod_{2\le j\le n}e^{-\varphi_j}$,  and  $\{f_l\}_{l\in\mathbb{Z}_{\ge0}}$ is a complete orthonormal basis of $H^2_{\lambda_1}(D_1,\partial D_1)$ such that $ord_{z_1}(f_l^*)=l$ for any $l\ge0$.
For any $l\in\mathbb{Z}_{\ge0}$, it follows from inequality (B) for the case $n-1$  and Lemma \ref{l:a K<+infty} that there is $\tilde g_l\in H^2_{\hat\lambda_1}(M_1,S_1)$ such that $(\tilde g_l-g_l,\hat z_1)\in I'$
and
\begin{equation}
	\label{eq:0817c}
	\|g_l\|^2_{M_1,\hat\rho_1}\ge\frac{1}{B_{M_1,\hat\rho_1}^{I',g_l}(\hat z_1)}\ge\frac{\pi^{n-1}}{\left(\prod_{2\le j\le n}(\beta_j+1)\right)K_{S_1,\hat\lambda_1}^{I',g_l}(\hat z_1)}=\frac{\pi^{n-1}\|\tilde g_l\|^2_{S_1,\hat\lambda_1}}{\prod_{2\le j\le n}(\beta_j+1)},
\end{equation}
where $S_1=\prod_{2\le j\le n}\partial D_j$ and $\hat\lambda_1=\prod_{2\le j\le n}(\frac{\partial G_{D_j}(w_j,z_j)}{\partial v_{w_j}})^{-1}e^{-\varphi_j}$ on $S_1$.
Note that 
$$f^*=\sum_{l\in \mathbb{Z}_{\ge0}}f^*_lg_l\text{ (convergence on any compact subset of $M$)},$$
$(f^*-h_0,z_0)\in I$, $(\sum_{l\ge\beta_1+1}f^*_lg_l,z_0)\in I$
and $(f_l^*\tilde g_l^*-f_l^*g_l,z_0)\in I$ for any $l\ge0$, then we have $(\sum_{0\le l\le \beta_1}f_l^*\tilde g_l^*-h_0,z_0)\in I$.
As $\ll f_l,f_{l'}\gg_{\partial D_1,\lambda_1}=0$ for any $l\not=l'$, following from equality \eqref{eq:0817b} and inequality \eqref{eq:0817c}, we have 
\begin{equation}
\label{eq:0817f}
	\begin{split}
	\|f\|_{\partial D_1\times M_1,p_1\rho}^2&\ge\sum_{l=0}^{\beta_1}\|f_l\|^2_{\partial D_1,\lambda_1}\times\|g_l\|^2_{M_1,\hat\rho_1}\\
	&\ge\sum_{l=0}^{\beta_1}\|f_l\|^2_{\partial D_1,\lambda_1}\times\frac{\pi^{n-1}\|\tilde g_l\|^2_{S_1,\hat\lambda_1}}{\prod_{2\le j\le n}(\beta_j+1)}\\
	&=\frac{\pi^{n-1}}{\prod_{2\le j\le n}(\beta_j+1)}\|\sum_{l=0}^{\beta_1}f_l\tilde g_l\|^2_{S,\lambda}\\
	&\ge\frac{\pi^{n-1}}{\left(\prod_{2\le j\le n}(\beta_j+1)\right)K_{S,\lambda}^{I,h_0}(z_0)}.
	\end{split}
\end{equation}
Similarly, we have 
\begin{equation}
	\label{eq:0817d}
	\begin{split}
		\|f\|_{\partial D_j\times M_j,p_j\rho}^2\ge\frac{\pi^{n-1}}{\left(\prod_{1\le l\le n,\,l\not=j}(\beta_l+1)\right)K_{S,\lambda}^{I,h_0}(z_0)}.
	\end{split}
\end{equation}
for any $1\le j\le n$, where $M_j=\prod_{1\le l\le n,\,l\not=j}D_l$.
It follows from equality \eqref{eq:0817a} and inequality \eqref{eq:0817d} that 
\begin{equation}
	\label{eq:0817e}\begin{split}
			\frac{1}{K_{\partial M,\rho}^{I,h_0}(z_0)}&=\|f\|_{\partial M,\rho}^2\\
			&=\sum_{1\le j\le n}\frac{1}{p_j}\|f\|_{\partial D_j\times M_j,p_j\rho}^2\\
			&\ge\sum_{1\le j\le n}\frac{1}{p_j}\frac{\pi^{n-1}}{\left(\prod_{1\le l\le n,\,l\not=j}(\beta_l+1)\right)K_{S,\lambda}^{I,h_0}(z_0)}\\
			&=	\left(\sum_{1\le j\le n}\frac{\beta_j+1}{p_j}\right)\frac{\pi^{n-1}}{\left(\prod_{1\le j\le n}(\beta_j+1)\right)K_{S,\lambda}^{I,h_0}(z_0)},
	\end{split}
\end{equation}
hence inequality (A) holds for the case $n$.

\

\emph{Step 2:}
In this step, we prove the necessity of the characterization in Proposition \ref{p:app3} by using Theorem \ref{thm:main1-2} and Theorem \ref{thm:main2-3}. 

Assume that 
$$\left(\prod_{1\le j\le n}(\beta_j+1)\right)K_{S,\lambda}^{I,h_0}(z_0)= \left(\sum_{1\le j\le n}\frac{\beta_j+1}{p_j}\right)\pi^{n-1}K_{\partial M,\rho}^{I,h_0}(z_0).$$
Following the notations in Step 1, 
by inequalities \eqref{eq:0817c}, \eqref{eq:0817f}, \eqref{eq:0817d} and \eqref{eq:0817e},  
\begin{equation}
	\label{eq:220817a}\frac{1}{B_{M_1,\hat\rho_1}^{I',g_{l}}(\hat z_1)}=\frac{\pi^{n-1}}{\left(\prod_{2\le j\le n}(\beta_j+1)\right)K_{S_1,\hat\lambda_1}^{I',g_{l}}(\hat z_1)}
\end{equation}
holds for any $0\le l\le \beta_1$ satisfying that $f_l\not\equiv0$.  Note that $(f,z_0)\not\in I$, then there exists $l\in\{0,\ldots,\beta_1\}$ such that $f_{l}\not\equiv0$ and $(g_{l},\hat z_1)\not\in I'$. 
Following from inequality (A) and equality \eqref{eq:220817a}, we have 
\begin{equation}
	\label{eq:220817b}\frac{1}{B_{M_1,\hat\rho_1}^{I',g_{l}}(\hat z_1)}=\frac{\pi^{n-1}}{\left(\prod_{2\le j\le n}(\beta_j+1)\right)K_{S_1,\hat\lambda_1}^{I',g_{l}}(\hat z_1)}\le \frac{\pi}{K_{\partial M_1,\rho^*_{1}}^{I',g_l}(\hat z_1)},
\end{equation}
where $\rho^*_{1}$ is a positive Lebesgue measurable function on $\partial M_1$ such that $\rho^*_{1}=\frac{1}{(n-1)(\beta_{j}+1)}(\frac{\partial G_{D_{j}}(w_{j},z_{j})}{\partial v_{w_{j}}})^{-1}\prod_{2\le l\le n}e^{-\varphi_l}$ on $\partial D_{j}\times\prod_{l\not=1,j}D_l$ for any $j\not=1$. Denote that $\tilde\psi:=\max_{2\le j\le n}\{2(n-1)(\beta_j+1)\log|w_j-z_j|\}$ on $M_1$. Then it is clear that $\mathcal{I}(\tilde\psi)_{\hat z_1}\subset I'$ and $\{\tilde \alpha=(\tilde\alpha_2,\ldots,\tilde\alpha_n)\in\mathbb{Z}_{\ge0}^{n-1}:\sum_{2\le j\le n}\frac{\tilde\alpha_j+1}{(n-1)(\beta_j+1)}=1\}\cap\{\tilde \alpha=(\tilde\alpha_2,\ldots,\tilde\alpha_n)\in\mathbb{Z}_{\ge0}^{n-1}:\tilde\alpha_j\le\beta_j,\,\forall j\}=\{(\beta_2,\ldots,\beta_n)\}$. 
Note that $f=\sum_{l\in\mathbb{Z}_{\ge0}^n}f_lg_l,$
$ord_{z_1}f^*_l=l$ for any $l\ge0$ and 
$$(f^*-\sum_{\alpha\in E_0}d_{\alpha}(w-z_0)^{\alpha},z_0)\in I,$$ where $E_0=\{\alpha\in\mathbb{Z}_{\ge0}^n:\alpha_j\le\beta_j$ for any $1\le j\le n\}$.
Using Theorem  \ref{thm:main1-2}, we obtain that 
$$\frac{1}{B_{M_1,\hat\rho_1}^{I',g_{l}}(\hat z_1)}=\frac{\pi}{K_{\partial M_1,\rho^*_{1}}^{I',g_{l}}(\hat z_1)}$$
and $(g_{l}-c_l\prod_{2\le j\le n}(w_j-z_j)^{\beta_j},\hat z_1)\in I'$ for any $l\in\{0,\ldots,\beta_1\}$ such that $f_{l}\not\equiv0$ and $(g_{l},\hat z_1)\not\in I'$, where $c_l$ is a constant,  which shows that 
$$d_{\alpha}=0$$
for any $\alpha\in E_0$ satisfying that there exists $j\in\{2,\ldots,n\}$ such that $\alpha_j<\beta_j$. By similar discussion, we have 
$d_{\alpha}=0$
for any $\alpha\in E_{0}\backslash\{\beta\}$, i.e., $$h_0=d_{\beta}(w-z_0)^{\beta}$$
near $z_0.$ It follows from $\left(\prod_{1\le j\le n}(\beta_j+1)\right)K_{S,\lambda}^{I,h_0}(z_0)= \left(\sum_{1\le j\le n}\frac{\beta_j+1}{p_j}\right)\pi^{n-1}K_{\partial M,\rho}^{I,h_0}(z_0)$  and Theorem \ref{thm:main2-3} that $\varphi_j$ is harmonic on $D_j$ and $\chi_{j,z_j}^{\beta_j+1}=\chi_{j,\frac{\varphi_j}{2}}$ for any $1\le j\le n$.

Thus, Proposition \ref{p:app3} holds.
\end{proof}

Let $\tilde\rho$ be as in Corollary \ref{c:app2}. Assume that $\mathcal{I}(\psi)_{z_0}\subset I$, i.e., $\sum_{1\le j\le n}\frac{\beta_j+1}{p_j}\le1$ by Lemma \ref{l:psi}. Proposition \ref{p:app3} and Theorem \ref{thm:main1-2} imply the following corollary.

\begin{Corollary}
	\label{c:app4}
	Assume that $B_{\tilde\rho}^{I,h_0}(z_0)>0$. Inequality 
		\begin{equation}\nonumber
		\left(\prod_{1\le j\le n}(\beta_j+1)\right)K_{S,\lambda}^{I,h_0}(z_0)\ge\left(\int_0^{+\infty}c(t)e^{-t}dt\right)\left(\sum_{1\le j\le n}\frac{\beta_j+1}{p_j}\right)\pi^nB_{\tilde\rho}^{I,h_0}(z_0) 
	\end{equation}
holds, and equality holds if and only if the following statements hold:

$(1)$ $\sum_{1\le j\le n}\frac{\beta_j+1}{p_j}=1$ and  $h_0=d_{\beta}(w-z_0)^{\beta}$ near $z_0$;

$(2)$ $\varphi_j$ is harmonic on $D_j$ for any $1\le j\le n$;

$(3)$  $\chi_{j,z_j}^{\beta_j+1}=\chi_{j,-\frac{\varphi_j}{2}}$  for any $1\le j\le n$, where $\chi_{j,z_j}$ and $\chi_{j,-u_j}$ are the characters associated to functions $G_{\Omega_j}(\cdot,z_j)$ and $-u_j$ respectively. 
\end{Corollary}

\

\vspace{.1in} {\em Acknowledgements}. The authors would like to thank Prof. Akira Yamada, Dr. Shijie Bao and Dr. Zhitong Mi for checking the manuscript and  pointing out some typos. The first named author was supported by National Key R\&D Program of China 2021YFA1003100, NSFC-11825101, NSFC-11522101 and NSFC-11431013.

\bibliographystyle{references}
\bibliography{xbib}

\end{document}